\documentclass[a4paper,reqno, 11pt]{amsart}  
\usepackage[DIV=12, oneside]{typearea}
\usepackage[utf8]{inputenc}
\usepackage[T1]{fontenc}
\usepackage[english]{babel}
\usepackage[centertags]{amsmath}
\usepackage{amstext,amssymb,amsopn,amsthm}
\usepackage{nicefrac, esint}
\usepackage{mathrsfs}
\usepackage{dsfont}
\usepackage{bbm}
\usepackage{thmtools}
\usepackage{graphicx}
\usepackage{float}
\usepackage[titletoc,title]{appendix}

\usepackage[backgroundcolor=white, bordercolor=blue,
linecolor=blue]{todonotes}
\usepackage[colorlinks=true, linkcolor=black, citecolor=black]{hyperref}
\usepackage{enumitem}
\setlist[enumerate]{itemsep=0mm}

\addto\extrasenglish{}
\addto\extrasenglish{}
\addto\extrasenglish{}
\declaretheorem[name=Theorem, numberwithin=section]{theorem}
\newtheorem{corollary}[theorem]{Corollary}
\newtheorem{lemma}[theorem]{Lemma}
\newtheorem{proposition}[theorem]{Proposition}
\newtheorem{definition}[theorem]{Definition}

\newtheorem*{example*}{Example}
\newtheorem*{robremark*}{Robustness remark}

\declaretheoremstyle[bodyfont=\normalfont]{remark-style}
\declaretheorem[name=Remark, numberwithin=section, style=remark-style, sibling=theorem]{remark}
\declaretheorem[name=Example, numberwithin=section, style=remark-style, sibling=theorem]{example}

\numberwithin{equation}{section}
\newtheorem{assumption}{Assumption}

\newcommand{\N}{\mathds{N}}
\newcommand{\R}{\mathds{R}}

\newcommand{\Z}{\mathds{Z}}

\newcommand{\ma}{\mu_{\text{axes}}}

\newcommand{\am}{\alpha_{\max}}
\def\hmath$#1${\texorpdfstring{{\rmfamily\textit{#1}}}{#1}}

\newcommand{\cE}{\mathcal{E}}
\newcommand{\cK}{\mathcal{K}}

\newcommand{\cU}{\mathcal{U}}

\newcommand{\cF}{\mathcal{F}}

\newcommand{\eps}{\varepsilon}

\newcommand{\loc}{\mathrm{loc}}

\newcommand{\amax}{{\alpha_{\max}}}

\newcommand{\BIGOP}[1]
{
\mathop{\mathchoice%
{\raise-0.22em\hbox{\huge $#1$}}%
{\raise-0.05em\hbox{\Large $#1$}}{\hbox{\large $#1$}}{#1}}}
\newcommand{\bigtimes}{\BIGOP{\times}}
\def\Xint#1{\mathchoice
   {\XXint\displaystyle\textstyle{#1}}%
   {\XXint\textstyle\scriptstyle{#1}}%
   {\XXint\scriptstyle\scriptscriptstyle{#1}}%
   {\XXint\scriptscriptstyle\scriptscriptstyle{#1}}%
   \!\int}
\def\XXint#1#2#3{{\setbox0=\hbox{$#1{#2#3}{\int}$}
     \vcenter{\hbox{$#2#3$}}\kern-.5\wd0}}

\def\dashint{\Xint-}
\newcommand{\BIGboxplus}{\mathop{\mathchoice%
{\raise-0.35em\hbox{\huge $\boxplus$}}%
{\raise-0.15em\hbox{\Large $\boxplus$}}{\hbox{\large $\boxplus$}}{\boxplus}}}

\DeclareMathOperator{\dist}{dist}

\DeclareMathOperator{\supp}{supp}

\DeclareMathOperator{\dvg}{div}

\DeclareMathOperator*{\osc}{osc}

\newcommand{\U}{\widetilde{u}}

\renewcommand{\d}{\textnormal{d}}

\begin{document}
\allowdisplaybreaks
 \title{Parabolic problems for direction-dependent local-nonlocal operators}

\author{Jamil Chaker}
\author{Moritz Kassmann}
\author{Marvin Weidner}

 \address{Fakult\"{a}t f\"{u}r Mathematik\\Universit\"{a}t Bielefeld\\Postfach 100131\\D-33501 Bielefeld}
\email{jchaker@math.uni-bielefeld.de}

\address{Fakult\"{a}t f\"{u}r Mathematik\\Universit\"{a}t Bielefeld\\Postfach 100131\\D-33501 Bielefeld}
\email{moritz.kassmann@uni-bielefeld.de}
\urladdr{www.math.uni-bielefeld.de/$\sim$kassmann}

\address{Universitat de Barcelona, Departament de Matem\`atiques i Inform\`atica, Gran Via de les Corts Catalanes 585, 08007 Barcelona, Spain}
\email{mweidner@ub.edu}
\urladdr{https://sites.google.com/view/marvinweidner/}


\thanks{Jamil Chaker is supported by the DFG through Forschungsstipendium Project 410407063. Moritz Kassmann is supported by the DFG through CRC 1283. Marvin Weidner is supported by the European Research Council (ERC) under the Grant Agreement No 801867, and by the AEI project PID2021-125021NA-I00 (Spain).}


\allowdisplaybreaks

\begin{abstract}
We study parabolic equations governed by integro-differential operators with nonlocal components in some directions and local components in the remaining directions. The setting contains the purely nonlocal, as well as the purely local case. Our approach is based on an energy method allowing for jumping measures that are singular or supported on cusps. In addition, the jumping measure may depend on the direction. The emphasis of our study is on the weak Harnack inequality and H\"older regularity estimates for solutions of such equations. The main regularity estimates are robust in the sense that the constants can be chosen independently of the order of differentiability of the operators.
\end{abstract}

\maketitle

\section{Introduction}  
The aim of this work is to study regularity properties of weak solutions to parabolic equations of the form $\partial_t u-Lu=f$ in $I \times \Omega\subset \R^{d+1}$ 
governed by a linear local-nonlocal operator of the following type 
\begin{align*} 
	L u (t,x) = \text{ p.v.} \int_{\R^d} (u(t,y)-u(t,x))\mu(t,x,\d y) + \dvg \left(A(t,x)(\partial_k u(t,x))_{k = d_1+1}^{d}\right) \,.
\end{align*}
This operator is determined by a family of jumping measures $\mu(t,x,\cdot)$ and matrices $A(t,x) \in \R^{d_2 \times d_2}$ for $d=d_1 + d_2$, $t\in\R$, $x\in\R^d$. The operator $L$ is a nonlocal operator in the first $d_1$ coordinates and a second-order differential operator in the last $d_2$ coordinates. Note that our approach includes the purely nonlocal case $d=d_1$ and the purely local case $d=d_2$. In the purely nonlocal case, our results include the regularity results proved in \cite{DyKa15, ChKa20}, and in \cite{KaWe22a} when assuming symmetry of $\mu$.

Let us provide three characteristic examples, which are covered in this work for the first time. The first example is given by the following simple local-nonlocal translation-invariant operator 
\begin{align}\label{eq:frac_aniso_laplace}
-L = (-\partial_{11})^{\alpha_1/2} + \ldots + (-\partial_{d_1 d_1})^{\alpha_{d_1}/2} + (-\partial_{(d_1+1)(d_1+1)}) + \ldots (-\partial_{dd})\,,
\end{align}
where $\alpha_1,\dots,\alpha_{d_1}\in(0,2)$ and $d_1,d_2 \ge 0$ with $d_1 + d_2 = d$, see \autoref{sec:lnl}. A second example is the purely nonlocal ($A \equiv 0$) operator 
\begin{align}\label{eq:cusp-frac-laplace}
	-Lu (x) = \text{ p.v.} \int_{\R^d} (u(x)-u(y)) |x-y|^{-d-\gamma} \mathds{1}_{\Gamma}(x-y)\d y \,,
\end{align} 
where $\gamma$ can be any arbitrarily large positive number if the set $\Gamma \subset \R^d$ is chosen appropriately, see \autoref{sec:cusps}   for further details.  The main reason that allows us to study operators including \eqref{eq:frac_aniso_laplace} and \eqref{eq:cusp-frac-laplace} is that our approach is tailor-made for direction-dependent singular jumping measures. To this end we introduce $(\ma(x,\cdot))_{x \in \R^d}$, which serves as a family of reference measures and plays an important role within this article
\begin{align}
	\label{def:muaxes}
	\ma(x,\d y)=\sum_{k=1}^{d_1} \Big( (2-\alpha_k) |x_k-y_k|^{-1-\alpha_k}\,  \d y_k\prod_{i\neq k}\delta_{\{x_i\}}(\d y_i) \Big)
\end{align}
with $\alpha_1,\dots,\alpha_{d_1}\in(0,2)$. A third purely nonlocal operator that is characteristic for this work is then defined by $\mu(t,x,\d y)=a(t,x,y)\ma(x,\d y)$, where $a$ is a measurable, positive, bounded function, and $d_1=d$, $A \equiv 0$. For more details on this example see \autoref{sec:mu-axes}.

Note that the general form of $L$ reduces to \eqref{eq:frac_aniso_laplace} up to constants if we set $\mu(t,x,\d y) = \ma(x,\d y)$ and $A = (\delta_{jk})_{j,k = d_1+1}^d$. Then, the directional orders of differentiation of the operator are determined by the exponents $\alpha_k$. 
It is instructive to understand the jumps resp. the L\'{e}vy process that is driven by the jumping measure $\ma(0,\d y)$. First note, that $\ma(0,\d y)$ is a 
L\'{e}vy measure supported on the union of the coordinate axes. The corresponding L\'{e}vy process $Z_t = (Z_t^1,\dots,Z_t^{d_1})$ is a $d_1$-dimensional process consisting of independent one-dimensional $\alpha_k$-stable L\'{e}vy processes $Z_t^k$. Due to the independence, the process $Z_t$ can only jump into the direction of the coordinate axes, which is done like a one-dimensional $\alpha_k$-stable L\'{e}vy process in the $k$-th coordinate direction. The operator $L$ defined in \eqref{eq:frac_aniso_laplace} corresponds to the generator of the L\'{e}vy process $(Z_t, B_t)$, where $Z_t$ is as before and $B_t$ is an independent $d_2$-dimensional Brownian motion. The process jumps at any given time with probability one in at most one of the first $d_1$-directions. This is done like a one-dimensional $\alpha_k$-stable L\'{e}vy process in the $k$-th coordinate direction. Meanwhile, the process moves continuously inside the $d_2$-dimensional subspace spanned by the remaining $x_{d_1+1},\dots,x_d$-directions like a $d_2$-dimensional Brownian motion.

\subsubsection*{An anisotropic metric} In order to deal with anisotropies, we consider for given $d_1, d_2 \in \{0,1,\dots,d\}$ with $d = d_1 + d_2$, $\alpha_1,\dots,\alpha_{d_1} \in(0,2)$, and $\alpha_{d_1 + 1},\dots,\alpha_d = 2$ corresponding rectangles.
Throughout the article, we set $\amax := \max\{\alpha_k |\, 1 \leq k \leq d \} \vee 1$.
\begin{definition}\label{def:M_r}
Let $\alpha_1,\dots,\alpha_{d} \in (0,2]$. For $r>0$ and $x\in\R^d$ we define 
\begin{align*}
M_r(x) =\bigtimes_{k=1}^d 
\left(x_k-r^{\frac{\amax}{\alpha_k}},x_k+r^{\frac{\amax}{ \alpha_k}}\right) 
\quad \text{ and } M_r = M_r(0) \,.
\end{align*}
\end{definition}
It is worth emphasizing that $\am$ can be replaced by any real number $a\geq \am$.
It has to be ensured that there is an underlying metric in the Euclidean space $\R^d$ such that $M_r(x)$ are balls with radius $r>0$ centered at $x\in\R^d$. 
Replacing $\am$ by any $a\geq \am$ will still satisfy this property.  For each $k \in \lbrace 1,\dots,d\rbrace$, we define $E_r^k(x) = \lbrace y \in \R^d : \vert x_k - y_k \vert < r^{\am/{\alpha_k}}\rbrace$. Note 
\begin{align}
\label{def:E_r}
M_r(x) = \bigcap_{k = 1}^d E_r^k(x).
\end{align}

\subsubsection*{Assumptions and main results} The purpose of this paper is to study regularity results for weak solutions to $\partial_t u-Lu=f$. 
Since we will define weak solutions with the help of bilinear forms, let us introduce those objects in the following. 
The actual definition of weak solutions will be given in \autoref{sec:Prel}.\\
Let $(\mu(t,x,\cdot))_{(t,x) \in I \times \R^d}$ be a family of measures and $A = \left(A_{j,k}\right)_{j,k = d_1+1}^d : I \times \R^d \rightarrow \R^{d_2 \times d_2}$. For $u,v\in L^2_{\loc}(\R^d)$ and $\Omega\subset\R^d$ open and bounded, we define
\begin{align*}
\mathcal{E}^{\mu(t)}_{\Omega}(u,v) &= \int_{\Omega}\int_{\Omega} 
(u(y)-u(x))(v(y)-v(x))\,\mu(t,x,\d y)\, \d x,\\
\mathcal{E}^{A(t)}_{\Omega}(u,v) &= \int_{\Omega} \sum_{j,k=d_1+1}^d A_{j,k}(t,x) \partial_j u(x) \partial_k v(x) \d x,\\
\mathcal{E}_{\Omega}^t(u,v) &= \mathcal{E}^{\mu(t)}_{\Omega}(u,v) + \mathcal{E}^{A(t)}_{\R^d}(u,v),
\end{align*}
and $\mathcal{E}^{\mu(t)}(u,v) = \mathcal{E}^{\mu(t)}_{\R^d}(u,v)$, $\mathcal{E}^{A(t)}(u,v)= \mathcal{E}^{A(t)}_{\R^d}(u,v)$, $\mathcal{E}^t(u,v) = \mathcal{E}^t_{\R^d}(u,v)$ whenever the 
quantities are finite.

Let us formulate and explain the main assumptions on $(\mu(t,x,\cdot))_{(t,x) \in I \times \R^d}$ and $A$. 
We assume symmetry in the following way: 

\begin{assumption}\label{assumption:symmetry}
For all measurable sets $F,G \in \mathcal{B}(\R^d)$ and every $t \in I$, $x \in \R^d$:
\begin{align}
\label{assmu0}
\int_F \int_G \mu(t,z,\d y) \d z = \int_G \int_F \mu(t,z,\d y) \d z, \qquad A_{j,k}(t,x) = A_{k,j}(t,x).
\end{align}
\end{assumption}

Let $\alpha_1,\dots,\alpha_{d_1}\in[\alpha_0,2)$ be given for some $\alpha_0\in(0,2]$ and let $\beta > 0$.
The following assumption deals with the tail behavior of $(\mu(t,x,\cdot))_{(t,x)\in I \times \R^d}$ and the boundedness of $A$.

\begin{assumption}\label{assumption:tail}
There exist $\Lambda \ge 1$ and  $\theta > 0$ such that for every $t \in I$, $x_0 \in \R^d$:
\begin{align}
\label{assmu1new}
\mu(t,x_0, \R^d \setminus E_{\rho}^k(x_0)) &\le 
\begin{cases}
\Lambda (2-\alpha_k)\rho^{-\am}, &\rho \in (0,2),\\
\Lambda (2-\alpha_k)\rho^{-\theta}, &\rho > 2.
\end{cases}
\quad \forall k \in \lbrace 1, \dots , d \rbrace,\\
\label{eq:A-upper}
\sum_{j,k=d_1+1}^d A_{j,k}(t,x) \xi_i \xi_j &\le \Lambda |\xi|^2 \quad \forall x \in \R^d, ~~ \xi = (\xi_{d_1+1},\dots,\xi_d) \in \R^{d_2}.
\end{align}
\end{assumption}

The tail estimate \eqref{assmu1new} for $\rho \in (0,2)$ will later be used in order to work with Lipschitz-continuous cut-off functions, see \autoref{thmcutoffest}. $\theta$ can be any positive number with $\theta = \am$ being the most natural choice.
The next assumption explains, in which sense the family $(\mu(t,x,\cdot))_{(t,x) \in I \times \R^d}$ is comparable with $(\ma(x, \cdot))_{x \in \R^d}$.
Given $\alpha_1,\dots, \alpha_{d_1} \in (0,2)$, we define 
\[ \beta := \sum_{k = 1}^{d_1} \frac{1}{\alpha_k} + \sum_{k = d_1 + 1}^d \frac{1}{2}.\]

\begin{assumption}\label{assumption:func-ineq}
There exists $\Lambda \ge 1$ such that for every $t \in I$, $x_0 \in \R^d$, $r \in (0,1]$, $\lambda \in (1,2]$ and $v \in V(M_{\lambda r}(x_0) | \R^d)$:
\begin{align}
\label{eq:poincare-assum}
\|v-[v]_{M_r(x_0)}\|_{L^2(M_r(x_0))}^2 &\leq \Lambda r^{\am}\mathcal{E}^t_{M_{r}(x_0)}(v,v),\\
\label{eq:sobolev-assum}
\Vert v \Vert_{L^{\frac{2\beta}{\beta-1}}(M_{r}(x_0))}^2 \le \Lambda \mathcal{E}^t_{M_{\lambda r}(x_0)}(v,v) &+ \Lambda r^{-\am}\left( \sum_{k=1}^d (\lambda^{\frac{\am}{\alpha_k}} - 1)^{-\alpha_k}\right) \Vert v \Vert_{L^2(M_{\lambda r}(x_0))}^2,\\
\label{eq:fct-space-comp}
\mathcal{E}^{\mu(t)}_{M_r(x_0)}(v,v) &\le \Lambda \mathcal{E}^{\ma}_{M_r(x_0)}(v,v).
\end{align}
\end{assumption}

\autoref{assumption:func-ineq} serves two purposes. On the one hand, it contains the Poincar\'e inequality \eqref{eq:poincare-assum} and the Sobolev inequality \eqref{eq:sobolev-assum}. On the other hand, through \eqref{eq:fct-space-comp} it guarantees the finiteness of the  nonlocal terms used in the definition of (weak) solutions, see \eqref{weaksol}. Note that we assume comparability on small scales only because we prove the weak Harnack inequality and Hölder regularity estimates on scales comparable to the size of $\Omega$.

Using the aforementioned assumptions, we can define the class of admissible pairs $(\mu,A)$ that we use in this paper.

\begin{definition}[Class of admissible pairs $(\mu,A)$]\label{def:admissible-pairs}
Let $\Lambda \ge 1$ and $\alpha_1,\dots,\alpha_{d_1}\in[\alpha_0,2)$ be given for some $\alpha_0\in(0,2)$ and $\alpha_{d_1+1},\dots,\alpha_d = 2$.
We call a family of measures $(\mu(t,x,\cdot))_{(t,x) \in I \times \R^d}$ and a matrix-valued function $(A_{j,k})_{j,k=d_1+1}^d$ admissible, 
if they satisfy \autoref{assumption:symmetry}, \autoref{assumption:tail} and \autoref{assumption:func-ineq}. We denote the class of such pairs $(\mu,A)$ by $\cK(\alpha_0,\Lambda)$.
\end{definition}

Note that the constants in our two main theorems depend on the choice of $\alpha_0$ and $\Lambda$ but not on the values of $\alpha_1, \ldots, \alpha_d$. Thus we suppress the dependence on $\alpha_1, \ldots , \alpha_d$ in the notation of $\cK(\alpha_0,\Lambda)$.

\begin{remark}\label{rem:classes}{\ } \vspace*{-1ex}
	\begin{itemize}
		\item [(i)] A basic example is $(\mu,0) \in \cK(\alpha_0,\Lambda)$ if $\mu(t,x, \d y)$ is symmetric in the sense of \autoref{assumption:symmetry} and the comparability $\mu(t,x, \d y) \asymp |x-y|^{-d-\alpha} \d y$ holds true. In this case, one chooses all exponents $\alpha_1, \ldots, \alpha_d$ to be equal to $\alpha$. The verification of \autoref{assumption:tail} and \autoref{assumption:func-ineq} can be found in \cite{DyKa15}.
		\item [(ii)] An important example is $(\ma,0) \in \cK(\alpha_0,\Lambda)$  where  $\Lambda \ge 1$ needs to be chosen appropriately in dependence of $d$ and $\alpha_0$. In this case $d_1 = d$, see \autoref{sec:mu-axes} for details.
		\item [(iii)] The example \eqref{eq:frac_aniso_laplace} is covered by $(\ma,A) \in \cK(\alpha_0,\Lambda)$, where we choose $A = (\delta_{jk})_{j,k = d_1 + 1}^d$, see \autoref{sec:lnl}.
		\item[(iv)] In \autoref{sec:cusps} we provide examples of the form $(\mu,0) \in \cK(\alpha_0,\Lambda)$ where $\mu(t,x, \d y)$ is comparable to $|x-y|^{-d-\gamma} \d y$ for arbitrarily large numbers $\gamma$. The catch here is that comparability is assumed only on cusp-like subdomains of $\R^d$.
	\end{itemize}
\end{remark}

In this work we study operators defined by admissible pairs $(\mu,A)$ and establish local regularity results such as Hölder regularity estimates. 
Our main auxiliary result is a so-called weak Harnack inequality.
A well-known consequence of the weak Harnack inequality is a result on the decay of oscillation for solutions and therefore Hölder regularity estimates. 

Our two main results are as follows. 

\begin{theorem}[Weak Harnack inequality]
\label{thm:weakHarnack}
There is a constant $c = c(d,\alpha_0,\Lambda) > 0$ such that for every $(\mu,A) \in \cK(\alpha_0,\Lambda)$ and every supersolution $u$ of $\partial_t u - L u = f$ in $Q = (-1,1) \times M_2(0)$, satisfying $u\geq 0$ in $(-1,1) \times \R^d$, the following holds:
\begin{align}
\label{eq:weakHarnack}
\Vert u \Vert_{L^1(U_{\ominus})} \le c \left( \inf_{U_{\oplus}} u + \Vert f \Vert_{L^{\infty}(Q)} \right),
\end{align}
where $U_{\oplus} = (1-\frac{1}{2^{\am}},1) \times M_{1/2}(0)$ and $U_{\ominus} = (-1,-1 +\frac{1}{2^{\am}}) \times M_{1/2}(0)$.
\end{theorem}

\begin{remark}
One obtains the following scaled version of \autoref{thm:weakHarnack} by applying \autoref{thmscaling}:
\begin{align*}
\begin{split}
\Vert u \Vert_{L^1((-r^{\am}, (-1+\frac{1}{2^{\am}})r^{\am})\times M_{r/2}(0))} &\le c\Bigg( \inf_{((1-\frac{1}{2^{\am}})r^{\am},r^{\am}) \times M_{r/2}(0)} u\\
&\qquad + r^{\am} \Vert f \Vert_{L^{\infty}((-r^{\am},r^{\am}) \times M_{2r}(0))}\Bigg).
\end{split}
\end{align*}
\end{remark}

\begin{theorem}[Hölder regularity estimate]
\label{thm:HR}
There is a constant $\gamma = \gamma(d,\alpha_0,\Lambda,\theta) \in (0,1)$ such that for every $(\mu,A) \in \cK(\alpha_0,\Lambda)$ and every solution $u$ to $\partial_t u - L u = 0$ in $Q = I \times \Omega$ and every $Q' \Subset Q$, the following holds:
\begin{align}
\label{eq:HR}
\sup_{(t,x),(s,y) \in Q'} \frac{\vert u(t,x) - u(s,y) \vert}{\left( \vert x-y \vert + \vert t-s \vert^{1/\am} \right)^{\gamma}} \le \frac{\Vert u \Vert_{L^{\infty}(I \times \R^d)}}{\eta^{\gamma}},
\end{align}
where $\eta = \eta(Q,Q',\alpha_0) > 0$ is some constant.
\end{theorem}

\begin{remark}{\ }\vspace*{-1ex} 
\begin{itemize}
\item[(i)] \autoref{thm:weakHarnack} and \autoref{thm:HR} respectively the two estimates \eqref{eq:weakHarnack} and \eqref{eq:HR} are robust in the following sense. The constants $c, \gamma, \eta$ depend on $\alpha_1,\ldots,\alpha_d$ only through the lower bound $\alpha_0$. Choosing $\alpha_k^n$ with $\alpha_k^n \nearrow 2$ as $n \to \infty$ for some (or all) $k \in \{1,\ldots,d \}$, the estimates hold true with constants independent of $n$.
\item[(ii)] We choose to present the regularity estimate \eqref{eq:HR} in a way that does not reflect the possible anisotropy of the problem. Since the decay-of-oscillation result \autoref{thmoscdecay} captures the anisotropy, one could easily derive an anisotropic version of \eqref{eq:HR}. In terms of further applications like compactness results, such a version does not seem to be very important to us. 
\item[(iii)] We point out that for \autoref{thm:weakHarnack} it suffices if \autoref{assumption:tail} is satisfied only for $\rho \in (0,2)$. However, the proof of \autoref{thm:HR} requires some additional knowledge about the decay of the tails of $(\mu(t,x,\cdot))_{(t,x) \in I \times \R^d}$ at infinity.
% Let us also mention \cite{FeKa13}, where it is demonstrated how H\"older estimates can be deduced from a weak Harnack inequality using a moment condition on $(\mu(t,x,\cdot))_{(t,x) \in I \times \R^d}$ instead.
\end{itemize}

\end{remark}
		
\subsubsection*{Related literature}
In the sequel we comment on related results in the literature. Let us first mention results related to systems of jump processes driven by stable processes and to nonlocal operators with singular jump intensities. Several works have studied systems of Markov jump processes given by $\d Y^i_t = A_{ik} (Y_{t-}) \d Z^k_t$, where $Z^k_t$ $(k=1,\ldots, d)$ are one-dimensional independent symmetric $\alpha_k$-stable components. In the case 
$\alpha= \alpha_k$ for all $k$, \cite{BaCh06} proves unique weak solvability with the help of the martingale problem. Strong solvability for such systems including irregular drifts are considered in \cite{CZZ17}. Under additional assumptions on the coefficients, H\"{o}lder regularity of corresponding harmonic functions is established in \cite{BaCh10}. In \cite{DeFo13} the transition probability is shown to have a density. \cite{KRS18} establishes the strong Feller property for the corresponding stochastic jump process. In \cite{KuRy18} two-sided heat kernel estimates are obtained if the matrix $(A_{ik})$ is diagonal and the diagonal elements are bounded and H\"{o}lder continuous. \\
Systems of jump processes have also been investigated when the indices of stability $\alpha_k$ are different. Regularity of corresponding harmonic functions is extended in \cite{Cha16}, making use of the approach in \cite{BaCh10}. For such systems, the results of \cite{DeFo13} are extended in \cite{FPR18}. Uniqueness and existence of weak solutions is proved in \cite{Cha19} under the additional assumption that the matrix $(A_{ik})$ is diagonal. In the case of independent L\'{e}vy processes, e.g. with different indices of stability, regularity properties of the transition semigroup are studied in \cite{KKR22a}, \cite{KKR22b}, and \cite{KPP23}.

The aforementioned works address Markov processes resp. integro-differential operators that, in general, are not symmetric with respect to the Lebesgue measure. For symmetric Markov processes resp. nonlocal operators in divergence form, the approach via Dirichlet forms resp. variational calculus has proven to be successful. One important feature of these approaches is that one does not need to impose further regularity assumptions on the coefficients other than measurability and boundedness. In the case where the jumping measure is absolutely continuous and has isotropic bounds, H\"{o}lder regularity of weak solutions to nonlocal equations corresponding to symmetric Dirichlet forms is proved in \cite{BaLe02}, \cite{ChKu03} resp. in \cite{Kas09}, \cite{DKP14}, \cite{DKP16}, \cite{FeKa13}, \cite{DyKa15}. See \cite{KaWe22a, KaWe22b} for extensions to nonsymmetric jumping kernels. Note that the current work includes the regularity results from \cite{DyKa15, ChKa20} resp.  \cite{KaWe22a} when assuming symmetry of $\mu$. As explained in \autoref{rem:classes}, in order to recover those results one chooses all exponents $\alpha_k$ to be equal. The singular case with identical exponents $\alpha_k$ is covered in \cite{KaSc14, DyKa15}. The case of different exponents is covered in \cite{ChKa20}. If all exponents $\alpha_k$ are equal, sharp lower pointwise heat kernel estimates are proved in \cite{Xu13}. Sharp upper bounds are proved in \cite{KKK19} without any further regularity assumptions on the coefficients. Moreover, we refer to \cite{ChKi22,CKW23,DuBo23} for results on nonlinear anisotropic nonlocal operators.

The present work extends \cite{ChKa20} in three ways. First, our main results \autoref{thm:weakHarnack} and \autoref{thm:HR} establish local regularity properties for \emph{parabolic} nonlocal problems. Second, our results are \emph{robust} in the sense that all important constants remain bounded if some of the indices $\alpha_k$ approach the limit value $2$, see \autoref{sec:mu-axes}. As a result of this feature we are able to include \emph{mixed local-nonlocal} operators that, with respect to some coordinates, are of second order, see \autoref{thm:genmainresult-ln}.

Note that the expression \emph{mixed local-nonlocal} is also used in  context with related operators of different nature such as $-\Delta + (-\Delta)^{\alpha/2}$. Different from the operator \eqref{eq:frac_aniso_laplace}, the operator $-\Delta + (-\Delta)^{\alpha/2}$ is the sum of two uniformly elliptic operators, which both have regularizing properties, whereas the nonlocal and local parts in \eqref{eq:frac_aniso_laplace} are both degenerate. Only through the summation of the two parts, \eqref{eq:frac_aniso_laplace} becomes a uniformly elliptic operator and solutions have certain regularity properties. Another important difference is that harmonic functions with respect to $-\Delta + (-\Delta)^{\alpha/2}$ clearly satisfy a Harnack inequality, but, due to the severe anisotropy of \eqref{eq:frac_aniso_laplace}, solutions lack this property in case of \eqref{eq:frac_aniso_laplace}. Thus, the full Harnack inequality is not relevant for our work.

While operators of the form \eqref{eq:frac_aniso_laplace} have not yet been studied from the viewpoint of regularity (\cite{FaVa17}, \cite{FaVa19} being two exceptions), there has been a lot of research on mixed local-nonlocal operators in the sense described above. An early result on the weak Harnack inequality and Hölder regularity is \cite{Kas03}. Let us list more recent works including time-dependent and nonlinear variations of such problems. For instance, a De Giorgi-Nash-Moser-type theory has been developed in \cite{Foo09a,Foo09b,GaKi22,GaKi23,FSZ22,Nak22,Nak23,ShZh23,BLS23,
GKK23}. Moreover, we refer to \cite{DeMi22,BySo23,APT23,BiPr23,Das23,AnCo23} for the study of gradient regularity and to \cite{BDVV21,BDVV22,BMS23,DFZ23,Gof22,BMP23} for other related results.

\subsubsection*{Outline}
This article consists of seven sections and one appendix, and is organized as follows. In \autoref{sec:Prel}, we give a number of important definitions and introduce the concept of weak solutions. 
\autoref{sec:logineq} provides a technical inequality for the logarithm of supersolutions to the parabolic equation $\partial_t u-Lu = f$. In \autoref{sec:Moser} we present the 
anisotropic Moser iteration scheme for negative and small positive exponents. Finally, \autoref{sectionHI} contains the proof of the weak Harnack inequality and \autoref{sectionHR} the proof of the Hölder regularity estimate. In \autoref{sec:examples} we present examples of appropriate families of measures. Finally, in the appendix, \autoref{sec:Sal}, we prove that a weak Poincar\'e inequality implies a strong Poincar\'e inequality in a very general setting. We consider this result to be of independent interest.

%%%%%%%%%%%%%%%%%%%%%%%%%%%%%%%%%%%%%%
\section{Preliminaries}\label{sec:Prel}
%%%%%%%%%%%%%%%%%%%%%%%%%%%%%%%%%%%%%%

This section contains several important definitions and provides some technical results.
Furthermore, we introduce the concept of weak solutions, set up notation and terminology, prove an auxiliary lemma about cut-off functions, and deduce a weighted Poincar\'e inequality.

From now on let $d_1,d_2 \in \{0,1,\dots,d\}$ with $d_1 + d_2 = d$ and $\alpha_1,\dots,\alpha_{d_1} \in [\alpha_0,2)$ be given for some $\alpha_0 \in (0,2)$. Moreover, we set $\alpha_k = 2$ for any $k \in \{d_1+1,\dots,d\}$. Let $\am = \max\{\alpha_k |\, 1 \leq k \leq d \}$. 
Throughout this article $\Omega$ always denotes a bounded domain in $\R^d$ and $I$ an open, bounded interval in $\R$. 

Let us first discuss \autoref{assumption:tail}. Note that \eqref{def:E_r} implies that for every $\mu$ satisfying \autoref{assumption:tail}, it holds for every $t \in I$ $x_0 \in \R^d$ and $\rho \in (0,2)$:
\begin{align}
\label{assmu1}
\mu(t,x_0, \R^d \setminus M_{\rho}(x_0)) \le \sum_{k = 1}^{d} \mu(t,x_0, \R^d \setminus E_{\rho}^k(x_0)) \le \Lambda \sum_{k = 1}^{d} (2-\alpha_k) \rho^{-\am} \le c \rho^{-\am},
\end{align}
where $c = 2d_1\Lambda$ does not depend on $\alpha_k$. 
This line will be of importance in many upcoming arguments and plays a similar role as Assumption 3 in \cite{ChKa20}. 
Furthermore, note that Assumption ($K_1$) in \cite{FeKa13} implies an analogous condition in the isotropic setting. 
However, \eqref{assmu1} does not imply \eqref{assmu1new}. This is due to the fact that for fixed $k \in \lbrace 1,\dots,d \rbrace$, the right-hand side of \eqref{assmu1new} converges to zero as $\alpha_k \nearrow 2$, whereas this behavior cannot be concluded from the properties of \eqref{assmu1}. Note that for our reference measure $\ma$, the constant in \eqref{assmu1} can be computed exactly:
\begin{align}
\ma(x_0, \R^d \setminus M_{\rho}(x_0)) = 2 \sum_{k=1}^{d_1} \frac{2-\alpha_k}{\alpha_k} \rho^{-\am} ~~ \forall \rho > 0 \,. \label{eq:tail-mu-axes} 
\end{align}

%%%%%
\subsection{Function spaces and weak solution concept}

We define the following function spaces
\begin{align*}
H_{\Omega}(\R^d) & = \Big\{ u: \,\R^d\to\R \text{ mb.}  \quad | \quad 
u\equiv 0 \text{ on } \R^d\setminus\Omega, \quad \|u\|_{H_{\Omega}(\R^d)}^2 = \|u\|_{L^2(\Omega)}^2 + 
\sum_{k=1}^d\mathcal{E}^{(k)}(u,u) <\infty \Big\},\\
H(\Omega) &= \Big\{ u : \Omega \to \R \text{ mb.}  \quad | \quad   \| u \|_{H(\Omega)}^2 = \| u \|_{L^2(\Omega)}^2 + \sum_{k=1}^d\mathcal{E}^{(k)}_{\Omega}(u,u) < \infty \Big\},\\
V(\Omega | \R^d) &= \Big\{ u : \Omega \to \R \text{ mb.}   \quad | \quad   \| u \|_{V(\Omega | \R^d)}^2 = \| u \|_{L^2(\Omega)}^2 + \sum_{k=1}^d [u]^2_{V^{(k)}(\Omega | \R^d)} < \infty \Big\}.
\end{align*}
Here, we set
\begin{align*}
\cE_{\Omega}^{(k)}(u,u) &= 
\begin{cases}
\int_{\Omega}\int_{\Omega} (u(x) - u(y)) \mu^{(k)}(x,\d y) \d x ~~ \text{ for } k = 1,\dots,d_1,\\
\int_{\Omega} |\partial_k u(x)|^2 \d x~~ \text{ for } k = d_1+1,\dots,d,
\end{cases},\\
[u]_{V^{(k)}(\Omega | \R^d)}  &= 
\begin{cases}
\int_{\Omega}\int_{\R^d} (u(x) - u(y)) \mu^{(k)}(x,\d y) \d x ~~ \text{ for } k = 1,\dots,d_1,\\
\int_{\Omega} |\partial_k u(x)|^2 \d x~~ \text{ for } k = d_1+1,\dots,d,
\end{cases},\\
\ma(x,\d y)&= \sum_{k=1}^{d_1} \mu^{(k)}(x,\d y) = \sum_{k=1}^{d_1} \Big( (2-\alpha_k) |x_k-y_k|^{-1-\alpha_k}\,  \d y_k\prod_{i\neq k}\delta_{\{x_i\}}(\d y_i) \Big).
\end{align*}

Let us denote by $H_{loc}(\Omega)$ the space of all functions $v : \Omega \rightarrow \R$ such that $v\Phi \in H(\Omega)$ for every $\Phi \in C_c^{\infty}(\Omega)$. In particular, a function $v : \Omega \rightarrow \R$ satisfies $v \in H_{loc}(\Omega)$ if and only if $v \in H(\Omega')$ for every $\Omega' \Subset \Omega$.\\
In the following we introduce the concept of a local weak (super-/sub-/) solution to $\partial_t - L = f$. The according definition is inspired by \cite{FeKa13}.

\begin{definition}[Local weak solution]\label{defweaksolution}
Assume $Q = I \times \Omega \subset \R^{d+1}$ and $f \in L^{\infty}(Q)$. We say that $u \in L^{\infty}(I;L^{\infty}(\R^d))$ is a supersolution of 
\begin{align}
\partial_t u - Lu = f,~~ \text{ in } I \times \Omega = Q,
\end{align}
if
\vspace{-0.2cm}
\begin{itemize}
\item[(i)] there exist $x_0 \in \R^d$ and $r \in (0,3]$ such that $\Omega \subset M_r(x_0)$ and 
\begin{align*}
u \in C_{loc}(I;L^2_{loc}(\Omega)) \cap L^2_{loc}(I;H_{loc}(M_r(x_0))),
\end{align*}

\item[(ii)] for every subdomain $\Omega' \Subset \Omega$ and every subinterval $[t_1,t_2] \subset I$ and for every non-negative test function $\phi \in H^1_{loc}(I;L^2(\Omega')) \cap L^2_{loc}(I;H_{\Omega'}(\R^d))$,
\begin{align}
\label{weaksol}
\begin{aligned}
\int_{\Omega'} \phi(t_2,x)u(t_2,x) \d x &- \int_{\Omega'} \phi(t_1,x)u(t_1,x) \d x - \int_{t_1}^{t_2} \int_{\Omega'} u(t,x) \partial_t \phi(t,x) \d x \d t \\
& + \int_{t_1}^{t_2} \cE^t(u,\phi) \d t \ge \int_{t_1}^{t_2} \int_{\Omega'} f(t,x)\phi(t,x) \d x \d t.
\end{aligned}
\end{align}
\end{itemize}
Sometimes we write $\partial_t u - L u \ge f$ in $I \times \Omega$ whenever $u$ is a supersolution in the sense of the foregoing definition. Subsolutions and solutions are defined analogously, except for the fact that in the definition of a solution, there is no restriction on the sign of the test function $\phi$.
\end{definition}

%\begin{remark}
%We emphasize that although we have written (super-/sub-/) solution in the foregoing definition, it rather is a weak (super-/sub-/) solution than a classical solution. 
%Since we only use this concept of solutions and no confusion can arise, we simply write (super-/sub-/) solution.
%\end{remark}

Note that by construction, weak solutions satisfy $u \in V(\Omega | \R^d)$. The following lemma justifies that all terms in \eqref{weaksol} are well-defined.

\begin{lemma}\label{lemma:finiteenergweaksol}
Let $(\mu,A) \in \cK(\alpha_0,\Lambda)$. Under the assumptions in \autoref{defweaksolution}:
\[\int_{t_1}^{t_2} \vert \cE^t(u,\phi) \vert \d t < \infty ~~ \forall [t_1,t_2] \subset I.\] 
\end{lemma}

\begin{proof}
Let $[t_1,t_2] \subset I$. Note that there is $\eps \in (0,2)$ small enough such that $\Omega' \Subset M_{r-\eps}(x_0) \Subset M_r(x_0)$ and $M_{\eps}(x) \cap (\R^d \setminus M_{r-\eps}(x_0)) = \emptyset$ for every $x \in \Omega'$. 
By \eqref{assmu0}, \eqref{assmu1}, and \eqref{eq:fct-space-comp},
\begin{align*}
\vert & \cE_t^{\mu}(u,\phi)\vert \le \int_{M_{r-\eps}(x_0)} \int_{M_{r-\eps}(x_0)} \vert u(t,x)-u(t,y)\vert \vert \phi(t,x)-\phi(t,y)\vert \mu(t,x,\d y) \d x  \\
& \hspace*{4cm} + 2 \int_{ M_{r-\eps}(x_0)} \int_{\R^d \setminus M_{r-\eps}(x_0)} \phi(t,x) \vert u(t,x)-u(t,y)\vert \mu(t,x,\d y) \d x\\
& \le \Lambda \Vert u(t) \Vert_{H^{\ma}(M_{r-\eps}(x_0))} \Vert \phi(t) \Vert_{H^{\ma}(M_{r-\eps}(x_0))} \\
& \hspace*{4cm} + 4 \Vert u(t) \Vert_{L^{\infty}(\R^d)} \int_{\Omega'} \phi(t,x) \mu(t,x,\R^d \setminus M_{\eps}(x)) \d x \\
& \le c \Vert \phi(t) \Vert_{H^{\ma}(\R^d)} \left(\Vert u(t) \Vert_{H^{\ma}(M_{r-\eps}(x_0))} + \Vert u(t) \Vert_{L^{\infty}(\R^d)} \right),
\end{align*}
where $c = c(\Lambda,d,\eps) > 0$. Since $u \in L^{\infty}(I;L^{\infty}(\R^d))$, it holds that $\int_{t_1}^{t_2} \vert \cE^{\mu}_t(u,\phi) \vert \d t < \infty$.
\end{proof}

The concept of weak solutions requires a careful study of admissibility of test functions. The following remark justifies the usage of those test functions that we actually are working with. 

\begin{remark}
We will not use \eqref{weaksol} directly but we work with the following inequality
\begin{align}
\label{weaksol2}
\int_{\Omega'} \partial_t u(t,x)\phi(t,x)\d x + \cE^{\mu(t)}(u,\phi) \ge \int_{\Omega'} f(t,x)\phi(t,x)\d x
\end{align}
for a.e. $t \in I$, applied to test functions of the form $\phi(t,x) = \psi(x)u^{-q}(t,x)$, where $q \in \R$, $u$ is a positive supersolution in $I \times \Omega$ and $\psi$ is a suitable cut-off function (see \autoref{thmlogu}, \autoref{thmit1} and \autoref{thmit2}). In order to test with $\phi$ for $q > 0$ we assume in addition that $u > \eps$ for some $\eps > 0$. In particular, we assume that $u$ is a.e. differentiable in $t$. We should mention that null sets in time will change \eqref{weaksol2}. However this will not affect our computations since we integrate the quantities in time.
Such alternative formulation has been already used in the article \cite{FeKa13}. For an explanation why this approach is legitimate, we refer the reader to \cite[Lemma 7.1]{KaWe22b}, and the discussion in \cite[Chapter 7.4]{Wei22}. See also \cite[p. 27-28]{FeKa13}. 
At this point, we emphasize that \eqref{assmu1} is needed for the derivation of \eqref{weaksol2}. \\

\end{remark}

Let us introduce cylindrical domains as follows.
\begin{definition}[Cylindrical domains]
For $r>0$, $t\in\R$ and $x\in\R^d$, we define
\[ I_r(t) := (t - r^{\am}, t + r^{\am}) \quad \text{ and } \quad  Q_r(t,x) = I_r(t) \times M_r(x), \]
where $M_r(x)$ is defined as in \autoref{def:M_r}. Moreover, we define
\begin{align*}
& I_{\oplus}(r) = (0,r^{\am}) \quad \text{ and } \quad I_{\ominus}(r) = (-r^{\am},0),\\
&  Q_{\oplus}(r) = I_{\oplus}(r) \times M_r \quad \text{ and } \quad Q_{\ominus}(r) = I_{\ominus}(r) \times M_r.
\end{align*}
\end{definition} 

The following result describes the scaling behavior of supersolutions $u$ to $\partial_t u - L u = f$.

\begin{lemma}[Scaling properties]
\label{thmscaling}
Let $\Lambda \ge 1$, $\alpha_0 \in (0,2)$, $\tau \in \R$, $\xi \in \R^d$ and $r > 0$. 
Furthermore, let $(\mu(t,x,\cdot))_{(t,x) \in I \times \R^d}$ be a family of measures satisfying \eqref{assmu0} and
\begin{itemize}
\item \eqref{assmu1new} for every $x_0 \in \R^d$, $k \in \lbrace 1,\dots,d\rbrace$:
\begin{align}
\mu(t,x_0, \R^d \setminus E_{\rho}^k(x_0)) \le 
\begin{cases}
\Lambda (2-\alpha_k)\rho^{-\am}, &\rho \in (0,2r),\\
\Lambda (2-\alpha_k)\rho^{-\theta}, &\rho > 2r.
\end{cases}
\end{align}
\item \autoref{assumption:func-ineq} for every $x_0 \in \R^d$ and $\rho \in (0,3r)$,
%\item \eqref{assmu4} for every $x_0 \in \R^d$ and $\rho \in (0,r]$.
\end{itemize}
Assume $f\in L^{\infty}(Q_r(\tau,\xi))$. Let $u$ be a supersolution to $\partial_t u -L u = f$ in $Q_r(\tau,\xi)$.
Define $(\tilde{t},\tilde{x}) = (r^{\am} t + \tau, \left( r^{\am/\alpha_k}x_k + \xi_k \right)_{k=1}^d)$. Then $\U(t,x) := u(\tilde{t},\tilde{x})$ satisfies
\begin{align*}
&\left( \int_{M_1} \phi(t,x) \U(t,x) \right)^{1}_{t=-1} - \int_{Q_1(0,0)} \U(t,x) \partial_t \phi(t,x) \d x \d t\\
& + \int_{-1}^{1} \cE^{\tilde{\mu}(t)}(\tilde{u},\phi) + \cE^{\tilde{A}(t)}(\tilde{u},\phi) \d t \ge \int_{Q_1(0,0)} \tilde{f}(t,x) \phi(t,x) \d x \d t
\end{align*}
for every non-negative test function $\phi \in H^1((-1,1); L^2(M_1)) \cap L^2((-1,1); H_{M_1}^{\ma}(\R^d))$, where $\tilde{f}(t,x) = r^{\am} f(\tilde{t},\tilde{x})$ and
\begin{align*}
\tilde{\mu}(t,x, \d y) =  r^{\am} \mu^{*}(\tilde{t},\tilde{x},\d y), \qquad \tilde{A}(t,x) = A(\tilde{t},\tilde{x})
\end{align*}
with $\mu^{*}(t,x,\tilde{D}) := \mu(t,x,D)$, where $\tilde{D} = \lbrace \tilde{y} : y \in D \rbrace$, $D \subset \R^d$ measurable, is the push-forward measure of $\mu(t,x,\cdot)$. \\
In particular, $\U$ solves $\partial_t \U - \tilde{L} \U \ge \tilde{f}$ in $Q_1(0,0)$, where $\tilde{L}$ is associated to $\tilde{\cE}:= \cE^{\tilde{\mu}} + \cE^{\tilde{A}}$. Furthermore, it holds that $(\tilde{\mu},\tilde{A}) \in \cK(\alpha_0,\Lambda)$.
\end{lemma}

\begin{remark}
Let us briefly explain in two steps how the foregoing scaling properties go into the proof of \autoref{thm:HR}.
\begin{itemize}
\item[(i)] In the setting of \autoref{thmscaling}, let $Q = I \times \Omega$ with $Q \Supset Q_r(\tau,\xi)$. Let $u$ be a supersolution to $\partial_t u - L u = f$ in $Q$. It is easy to see that in this case, $u$ is a supersolution to the same equation in $Q_r(\tau,\xi)$ and the \autoref{thmscaling} can be applied.
\item[(ii)] Apparently, scaling behaves similarly when it is applied directly to domains $Q$ of the more general form $Q = (\tau - c_1 r^{\am},\tau) \times M_{c_2 r}(\xi)$, where $c_1,c_2 > 0$. In the situation of (i) with this choice of $Q$ and by defining $\tilde{t},\tilde{x},\U,\tilde{f},\tilde{\mu},\tilde{L}$ as in \autoref{thmscaling}, we obtain that $\U$ is a solution to $\partial_t \U - \tilde{L} \U \ge \tilde{f}$ in $(-c_1,0) \times M_{c_2}$. 
This observation goes into the proof of \autoref{thm:HR}.
\end{itemize}
\end{remark}

The main auxiliary result in the proof of the weak Harnack inequality is the following result by Bombieri and Giusti, see \cite[Lemma 2.2.6]{Sal02}. 

\begin{lemma}[Lemma by Bombieri and Giusti]
\label{thmBG}
Let $m, c_0 > 0$, $\theta \in [1/2,1]$, $\eta \in (0,1)$ and $0<p_0\leq \infty$. Let $(U(r))_{\theta\leq r\leq 1}$ be a non-decreasing family of domains $U(r)
\subset \R^{d+1}$. Furthermore, assume that $w : U(1) \rightarrow (0,\infty)$
is a measurable function which satisfies for every $s > 0$:
\begin{align} 
\label{BGAss1}
\vert U(1) \cap \{ \log w > s \} \vert \le \frac{c_0}{s}
\vert U(1) \vert.
\end{align}
Additionally suppose that for all $r,R \in [\theta,1], r<R$ and for all $p \in (0,1 \wedge \eta p_0)$
\begin{align} 
\label{BGAss2}
\Bigl( \int_{U(r)} w^{p_0} \Bigr)^{1/p_0} \le \left( \frac{c_0}{(R-r)^m
\vert U(1)\vert}\right)^{1/p-1/p_0} \Bigl( \int_{U(R)} w^p \Bigr)^{1/p} < \infty.
\end{align}
Then there is a constant $C=C(\theta,\eta,m,c_0, p_0)$ such that
\begin{align} 
\label{BGRes}
 \Bigl( \int_{U(\theta)} w^{p_0} \Bigr)^{1/p_0} \le C \vert U(1) \vert^{1/p_0}\ .
\end{align}
\end{lemma}

The weak Harnack inequality (\autoref{thm:weakHarnack}) follows from a two-fold application of \autoref{thmBG}. 
Hence, the main difficulty is to show that the assumptions \eqref{BGAss1} and \eqref{BGAss2} are 
satisfied and that \autoref{thmBG} can be applied for supersolutions under consideration.
In order to verify these assumptions, we need to establish Moser iteration schemes for positive and negative exponents.

\subsection{Cut-off functions}
\label{cutofffct}
The goal of this subsection is to introduce a class of suitable cut-off functions and derive an anisotropic gradient estimate (see \autoref{thmcutoffest}) with the help of \autoref{assumption:symmetry} and \autoref{assumption:tail}.

\begin{definition}\label{def:testfct-ln}
We say that $(\tau_{x_0,r,\lambda})_{x_0,r,\lambda} \subset C^{1}(\R^d)$ is an admissible family of cut-off functions if there exists $c \ge 1$ such that for all $x_0 \in \R^d$, $r > 0$ and $\lambda > 1$, it holds that
%Let $x_0\in \R^d$, $r>0$ and $\lambda>1$. We call a function $\tau_{x_0,r,\lambda}\in C^1(\R^d)$ an \textit{admissible cut-off function},
%if it satisfies the following properties
\begin{itemize}
 \item $\supp(\tau_{x_0,r,\lambda})\subset M_{\lambda r}(x_0),$
 \item $ \|\tau_{x_0,r,\lambda}\|_\infty\leq 1,$
 \item $\tau_{x_0,r,\lambda}\equiv 1 \text{ on } M_r(x_0),$
 \item for all $k\in\{1,\dots,d_1\}$: $\|\partial_k \tau_{x_0,r,\lambda}\|_{\infty}\leq c(\lambda^{\am/\alpha_k}-1)^{-1} r^{-\am/\alpha_k}$ and 
 \item for all $k\in \{d_1+1,\dots,d\}$: $\|\partial_k \tau_{x_0,r,\lambda}\|_{\infty}\leq c(\lambda-1)^{-1} r^{-1}$.
\end{itemize}
\end{definition}

\begin{remark}
The existence of such admissible cut-off functions is standard.
We simply write $\tau$ for any such function from $(\tau_{x_0,r,\lambda})_{x_0,r,\lambda}$, if the respective choice of $x_0, r$ and $\lambda$ is obvious.
\end{remark}

An important property of such cut-off functions is given by the following lemma.

\begin{lemma}
\label{thmcutoffest}
There is $c_1 = c_1(d,\alpha_0,\Lambda) > 0$ such that for each $(\mu,A)$ satisfying \autoref{assumption:tail} with $\Lambda \ge 1$, every $x_0\in\R^d$, $t \in I$, $r \in (0,1]$, $\lambda \in (1,2]$ and every admissible cut-off function $\tau$
\begin{align}
\label{cutoffest}
\sup_{(t,x) \in I \times \R^d} \Gamma^t(\tau,\tau)(t,x) \le c_1 r^{-\am} \sum_{k = 1}^d (\lambda^{\am/\alpha_k}-1)^{-\alpha_k},
\end{align}
where we define
\begin{align*}
\Gamma^t(\tau,\tau)(t,x) = \int_{\R^d} (\tau(x)-\tau(y))^2 \mu(t,x,\d y) &+ (A(t,x),\nabla \tau(x) ,\nabla \tau(x)).
\end{align*}
\end{lemma}
\begin{proof}
Let $x \in \R^d$ and $t \in I$ be arbitrary. Note that the bound for the second summand of $\Gamma^t(\tau,\tau)$ is an immediate consequence of \eqref{eq:A-upper} and the definition of $\tau$. Next, we turn to estimating the first summand of $\Gamma^t(\tau,\tau)$. For $y \in \R^d$, let $\ell=(\ell_0(x,y),\dots,\ell_{d}(x,y)) \in \R^{d  (d+1)}$ 
be a polygonal chain connecting $x$ and $y$ with \[ \ell_k(x,y) = (l^k_1,\dots,l^k_d), \quad \text{where } \begin{cases}
                                                       l^k_j = y_j, &\text{if } 
j\leq k, \\
                                                       l^k_j = x_j, &\text{if } 
j>k.
                                                      \end{cases} \]
 Then $x=\ell_0(x,y)$, $y=\ell_d(x,y)$, and $|\ell_{k-1}(x,y)-\ell_k(x,y)|=|x_k-y_k|$ for all $k\in\{1,\dots,d\}$. We observe
 \begin{align*}
 \int_{\R^d}(\tau(x)-\tau(y))^2\mu(t,x,\d y)\leq 
d\sum_{k=1}^d\int_{\R^d}(\tau(\ell_{k-1}(x,y))-\tau(\ell_{k}(x,y)))^2\mu(t,x,\d 
y)=:d\sum_{k=1}^d I_k.
 \end{align*}
 For $k\in\{1,\dots,d\}$, we set 
$\eta_k = \left( \lambda^{\am/\alpha_k} - 1\right)^{\alpha_k/\am}$. Then
 \begin{align*}
 I_k&=\int_{M_{\eta_kr}(x)}(\tau(\ell_{k-1}(x,y))-\tau(\ell_{k}(x,y)))^2\mu(t,x,\d 
y)\\
 & \qquad + \int_{\R^d\setminus 
M_{\eta_kr}(x)}(\tau(\ell_{k-1}(x,y))-\tau(\ell_{k}(x,y)))^2\mu(t,x,\d y)=:A_k+B_k.
 \end{align*}
Now \eqref{assmu1} implies that
\begin{align*}
B_k \le \mu(t,x,\R^d \setminus M_{\eta_k r}(x)) \le 2d\Lambda (\eta_k r)^{-\am} = 2d\Lambda r^{-\am} \left( \lambda^{\am/\alpha_k} -1 \right)^{-\alpha_k}.
\end{align*}
For the upper bound of $A_k$, we make the following observation for every $\rho \in (0,2)$ and $x \in \R^d$, starting from \autoref{assumption:tail}:
\begin{align*}
\int_{M_{\rho}(x)} \vert x_k - y_k \vert^2 \mu(t,x,\d y) &\le \int_{E_{\rho}^k(x)} \vert x_k - y_k \vert^2 \mu(t,x,\d y)\\
%& = \sum_{n = 0}^{\infty} \int_{E_{2^{-n}\rho}^k(x) \setminus E_{2^{-n-1}\rho}^k(x)} \vert x_k - y_k \vert^2 \mu(t,x,\d y)\\
& \le \sum_{n = 0}^{\infty} \left( (2^{-n}\rho)^{\am/\alpha_k} \right)^2 \mu(t,x, E_{2^{-n}\rho}^k(x) \setminus E_{2^{-n-1}\rho}^k(x)) \\
%& \le \sum_{n = 0}^{\infty} (2^{-n}\rho)^{2 \am/\alpha_k} \mu(t,x, \R^d \setminus E_{2^{-n-1}\rho}^k(x)) \\
& \le \Lambda \sum_{n = 0}^{\infty} (2^{-n}\rho)^{2 \am/\alpha_k} (2-\alpha_k) (2^{-n-1}\rho)^{-\am}\\
%& = \Lambda \rho^{2\am/\alpha_k - \am} \sum_{n = 0}^{\infty} (2-\alpha_k) 2^{-2\am n/\alpha_k + \am n + \am} \\
& = 4\Lambda \rho^{2\am/\alpha_k - \am} \frac{2-\alpha_k}{1 - 2^{\am-2\am/\alpha_k}} \le 8\Lambda \rho^{2\am/\alpha_k - \am}.
\end{align*}
Hence,
\begin{align*}
A_k &{\le} \Vert \partial_k \tau \Vert_{\infty}^2 \int_{M_{\eta_k r}(x)} \vert x_k - y_k \vert^2 \mu(t,x,\d y)\\
& \le 8 \Lambda c^2 \left( \lambda^{\am/\alpha_k} -1 \right)^{-2}r^{-2\am/\alpha_k} (\eta_k r)^{2\am/\alpha_k-\am}\\
%& = 8\Lambda c^2 \left( \lambda^{\am/\alpha_k} -1 \right)^{-2}r^{-2\am/\alpha_k}  \left( \lambda^{\am/\alpha_k} - 1\right)^{2-\alpha_k} r^{2\am/\alpha_k-\am}\\
& = 8\Lambda c^2 r^{-\am}\left( \lambda^{\am/\alpha_k} - 1\right)^{-\alpha_k}.
\end{align*}
Choosing $c_1 = d^2\left(2d\Lambda + 8\Lambda c^2 \right)$ and taking the supremum in $x$, proves the assertion. 
\end{proof}

\begin{remark}
The preceding result can already be found in \cite[Lemma 2.1]{ChKa20}. However, introducing \autoref{assumption:tail} 
in this work, makes the result robust in the sense, that \eqref{cutoffest} remains true in the limit case $\alpha_k \nearrow 2$.
\end{remark}

\subsection{Weighted Poincar\'e inequality}

We conclude this section by deriving a weighted Poincar\'e inequality from \eqref{eq:poincare-assum}.

\begin{theorem}[Weighted Poincar\'e inequality]
\label{thmweightedPoincare}
Let $\psi : M_{3/2} \rightarrow [0,1]$ be defined by $\psi(x) = \left((3 - 2 \sup_{k = 1, \dots, d} \vert x_k \vert^{\alpha_k/\am}) \wedge 1 \right) \vee 0$. Then there is $c_2 = c_2(d,\alpha_0,\Lambda) > 0$ such that for every $(\mu,A) \in \cK(\alpha_0,\Lambda)$ and every $v \in L^1(M_{3/2},\psi(x)\d x)$ it holds
\begin{align*}
\int_{M_{3/2}} (v(x) - v_{\psi})^2 \psi(x) \d x &\le c_2 \int_{M_{3/2}} \int_{M_{3/2}} (v(x)-v(y))^2 (\psi(x) \wedge \psi(y)) \mu(t,x,\d y) \d x\\
&\quad+  c_2 \int_{M_{3/2}} (A(t,x) \nabla v(x),\nabla v(x)) \psi(x) \d x,
\end{align*}
where $v_{\psi} = \left( \int_{M_{3/2}} \psi(x) \d x \right)^{-1} \int_{M_{3/2}} v(x) \psi(x) \d x$.
\end{theorem}

\begin{proof}
By \eqref{eq:poincare-assum} we know already that there is a constant $c = c(d) > 0$ such that for every $x_0 \in \R^d$ and $r > 0$ and every $v \in L^2(M_r(x_0))$ the following Poincar\'e inequality holds:
\begin{align}
\label{poincarethm}
\int_{M_r(x_0)} (v(x) - [v]_{M_r(x_0)})^2 \d x \le c r^{\am} \cE^t_{M_r(x_0)}(v,v).
\end{align}
Note that on $\R^d$, $d(x,y) := \sup_{k = 1, \dots, d} \vert x_k - y_k \vert^{\alpha_k/\am}$ defines a metric. We can write $\psi(x) = \phi(d(x,0))$ with $\phi(z) = ((3-2z) \wedge 1) \vee 0$ decreasing in $z$. By \cite[Theorem 1]{DyKa13}, we conclude
\begin{align*}
\int_{M_{3/2}} (v(x) - v_{\psi})^2 \psi(x) \d x &\le c \int_{M_{3/2}} \int_{M_{3/2}} (v(x)-v(y))^2 (\psi(x) \wedge \psi(y)) \mu(t,x,\d y) \d x\\
&\quad + c \int_{M_{3/2}} (A(t,x) \nabla v(x),\nabla v(x)) \psi(x) \d x
\end{align*}
for some $c = c(d,\alpha_0,\Lambda) > 0$. Therefore the assertion of the theorem follows.
\end{proof}

\section{An inequality for \texorpdfstring{$\log u$}{logu}}\label{sec:logineq}
In this section, we study the logarithm of supersolutions to the parabolic equation $\partial_t u - Lu = f$. 
The main result of this section is \autoref{thmlogu}, which allows us to verify the condition \eqref{BGAss1} in the Lemma by Bombieri and Giusti (see \autoref{thmBG}).
This result follows from the weighted Poincar\'e inequality (\autoref{thmweightedPoincare}) and a Caccioppoli-type estimate.

\begin{lemma}
For any $(\mu,A) \in \cK(\alpha_0,\Lambda)$, $t \in I$, $x_0 \in \R^d$, $r > 0$, $\lambda \in (1,2]$, $\eps > 0$, and $u \in H(M_{\lambda r}(x_0)) \cap L^{\infty}(\R^d)$ 
\begin{align}
\label{eq:algebra-logu-nl}
c_2 \cE^{\mu(t)}(\tau,\tau) &+ \cE^{\mu(t)}(u,-\tau^2 \U^{-1}) \\
&\ge c_1 \int_{M_{\lambda r}(x_0)} \int_{M_{\lambda r}(x_0)} (\tau^2(x) \wedge \tau^2(y)) \left(\log \frac{\U(y)}{\tau(y)} - \log \frac{\U(x)}{\tau(x)} \right)^2 \mu(t,x,\d y)\d x, \nonumber \\
c_2 \cE^{A(t)}(\tau,\tau) &+ \cE^{A(t)}(u,-\tau^2 \U^{-1}) \label{eq:algebra-logu-l} \\
&\ge c_1 \int_{M_{\lambda r}(x_0)} \tau^2(x) \left(A(t,x) \nabla  \log \frac{\U(x)}{\tau(x)} , \nabla \log \frac{\U(x)}{\tau(x)} \right)\d x, \nonumber 
\end{align}
where $\tau = \tau_{x_0,r,\lambda}$, $\U = u +\eps$, and $c_1, c_2 > 0$ depend only on $d$, $\Lambda$, $\alpha_0$..
\end{lemma}

\begin{proof}
For the proof of \eqref{eq:algebra-logu-nl}, we refer the reader to  \cite[Lemma 3.6]{KaWe22a}, where the same result is given in a more general nonsymmetric setup. The proof of \eqref{eq:algebra-logu-l} is standard.
\end{proof}

Let us now state and prove the main result of this section, compare \cite[Proposition 4.2]{FeKa13}.
\begin{proposition}[Estimate for $\log u$]
\label{thmlogu}
There is $C = C(d,\alpha_0,\Lambda) > 0$ such that for every $\mu \in \cK(\alpha_0,\Lambda)$ and every supersolution $u$ of $\partial_t u - Lu = f$  in $Q = (-1,1) \times M_2$ which satisfies $u \ge \eps > 0$ in $(-1,1) \times \R^d$, there is a constant $a = a(\U) \in \R$ such that the following two inequalities hold true for any $s > 0$:
\begin{align}
\label{logu1}
\vert Q_{\oplus}(1) \cap \lbrace \log \U < -s-a \rbrace \vert \le C \frac{|M_1|}{s}, \qquad \vert Q_{\ominus}(1) \cap \lbrace \log \U > s-a \rbrace \vert \le C \frac{|M_1|}{s},
\end{align}
where $\U := u + \Vert f \Vert_{L^{\infty}(Q)}$.
\end{proposition}

\begin{proof}
Let $\tau = \tau_{x_0,1,3/2}$ be as in \autoref{thmweightedPoincare}. We define the test function $\phi(t,x) = \tau^2(x)\U^{-1}(t,x)$ and the auxiliary function $v(t,x) = -\log \frac{\U(t,x)}{\tau(x)}$. Since $u$ is a supersolution to $\partial_t u - Lu = f$, we can apply \eqref{weaksol2} and obtain that for a.e. $t \in (-1,1)$
\begin{align*}
\int_{M_{3/2}} \tau^2(x) \partial_t v(t,x) \d x + \cE^t(\U,-\tau^2\U^{-1}) \le -\int_{M_{3/2}} \tau^2(x)\U^{-1}(t,x)f(t,x)\d x,
\end{align*}
where we used the observation $\partial_t \U(t,x) = -\U(t,x) \partial_t v(t,x)$.
Using $\Vert f/\U\Vert_{L^{\infty}(Q)} \le 1$, $\tau^2(x) \leq 1$, \eqref{eq:algebra-logu-nl}, \eqref{eq:algebra-logu-l}, and \autoref{thmcutoffest}:
\begin{align*}
\int_{M_{3/2}} &\tau^2(x)\partial_t v(t,x) \d x + \int_{M_{3/2}} \int_{M_{3/2}} (\tau^2(x) \wedge \tau^2(y))(v(t,x)-v(t,y))^2 \mu(t,x,\d y)\d x\\
&\quad + \int_{M_{3/2}} \tau^2(x) (A(t,x)\nabla v(t,x),\nabla v(t,x)) \d x\\
& \le -\int_{M_{3/2}} \tau^2(x)\U^{-1}(t,x)f(t,x)\d x + c_2 \cE^t(\tau,\tau) \le C \vert M_{3/2} \vert.
\end{align*}
Hence, applying the weighted Poincar\'e inequality \autoref{thmweightedPoincare} leads to
\begin{align}\label{eq:logproofhelp1}
\int_{M_{3/2}} \tau^2(x)\partial_t v(t,x) \d x + c_1 \int_{M_{3/2}} (v(t,x)-V(t))^2 \tau^2(x) \d x\le c_2 \vert M_1 \vert
\end{align}
for constants $c_1,c_2 > 0$ depending only on $d,\alpha_0,\Lambda$ and where
\begin{align*}
V(t) := \frac{\int_{M_{3/2}}\tau^2(x) v(t,x)\d x}{\int_{M_{3/2}} \tau^2(x) \d x}.
\end{align*}
Integrating inequality \eqref{eq:logproofhelp1} over $[t_1,t_2] \subset (-1,1)$, yields
\begin{align*}
\left( \int_{M_{3/2}} \tau^2(x)v(t,x)\d x \right)_{t=t_1}^{t_2} + c_1 \int_{t_1}^{t_2} \int_{M_{3/2}} (v(t,x)-V(t))^2 \tau^2(x) \d x \d t \le c_2 (t_2-t_1)\vert M_1 \vert.
\end{align*}
Note that $\tau$ satisfies
\begin{align*}
%& 0 \le \tau \le 1,\\
& \int_{M_{3/2}} \tau^2(x) \d x \le \vert M_{3/2} \vert = \left(\frac{3}{2}\right)^{\am\beta} \vert M_1 \vert \le 4^{d/{\alpha_0}} \vert M_1 \vert,
%& \tau \equiv 1 \text{ in } M_1.
\end{align*}
Thus, dividing by $\int_{M_{3/2}} \tau^2(x) \d x$, we obtain
\begin{align}
\begin{aligned}
\label{helpV}
V(t_2)-V(t_1) &+ \frac{c_3}{\vert M_1 \vert} \int_{t_1}^{t_2} \int_{M_1} (v(t,x)-V(t))^2 \d x \d t\\
& \le V(t_2)-V(t_1) + \frac{c_1}{ \int_{M_{3/2}} \tau^2(x) \d x} \int_{t_1}^{t_2} \int_{M_1} (v(t,x)-V(t))^2 \d x \d t\\
& \le c_2 (t_2-t_1) \frac{\vert M_1 \vert}{\int_{M_{3/2}} \tau^2(x) \d x} \le c_2 (t_2-t_1),
\end{aligned}
\end{align}
where $c_3 = c_1 / 4^{d/\alpha_0}$. Having at hand the estimate \eqref{helpV}, the proof of \eqref{logu1} follows line by line the arguments in \cite[Proof of Proposition 4.2]{FeKa13} (see also \cite[Theorem 9.2.1]{Wei22}) without any changes.
\end{proof}

\section{Moser iteration}\label{sec:Moser} \allowdisplaybreaks
In this section we establish the Moser iteration scheme for negative exponents, as well as for small positive exponents. 
By carrying out the iteration we can obtain estimates for $\inf u$ and for small positive moments of $u$ from which we can show the second assumption \eqref{BGAss2} of \autoref{thmBG} and which allows us to prove the weak Harnack inequality.
The main ingredients in this section are the Sobolev inequality \eqref{eq:sobolev-assum}, \autoref{thmcutoffest} and 
the following Caccioppoli-type estimate.

\begin{lemma}
For any $(\mu,A) \in \cK(\alpha_0,\Lambda)$, $t \in I$, $x_0 \in \R^d$, $r > 0$, $\lambda \in (1,2]$, $\eps > 0$, and $u \in H(M_{\lambda r}(x_0)) \cap L^{\infty}(\R^d)$ and $q \ge 1 - \kappa^{-1}$, with $q \neq 1$,
\begin{align}
\label{eq:algebra-Moser-nl}
\begin{split}
\cE^{\mu(t)}_{M_{\lambda r}(x_0)}(\tau \U^{\frac{-q+1}{2}},\tau \U^{\frac{-q+1}{2}})
&\le c_1|q-1| \cE^{\mu(t)}(u,-\tau^{2}\U^{-q})\\
&\quad+ c_2 \left(1 \vee |q-1| \right) \Vert \U^{-q+1} \Gamma^{\mu(t)}(\tau,\tau)\Vert_{L^{1}(M_{\lambda r}(x_0))},
\end{split}
\end{align}
\begin{align}
\label{eq:algebra-Moser-l}
\begin{split}
\cE^{A(t)}_{M_{\lambda r}(x_0)}(\tau \U^{\frac{-q+1}{2}},\tau \U^{\frac{-q+1}{2}})
&\le c_1|q-1| \cE^{A(t)}(u,-\tau^{2}\U^{-q})\\
&\quad+ c_2 \left(1 \vee |q-1| \right) \Vert \U^{-q+1} \Gamma^{A(t)}(\tau,\tau)\Vert_{L^{1}(M_{\lambda r}(x_0))},
\end{split}
\end{align}
where $\kappa = 1 + \frac{1}{\beta}$, $\tau = \tau_{x_0,r,\lambda}$, $\U = u +\eps$, and $c_1, c_2 > 0$ depend only on $d$, $\Lambda$, $\alpha_0$.
\end{lemma}

\begin{proof}
For the proof of \eqref{eq:algebra-Moser-nl}, we refer to \cite[Lemma 3.5]{KaWe22a}. The proof of \eqref{eq:algebra-Moser-l} is standard.
\end{proof}

\subsection{Moser iteration for negative exponents} \allowdisplaybreaks
Recall that $I$ always denotes an open, bounded interval in $\R$ and $\Omega$ is a bounded domain in $\R^d$ such that $\Omega \subset M_r(x_0)$ for some $x_0 \in \R^d$ and $r \in (0,3]$.
%In this section we work with non-negative supersolutions $u$ of $\partial_t u - L u = f$ in $Q = I \times \Omega$. 

\begin{theorem}[Moser iteration scheme for negative exponents]
\label{thmit1}
Let $1/2 \le r < R \le 1$ and $p > 0$. There is $c = c(d,\alpha_0,\Lambda) > 0$ such that for every $\mu \in \cK(\alpha_0,\Lambda)$ and every non-negative supersolution $u$ of $\partial_t u - L u = f$ in $Q = I \times \Omega$, with $Q_{\ominus}(R) \Subset Q$ and $u \ge \eps > 0$ in $Q$ it holds
\begin{align}
\label{it1}
\Vert \U^{-1} \Vert^p_{L^{\kappa p}(Q_{\ominus}(r))} \le c (p+1) \left( R - r \right)^{-\am} \Vert \U^{-1} \Vert^p_{L^{p}(Q_{\ominus}(R))},
\end{align}
where $\kappa = 1 + \frac{1}{\beta}$ and $\U = u + \Vert f \Vert_{L^{\infty}(Q)}$. 
\end{theorem}

\begin{proof}
Throughout this proof $c > 0$ will always be considered as a placeholder for some (changing) positive constant that depends only on $d,\alpha_0,\Lambda$ (and on $\beta$, but note that $d/2 \le \beta \le d/\alpha_0$). 
As in the proof of \cite[Proposition 3.4]{FeKa13}, we redefine in the case $\Vert f \Vert_{L^{\infty}(Q)} = 0$ the function $\U = u + \delta$ for some $\delta > 0$. This guarantees that $u \ge \delta > 0$ in $I \times \R^d$ and ensures that the right-hand side of \eqref{helpMoserNeg1} is finite. Taking the limit $\delta \searrow 0$ in the end finishes the proof.\\
For $q > 1$, let $v(t,x) = \U^{\frac{1-q}{2}}(t,x)$ and let $\tau_{x_0 = 0,r = r,\lambda = R/r}  =: \psi : \R^d \rightarrow [0,1]$ be an admissible cut-off function in the sense of \autoref{cutofffct}. We define $\phi(t,x) = \U^{-q}(t,x) \tau^{2}(x)$. 
By \eqref{weaksol2},
\begin{align*}
\int_{M_R} -\tau^{2}(x) \U^{-q}(t,x) \partial_t \U(t,x) \d x + \mathcal{E}^t(\U,-\tau^{2}\U^{-q}) \le \int_{M_R} -\tau^{2}(x) \U^{-q}(t,x)f(t,x)\d x.
\end{align*}
Using \eqref{eq:algebra-Moser-nl}, \eqref{eq:algebra-Moser-l} and the observation $\partial_t (v^2) = (1-q) \U^{-q} \partial_t \U$ leads to
\begin{align}
\label{helpMoserNeg1}
\begin{aligned}
 &\int_{M_R} \tau^{2}(x) \partial_t{(v^2)}(t,x) \d x+ \cE^t_{M_R}(\tau \U^{\frac{-q+1}{2}},\tau \U^{\frac{-q+1}{2}})\\
& \quad\le c\left(1 \vee |q-1| \right) \Vert \U^{-q+1} \Gamma^t(\tau,\tau)\Vert_{L^{1}(M_R)} + c_1(q-1)\int_{M_R}  \tau^{2}(x) \vert \U^{-q}(t,x) \vert \vert f(t,x) \vert \d x.
\end{aligned}
\end{align}
By \autoref{thmcutoffest} and using the fact that $\Vert f/\U \Vert_{L^{\infty}(Q)} \le 1$, 
we obtain
\begin{align}\label{eq:helpMoserNeg4}
\begin{aligned}
\int_{M_R} & \tau^{2}(x) \partial_t(v^2)(t,x)\d x + \mathcal{E}^t_{M_r}(v,v)\\
& \le c( 1 \vee |q-1|)r^{-\am}\left( \sum_{k = 1}^d \left((R/r)^{\am/\alpha_k}-1\right)^{-\alpha_k}\right) \int_{M_R} v^2(t,x)\d x.
\end{aligned}
\end{align}
Let $\chi_{\ominus} : \R \rightarrow [0,1]$, $\chi_{\ominus}(t) = \left( \frac{t+R^{\am}}{R^{\am}-r^{\am}} \wedge 1\right) \vee 0$. Note that the function $\chi_{\ominus}$ has the following properties: 
\[\chi_{\ominus}(-R^{\am}) = 0, \hspace*{0.5em}  \Vert \chi_{\ominus}'\Vert_{\infty} \le  r^{-\am}\left( (R/r)^{\am} -1 \right)^{-1}, \hspace*{0.5em} \chi_{\ominus} \equiv 1 \text{ on } I_{\ominus}(r) \hspace*{0.5em}  \text{ and } \Vert \chi_{\ominus} \Vert_{\infty} \le 1.\] 
Multiplying \eqref{eq:helpMoserNeg4} by $\chi_{\ominus}^2$ and integrating from $-R^{\am}$ to some $t \in I_{\ominus}(r)$ leads to
\begin{align*}
&\int_{M_R} \tau^{2}(x) (\chi_{\ominus}(t)v(t,x))^2 \d x + \int_{-R^{\am}}^t \chi_{\ominus}^2(s) \cE^s_{M_r}(v(s),v(s)) \d s\\
& \le c(1 \vee |q-1|) r^{-\am} \left( \sum_{k = 1}^d \left((R/r)^{\am/\alpha_k}-1\right)^{-\alpha_k}\right)\int_{-R^{\am}}^t \chi_{\ominus}^2(s) \int_{M_R}v^2(s,x)\d x \d s\\
& + \int_{-R^{\am}}^t 2 \chi_{\ominus}(s) \vert \chi_{\ominus}'(s)\vert \int_{M_R} v^2(s,x) \d x \d s.
\end{align*}
Using the properties of $\chi$ implies
\begin{align}
\label{it1help1}
\begin{aligned}
&\sup_{t \in I_{\ominus}(r)} \int_{M_r} v^2(t,x) \d x + \int_{I_{\ominus}(r)} \cE^s_{M_r}(v(s),v(s)) \d s\\
&\le \frac{c(1 \vee |q-1|) }{r^{\am}}\left(  \sum_{k = 1}^d \left((R/r)^{\am/\alpha_k}-1\right)^{-\alpha_k} + \left( (R/r)^{\am} -1 \right)^{-1} \right) \int_{Q_{\ominus}(R)} v^2(s,x) \d x \d s.
\end{aligned}
\end{align}
Since both summands on the left-hand side of \eqref{it1help1} are non-negative, 
we can estimate each of them separately from above by the quantity on the right-hand side of \eqref{it1help1}. \\
We set $\Theta = \frac{\beta}{\beta-1}$ and let $\Theta'$ be the corresponding Hölder exponent, that is $\Theta' = \beta$. 
Furthermore, let $\kappa = 1 + \frac{1}{\Theta'} = 1 + \frac{1}{\beta}$. \\
Applying the Hölder inequality and the Sobolev inequality  \eqref{eq:sobolev-assum} with $\lambda = \frac{r+R}{2r}$, leads to
\begin{align}\label{eq:helpMoserneg}
\begin{aligned}
&\int_{Q_{\ominus}(r)} v^{2\kappa}(t,x) \d x \d t = \int_{I_{\ominus}(r)}\int_{M_r} v^{2}(t,x)v^{2/\beta}(t,x) \d x \d t\\
&\le \int_{I_{\ominus}(r)} \left(\int_{M_r} v^{2\Theta}(t,x) \d x \right)^{1/\Theta} \left( \int_{M_r} v^{2}(t,x)\d x\right)^{1/\beta} \d t \\
& \le c \sup_{t \in I_{\ominus}(r)} \left( \int_{M_r} v^{2}(t,x) \d x\right)^{1/\beta} \Bigg[\int_{I_{\ominus}(\frac{r+R}{2})} \cE^s_{M_{\frac{r+R}{2}}}(v(s),v(s))\\
& \hspace*{6em}+ r^{-\am}\left( \sum_{k = 1}^d \left(\left(\frac{r+R}{2r}\right)^{\am/\alpha_k} -1\right)^{-\alpha_k}\right) \int_{Q_{\ominus}(\frac{r+R}{2})} v^2(s,x) \d x \d s \Bigg].
\end{aligned}
\end{align}
In the next step, we combine this estimate with \eqref{it1help1}. \\
We apply \eqref{it1help1} once with $r = r, R = \frac{r+R}{2}$ (note: $\lambda = \frac{r+R}{2r}$) in order to estimate the first factor in the above inequality 
and once with $r = \frac{r+R}{2}, R = R$ (note: $\lambda = \frac{2R}{r+R}$) in order to estimate the nonlocal part in the second factor in the foregoing inequality. \\
For brevity, we introduce
\begin{align*}
&\Upsilon_1(r,R,\alpha_k) =  \sum_{k = 1}^d \left(\left(\frac{r+R}{2r}\right)^{\am/\alpha_k}-1\right)^{-\alpha_k}, & \Upsilon_2(r,R) = \left( \left( \frac{r+R}{2r} \right)^{\am} -1 \right)^{-1},\\
&\Upsilon_3(r,R,\alpha_k) = \sum_{k = 1}^d \left(\left( \frac{2R}{r+R} \right)^{\am/\alpha_k}-1\right)^{-\alpha_k}, & \Upsilon_4(r,R) = \left( \left( \frac{2R}{r+R} \right)^{\am} -1 \right)^{-1}.
\end{align*}
By \eqref{eq:helpMoserneg} and \eqref{it1help1},
\begin{align*}
\int_{Q_{\ominus}(r)}&  v^{2\kappa}(t,x) \d x \d t  \le  c \Big[ (1 \vee |q-1|) r^{-\am}\left( \Upsilon_1(r,R,\alpha_k) + \Upsilon_2(r,R) \right)\Big]^{1/\beta} \times\\
& \quad \times \Big[ (1 \vee |q-1|) \left(\frac{r+R}{2}\right)^{-\am}\left( \Upsilon_3(r,R,\alpha_k)+ \Upsilon_4(r,R) \right) + r^{-\am}\Upsilon_1(r,R,\alpha_k) \Big] \times \\
& \quad \times \left(\int_{Q_{\ominus}(R)} v^2(s,x) \d x \d s\right)^{\kappa} \\
& \le c \left( (1 \vee |q-1|) (R-r)^{-\am}\right)^{1/\beta} \times \left[ (1 \vee |q-1|) (R-r)^{-\am} \right] \times\\
& \quad \times \left(\int_{Q_{\ominus}(R)} v^2(s,x) \d x \d s\right)^{\kappa},
\end{align*}
where we used $\Upsilon_1(r,R,\alpha_k),\Upsilon_2(r,R) \leq 4d((R-r)/r)^{-\am}$ and $\Upsilon_3(r,R,\alpha_k), \Upsilon_4(r,R) \leq d((R-r)/(R+r))^{-\am}$ in the last step.
We set $q = p+1$. Since, $q \ge 1$, we have $1 \vee |q-1| \le p+1$ and therefore
\begin{align*}
\left(\int_{Q_{\ominus}(r)} v^{2\kappa}(t,x) \d x \d t \right)^{1/\kappa} &\le c (p+1) \left( R - r \right)^{-\am}\left(\int_{Q_{\ominus}(R)} v^2(s,x) \d x \d s\right).
\end{align*}
Using the definition of $v$ finishes the proof.
\end{proof}

The next result can be interpreted as a lower bound for the infimum of a non-negative supersolution $u$ to $\partial_t u - L u = f$ and is obtained by iterating the result from \autoref{thmit1}.

\begin{corollary}
\label{thmmoser1}
Let $1/2 \le r < R \le 1$ and $p \in (0,1]$. There is $c = c(d,\alpha_0,\Lambda) > 0$ such that for every $\mu \in \cK(\alpha_0,\Lambda)$ and every non-negative supersolution $u$ of $\partial_t u - L u = f$ in $Q = I \times \Omega$, with $Q_{\ominus}(R) \Subset Q$ and $u \ge \eps > 0$ in $Q$, it holds
\begin{align}
\label{moser1}
\sup_{Q_{\ominus}(r)} \U^{-1} \le \left( \frac{c}{(R-r)^{d \am /\alpha_0+\am}}\right)^{1/p} \Vert \U^{-1} \Vert_{L^p(Q_{\ominus}(R))},
\end{align}
where $\U = u + \Vert f \Vert_{L^{\infty}(Q)}$.
\end{corollary}

\begin{proof}
By \autoref{thmit1}
\begin{align}\label{itineq}
\Vert \U^{-1} \Vert_{L^{\kappa p}(Q_{\ominus}(r))} \le A^{1/p} \Vert \U^{-1} \Vert_{L^{p}(Q_{\ominus}(R))},
\end{align}
where $\kappa = 1 + \frac{1}{\beta}$ and $A = c (p+1) \left( R - r \right)^{-\am}$.\\
For $j \in \N_0$, let $r_j = r + \frac{R-r}{2^{j}}$ and $p_j = p \kappa^j$. 
Note that $(r_j)_{j\in\N_0}$ is a decreasing sequence with $r_0=R$ and $r_j \searrow r$ as $j\to\infty$. Furthermore, $(p_j)_{j\in\N_0}$ is an increasing sequence with $p_0=p$ and $p_j \nearrow +\infty$ as $j\to\infty$.
Iterating inequality \eqref{itineq}, we obtain for each $m \in \N_0$
\begin{align}
\label{helpmoser1}
\begin{aligned}
\Vert \U^{-1} \Vert^p_{L^{p_{m+1}}(Q_{\ominus}(r))} &\le \Vert \U^{-1} \Vert^p_{L^{p_{m+1}}(Q_{\ominus}(r_{m+1}))} \le A_{m}^{p/{p_{m}}} \Vert \U^{-1} \Vert^p_{L^{p_{m}}(Q_{\ominus}(r_m))} \\
& \leq\dots\leq  \Vert \U^{-1} \Vert^p_{L^{p_{0}}(Q_{\ominus}(r_{0}))} \prod_{j = 0}^{m} A_j^{p/{p_{j}}} = \Vert \U^{-1} \Vert^p_{L^{p}(Q_{\ominus}(R))} \prod_{j = 0}^{m} A_j^{1/{\kappa^j}},
\end{aligned}
\end{align}
where $A_j := c(p \kappa^j + 1) \left( r_j - r_{j+1} \right)^{-\am}$. 
Note that $A_j \le c (2 \kappa^j) \left( \frac{2^{j+1}}{R-r} \right)^{\am}\le \frac{c_1^j}{(R-r)^{\am}}$ for some $c_1 = c_1(d,\alpha_0,\Lambda) > 0$, where we used that $0 < p \le 1$ and $\kappa > 1$. It holds
\begin{align*}
\sum_{j = 0}^{\infty} \frac{j}{\kappa^j} \le \sum_{j = 0}^{\infty} \frac{j}{(1+ \alpha_0/d)^j} \le c_2,
\end{align*}
for some $c_2 = c_2(d,\alpha_0) > 0$. Hence, by using $\beta \le d/\alpha_0$, we obtain 
\begin{align*}
\prod_{j = 0}^{\infty} A_j^{1/\kappa^j} & \le \prod_{j = 0}^{\infty} (R-r)^{-\am\kappa^{-j}}c_1^{j/\kappa^j} = (R-r)^{-\am \sum_{j = 0}^{\infty} \kappa^{-j}} c_1^{\sum_{j = 0}^{\infty} j/\kappa^j} \le \frac{c_3}{(R-r)^{\am(\beta+1)}} \\
& \le \frac{c_3}{(R-r)^{d\am/\alpha_0+\am}},
\end{align*}
for some $c_3 = c_3(d,\alpha_0,\Lambda) \ge 1$. Consequently, $\prod_{j = 0}^{\infty} A_j^{1/\kappa^j} \le \frac{c_3}{(R-r)^{d\am/\alpha_0+\am}}$ and by taking the limit $m \to \infty$ in \eqref{helpmoser1}, we get that
\begin{align*}
\sup_{Q_{\ominus}(r)} \U^{-1} &= \lim_{m \to \infty} \Vert \U^{-1} \Vert_{L^{p_{m+1}}(Q_{\ominus}(r))} \le \Vert \U^{-1} \Vert_{L^{p}(Q_{\ominus}(R))} \left(\prod_{j = 0}^{\infty} A_j^{1/{\kappa^j}} \right)^{1/p}\\
&\le \left(\frac{c_3}{(R-r)^{d\am/\alpha_0+\am}}\right)^{1/p} \Vert \U^{-1} \Vert_{L^p(Q_{\ominus}(R))}.
\end{align*}
\end{proof}

\subsection{Moser iteration for small positive exponents} \allowdisplaybreaks
In this subsection, we prove a Moser Iteration scheme for small positive exponents. 
The main result of this subsection can be proven by a similar technique as in the proof of \autoref{thmit1}. 
For this reason, we only provide a sketch of the proof.
\begin{theorem}[Moser iteration scheme for small positive exponents]
\label{thmit2}
Let $1/2 \le r < R \le 1$ and $p \in (0,\kappa^{-1}]$, where $\kappa = 1 + \frac{1}{\beta}$. There is $c = c(d,\alpha_0,\Lambda) > 0$ such that for every $\mu \in \cK(\alpha_0,\Lambda)$ and every non-negative supersolution $u$ of $\partial_t u - L u = f$ in $Q = I \times \Omega$, with $Q_{\oplus}(R) \Subset Q$ it holds
\begin{align}
\label{it2}
\Vert \U \Vert^p_{L^{\kappa p}(Q_{\oplus}(r))} \le c \left( R - r \right)^{-\am} \Vert \U \Vert^p_{L^{p}(Q_{\oplus}(R))},
\end{align}
where $\U = u + \Vert f \Vert_{L^{\infty}(Q)}$. 
\end{theorem}

\begin{proof}
In the case $\Vert f \Vert_{L^{\infty}(Q)} = 0$, we proceed as described in the proof of \autoref{thmit1}.
Let $q = 1-p \in [1- \kappa^{-1},1)$ and $v(t,x) = \U^{\frac{1-q}{2}}(t,x)$. We define $\phi(t,x) = \U^{-q}(t,x)\tau^2(x)$, where $\tau_{x_0 = 0, r = r, \lambda = R/r} =: \tau : \R^d \rightarrow [0,1]$ is an admissible cut-off function in the sense of \autoref{cutofffct}. 
%Note that the assumption $p \le \kappa^{-1}$ ensures that $q > 0$. 
%In fact, in the formulation of the theorem, we could have taken any upper bound for $p$ less than $1$.\\
By \eqref{weaksol2}, we obtain for a.e. $t \in I$
\begin{align}
\label{helpMoserPos1}
-\int_{M_R} \hspace{-0.2cm} \tau^{2}(x) \U^{-q}(t,x) \partial_t \U(t,x) \d x + \mathcal{E}^t(\U(t),-\tau^{2}\U^{-q}(t)) \le -\int_{M_R} \hspace{-0.2cm}\tau^{2}(x) \U^{-q}(t,x)f(t,x)\d x.
\end{align}
Using \eqref{eq:algebra-Moser-nl}, \eqref{eq:algebra-Moser-l}, $\partial_t (v^2) = (1-q)\U^{-q}\partial_t \U$, and \autoref{thmcutoffest}, we obtain by the same arguments as in the proof of \autoref{thmit1}:
\begin{align*}
-\int_{M_R} \hspace{-0.2cm} \tau^2(x) \partial_t (v^2)(t,x) \d x + \cE^t_{M_r}(v(t),v(t)) \le c r^{-\am} \sum_{k=1}^d \left( (R/r)^{\am/\alpha_k} - 1 \right)^{-\alpha_k} \int_{M_R} \hspace{-0.2cm} v^2(t,x) \d x,
\end{align*}
for some constant $c = c(d,\alpha_0,\Lambda) > 0$, where we used that $(1 \vee |q-1|) \le 1$.\\
Let $\chi_{\oplus} : \R \rightarrow [0,1]$ be defined by $\chi_{\oplus}(t) = \left( \frac{R^{\am}-t}{R^{\am} - r^{\am}} \wedge 1 \right) \vee 0$. 
We multiply the above inequality with $\chi_{\oplus}^2$ and integrate in $s$ from some $t \in I_{\oplus}(r)$ to $R^{\am}$. 
Similar to the proof of \autoref{thmit1}, we obtain
\begin{align*}
\sup_{t \in {I_{\oplus}(r)}}& \int_{M_r} v^2(t,x) \d x + c_1 \int_{I_{\oplus}(r)} \cE^s_{M_r}(v(s),v(s)) \d s \\
& \le c \left[r^{-\am} \left( \sum_{k=1}^d \left( (R/r)^{\am/\alpha_k} - 1 \right)^{-\alpha_k}+ ((R/r)^{\am} - 1)^{-1} \right)  \right]\int_{Q_{\oplus}(R)} v^2(s,x) \d x \d s.
\end{align*}
Proceeding as in the proof of \autoref{thmit1}, we obtain 
\begin{align*}
\left( \int_{Q_{\oplus}(r)} \U(t,x) \d x \d t \right)^{1/\kappa} \le c(R-r)^{-\am} \int_{Q_{\oplus}(R)} \U^{p}(t,x) \d x \d t,
\end{align*}
which finishes the proof.
\end{proof}

By an iteration argument, similar to the one we used in order to obtain \autoref{thmmoser1}, we now derive an estimate for small positive moments of non-negative supersolutions $u$ to $\partial_t u - Lu = f$.

\begin{corollary}
\label{thmmoser2}
Let $1/2 \le r < R \le 1$, $\kappa=1+1/\beta$ and $p \in (0, \kappa^{-1})$. There is $c = c(d,\alpha_0,\Lambda) > 0$ such that for every $\mu \in \cK(\alpha_0,\Lambda)$ and every non-negative supersolution $u$ of $\partial_t u - L u = f$ in $Q = I \times \Omega$, with $Q_{\oplus}(R) \Subset Q$, it holds that
\begin{align}
\label{moser2}
\int_{Q_{\oplus}(r)} \hspace{-0.2cm} \U(t,x) \d x \d t \le \left( \frac{c}{\vert Q_{\oplus}(1)\vert (R-r)^{2d\am/\alpha_0 + 2\am + \am^2/\alpha_0 + \am^2/d}}\right)^{1/p - 1} \hspace{-0.2cm} \Vert \U\Vert_{L^p(Q_{\oplus}(R))},
\end{align}
where $\U = u + \Vert f \Vert_{L^{\infty}(Q)}$. 
\end{corollary}

\begin{proof}
This proof is along the lines of \cite[Theorem 3.7]{FeKa13}. 
\autoref{thmit2} implies that
\begin{align}
\label{helpmoser2}
\Vert \U\Vert_{L^{\kappa p}(Q_{\oplus}(r))} \le A^{1/p} \Vert \U\Vert_{L^{p}(Q_{\oplus}(R))},
\end{align}
where $A = c \left( R - r \right)^{-\am}$.
Let $p \in (0,\kappa^{-1})$ be fixed. 
Similar to the proof of \autoref{thmmoser1} we set $p_j = \kappa^{-j}$ and $r_j = r + \frac{R-r}{2^j}$ for $1 \le j \le n$, where $n \in \N$ is such that $p_n \le p < p_{n-1}$. Note that $ R = r_0 \ge r_j > r_{j+1} > \dots \ge r_n > r$ and set $A_j = c 2^{\am j}(r_{j-1}-r_j)^{-\am} \le c4^{j}(R-r)^{-\am}$.\\
Note that
\begin{align*}
 \sum_{j=1}^n \kappa^j &= \frac{\kappa}{\kappa-1}\left(\frac{1}{p_n} - 1\right)
= (\beta + 1) \left(\frac{1}{p_n} - 1\right) \hspace*{0.5em}  \text{and} \hspace*{0.5em}
\sum_{j=1}^n (n-j+1) \kappa^j &\le \frac{\kappa^3}{(\kappa-1)^3} \left(
\frac{1}{p_n} -1\right).
\end{align*} 
By iterating \eqref{helpmoser2} and using the above line in the final step, we obtain the following estimate
\begin{align}
\label{helpmoser22}
\begin{aligned}
\Vert \U\Vert_{L^{1}(Q_{\oplus}(r))} &\le \Vert \U\Vert_{L^{1}(Q_{\oplus}(r_n))} = \Vert \U\Vert_{L^{\kappa p_1}(Q_{\oplus}(r_n))} \le A_n^{1/{p_1}} \Vert \U\Vert_{L^{p_1}(Q_{\oplus}(r_{n-1}))}\\
& = A_n^{1/{p_1}} \Vert \U\Vert_{L^{\kappa p_2}(Q_{\oplus}(r_{n-1}))}  \le \dots \le \prod_{j = 1}^{n} A_{n-j+1}^{1/{p_j}}\Vert \U\Vert_{L^{p_n}(Q_{\oplus}(r_0))} \\
& = \prod_{j = 1}^{n} A_{n-j+1}^{\kappa^j}\Vert \U\Vert_{L^{p_n}(Q_{\oplus}(R))} \le \prod_{j = 1}^{n} \left( c4^{n-j+1}(R-r)^{-\am} \right)^{\kappa^j}\Vert \U\Vert_{L^{p_n}(Q_{\oplus}(R))} \\
& = (c(R-r)^{-\am})^{\sum_{k = 1}^j \kappa^j} 4^{{\sum_{k = 1}^j (n-j+1)\kappa^j}}\Vert \U\Vert_{L^{p_n}(Q_{\oplus}(R))} \\
&\le \left((c(R-r)^{-\am})^{\frac{\kappa}{\kappa-1}} 4^{\frac{\kappa^3}{(\kappa-1)^3}} \right)^{\frac{1}{p_n}-1}\Vert \U\Vert_{L^{p_n}(Q_{\oplus}(R))}.
\end{aligned}
\end{align}
We can further estimate \eqref{helpmoser22} by using
\begin{align*}
\frac{1}{p_n} - 1 &= \kappa^n-1 \leq \kappa^n + \kappa^{n-1} - \kappa -1 =
(1+\kappa)(\kappa^{n-1}-1) 
%= (1+\kappa)\left( \frac{1}{p_{n-1}-1} \right)
%\nonumber \\
 \le (1+\kappa) \left( \frac{1}{p}-1\right),
\end{align*}
and the Hölder inequality as follows:
\begin{align*}
\Vert \U\Vert_{L^{1}(Q_{\oplus}(r))} &\le \left((c(R-r)^{-\am})^{\frac{\kappa}{\kappa-1}} 4^{\frac{\kappa^3}{(\kappa-1)^3}} \right)^{\frac{1}{p_n}-1}\Vert \U\Vert_{L^{p_n}(Q_{\oplus}(R))} \\
&\le \vert Q_{\oplus}(R)\vert^{\frac{1}{p_n}-\frac{1}{p}} \left((c(R-r)^{-\am})^{\frac{\kappa}{\kappa-1}} 4^{\frac{\kappa^3}{(\kappa-1)^3}} \right)^{\frac{1}{p_n}-1}\Vert \U\Vert_{L^{p}(Q_{\oplus}(R))}\\
&\le \vert Q_{\oplus}(1)\vert^{\frac{1}{p_n}-1} \left((c(R-r)^{-\am})^{\frac{\kappa}{\kappa-1}} 4^{\frac{\kappa^3}{(\kappa-1)^3}} \right)^{\frac{1}{p_n}-1}\Vert \U\Vert_{L^{p}(Q_{\oplus}(R))}\\
& \le \left(\vert Q_{\oplus}(1)\vert (c(R-r)^{-\am})^{\frac{\kappa}{\kappa-1}} 4^{\frac{\kappa^3}{(\kappa-1)^3}} \right)^{(1+\kappa)\left(\frac{1}{p}-1 \right)}\Vert \U\Vert_{L^{p}(Q_{\oplus}(R))}\\
& \le \left( \frac{c_1 \vert Q_{\oplus}(1)\vert}{(R-r)^{d\am/\alpha_0+\am}}\right)^{(1+\kappa)\left(\frac{1}{p} - 1\right)} \Vert \U\Vert_{L^p(Q_{\oplus}(R))},
\end{align*}
for some $c_1 = c_1(d,\alpha_0,\Lambda) \ge 1$, where we used in the third step that $\vert Q_{\oplus}(R) \vert \le \vert Q_{\oplus}(1) \vert$ and in the last step that $\frac{\kappa}{\kappa -1} = \beta + 1 \le \frac{d}{\alpha_0} + 1$.\\
Since  $0 < R-r < 1$ and $(1 + \kappa)(d\am/\alpha_0 +\am) \le (2 + \am/d)(d\am/\alpha_0 +\am) = 2d\am/\alpha_0 + 2\am + \am^2/\alpha_0 + \am^2/d$,
\begin{align*}
\int_{Q_{\oplus}(r)} \U(t,x) \d x \d t &\le \left( \frac{\left(c_1 \vert Q_{\oplus}(1)\vert\right)^{2+\am/d}}{(R-r)^{2d\am/\alpha_0 + 2\am + \am^2/\alpha_0 + \am^2/d}}\right)^{\frac{1}{p} - 1} \Vert \U\Vert_{L^p(Q_{\oplus}(R))}\\
& = \left( \frac{c_2}{\vert Q_{\oplus}(1)\vert (R-r)^{2d\am/\alpha_0 + 2\am + \am^2/\alpha_0 + \am^2/d}}\right)^{\frac{1}{p} - 1} \Vert \U\Vert_{L^p(Q_{\oplus}(R))}
\end{align*}
for some $c_2 = c_2(d,\alpha_0,\Lambda) > 0$, as desired. 
%Note that in the last step we used that $\vert Q_{\oplus}(1) \vert = 2^d$ and therefore only depends on $d$.
\end{proof}

\section{Weak Harnack inequality}\label{sectionHI}

The goal of this section is to establish the weak Harnack inequality for non-negative supersolutions $u$ to $\partial_t u - L u = f$. As we already pointed out, this result is a consequence of \autoref{thmBG}. However, the hard work has to be put into verifying the assumptions of \autoref{thmBG} in order to apply it in a beneficial way to our situation. 
This verification has already been done in the previous sections, which is why the subsequent proof is relatively short.

\begin{proof}[Proof of \autoref{thm:weakHarnack}]
Set $\U = u + \Vert f \Vert_{L^{\infty}(Q)}$. In the case $\Vert f \Vert_{L^{\infty}(Q)} = 0$, set $\U = u + \eps$ for some $\eps > 0$ and take the limit $\eps \searrow 0$ at the end of the proof. 
Note that this procedure is necessary in order to apply \autoref{thmmoser1} and \autoref{thmlogu}. 
We define $w = e^{-a} \U^{-1}$ and $\hat{w} = w^{-1} = e^a \U$, where $a = a(\U) \in \R$ is chosen such that \eqref{logu1} is satisfied, i.e. there is $c_1 > 0$ such that
\begin{align}
\label{harnackhelp1}
\vert Q_{\oplus}(1) \cap \lbrace \log w > s \rbrace \vert \le c_1 \frac{\vert M_1 \vert}{s}, \qquad \vert Q_{\ominus}(1) \cap \lbrace \log \hat{w} > s \rbrace \vert \le c_1 \frac{\vert M_1 \vert}{s} ~~ \forall s > 0.
\end{align}
Let $\cU = (U(r))_{1/2 \le r \le 1}$ and ${\hat{\cU}} = (\hat{U}(r))_{1/2 \le r \le 1}$ be two families of domains, defined by $U(r) = (1 - r^{\am},1) \times M_r$ and $\hat{U}(r) = (-1,-1+r^{\am}) \times M_r$. Note, that $U(1) = Q_{\oplus}(1), \hat{U}(1) = Q_{\ominus}(1)$ and $U(1/2) = U_{\oplus}, \hat{U}(1/2) = U_{\ominus}$.\\
We want to apply \autoref{thmBG} once to $(w,\cU)$ with $p_0 = \infty$ and arbitrary $\eta$ and once to $(\hat{w},\hat{\cU})$ with $p_0 = 1$ and $\eta = \frac{d}{d+\alpha_0} \le \frac{1}{\kappa}$. By \eqref{harnackhelp1}, the first condition \eqref{BGAss1} of \autoref{thmBG} is satisfied in both situations, whereas the second condition \eqref{BGAss2} follows from an application of \autoref{thmmoser1} to $(w,\cU)$ and 
an application of \autoref{thmmoser2} to $(\hat{w},\hat{\cU})$, respectively after a shift in time. All in all, \autoref{thmBG} implies 
\begin{align*}
\sup_{U(1/2)} w =e^{-a} \sup_{U(1/2)} \U^{-1} \le C, \qquad \Vert \hat{w} \Vert_{L^1(\hat{U}(1/2))} \le e^a \Vert \U \Vert_{L^1(\hat{U}(1/2))} \le \hat{C},
\end{align*}
for some $C=C(d,\alpha_0,\Lambda), \hat{C}=\hat{C}(d,\alpha_0,\Lambda) > 0$. Hence, by multiplying these two inequalities with each other, we obtain
\begin{align*}
\Vert \U \Vert_{L^1(\hat{U}(1/2))} \le  C\hat{C} \left( \sup_{U(1/2)} \U^{-1} \right)^{-1}  = C\hat{C} \inf_{U(1/2)} \U,
\end{align*}
and therefore
\begin{align*}
\Vert u \Vert_{L^1(U_{\ominus})} \le \Vert \U \Vert_{L^1(U_{\ominus})} \le c \inf_{U_{\oplus}} \U \le c \left( \inf_{U_{\oplus}} u + \Vert f \Vert_{L^{\infty}(Q)} \right),
\end{align*}
which finishes the proof with $c = C\hat{C}$.
\end{proof}

\section{Hölder regularity}\label{sectionHR}

In this section we derive a Hölder regularity estimate for solutions $u$ to $\partial_t u -L u = 0$. The main ingredient is the weak Harnack inequality that we have derived in the previous section. Indeed, the weak Harnack inequality implies a decay of oscillation estimate (see \autoref{thmoscdecay}), which in turn directly implies \autoref{thm:HR}, the Hölder regularity estimate. 
%The procedure that we use in the following has already been applied in \cite{FeKa13} and can be imitated almost entirely.\\

Let us now define the domains we are going to work within this section. We set $D_{\ominus} = (-2,-2 + 1/2^{\am}) \times M_{1/2}$ and $D_{\oplus} = (-1/2^{\am},0) \times M_{1/2}$. Furthermore, for $(t,x) \in \R^{d+1}$ we define a distance function $\hat{\rho}$ by 
\begin{align}\label{def:hatrho}
\hat{\rho}((t,x)) = \max \left(\frac{1}{2}(-t)^{1/\am} , \frac{1}{3}\sup_{k = 1, \dots, d} \vert x_k \vert^{\alpha_k / \am} \right),
\end{align}
if $t \in (-2,0]$ and $\hat{\rho}((t,x)) = \infty$ otherwise.
For $r > 0$, let
\begin{align*}
\hat{D}_r((t_0,x_0)) = (t_0 - 2r^{\am},t_0) \times M_{3r}(x_0),
\end{align*}
i.e. the set of all points $(t,x)$ whose distance to $(t_0,x_0)$ with respect to $\hat{\rho}$ is smaller than $r$. 
We set $\hat{D}_r((0,0)) = \hat{D}_r$. One can easily observe that $\bigcup_{r>0} \hat{D}_r = (-2,0) \times \R^d$.\\
Moreover, let
\begin{align*}
D(r) = (-2r^{\am},0) \times M_{2r}.
\end{align*}

As a first step, we establish the following consequence of \autoref{thm:weakHarnack}. 

\begin{corollary}
\label{thmCorHarnack}
Let $\sigma \in (0,1)$. There exist $\eps_0 = \eps_0(d,\alpha_0,\Lambda,\sigma) \in (0,1)$ and $\delta = \delta(d,\alpha_0,\Lambda,\sigma) \in (0,1)$ such that for every $\mu \in \cK(\alpha_0,\Lambda)$ and every function $w$ satisfying
\begin{itemize}
\item $w \ge 0$ a.e. in $(-2,0) \times \R^d$,
\item $\partial_t w - L w \ge -\eps_0$ in $D(1)$,
\item $\vert D_{\ominus} \cap \lbrace w \ge 1 \rbrace \vert \ge \sigma \vert D_{\ominus} \vert$,
\end{itemize}

it holds that $w \ge \delta$ a.e. in $D_{\oplus}$.
\end{corollary}

\begin{proof}
Applying \autoref{thm:weakHarnack} to $w$, leads to
\begin{align*}
\sigma \le \frac{\vert D_{\ominus} \cap \lbrace w \ge 1 \rbrace \vert}{\vert D_{\ominus} \vert} \le \frac{1}{\vert D_{\ominus} \vert}\int_{D_{\ominus}} w(t,x) \d x \d t \le c \left( \inf_{D_{\oplus}} w + \eps_0 \right),
\end{align*}
where $c = c(d,\alpha_0, \Lambda) > 0$.
%, since $\vert D_{\ominus} \vert = \left(\frac{1}{2}\right)^{\am + \am \beta} \le \left(\frac{1}{2}\right)^{\alpha_0 + \alpha_0 d/2}$ depends only on $d,\alpha_0$. 
Note that we had to perform a shift in time in order to be able to apply \autoref{thm:weakHarnack} in the foregoing argument. 
Choosing $\eps_0 < \sigma/c$ and setting $\delta = (\sigma -c \eps_0)/c > 0$ proves the assertion.
\end{proof}

The following result is a consequence of \autoref{thmCorHarnack}.

\begin{lemma}
\label{thmHRhelp}
There exist $\gamma_0 = \gamma_0(d,\alpha_0,\Lambda,\theta) \in (0,1)$ and $\delta = \delta(d,\alpha_0,\Lambda) \in (0,1)$ such that for every $\mu \in \cK(\alpha_0,\Lambda)$ and every function $w$ satisfying
\begin{itemize}
\item $w \ge 0$ a.e. in $\hat{D}(1)$,
\item $\partial_t w - Lw \ge 0$ in $\hat{D}(1)$,
\item $\vert D_{\ominus} \cap \lbrace w \ge 1 \rbrace \vert \ge \frac{1}{2} \vert D_{\ominus} \vert$,
\item $w \ge 2\left( 1 - (6\hat{\rho}((t,y)))^{\gamma_0} \right)$ a.e. in $(-2,0) \times (\R^d \setminus M_3)$,
\end{itemize}
it holds that $w \ge \delta$ a.e. in $D_{\oplus}$.
\end{lemma}

\begin{proof}
Since $w \ge 0$ in $\hat{D}(1)$, $w$ and $w^+$ coincide on $\hat{D}(1)$. Hence, it holds that $\partial_t w^+ - L w^+ \ge -f$ in $D(1)$, where $f(t,x) = (L w^-)(t,x)$. Note that
\begin{align*}
\Vert f \Vert_{L^{\infty}(D(1))} = \sup_{(t,x) \in D(1)} \int_{\R^d \setminus M_3} w^-(t,y) \mu(t,x,\d y),
\end{align*}
by the non-negativity of $w$ in $\hat{D}(1)$. Our goal is to find $\gamma_0 \in (0,1)$ such that $\Vert f \Vert_{L^{\infty}(D(1))} < \eps_0$, where $\eps_0 \in (0,1)$ is the constant from \autoref{thmCorHarnack} applied to $w^+$ with $\sigma = 1/2$. Then the assertion follows from \autoref{thmCorHarnack}.\\
We decompose $\R^d \setminus M_3 = \bigcup_{j = 1}^{\infty} (M_{3^{j+1}} \setminus M_{3^j})$ and note that $(\R^d \setminus M_{3^j}) \subset (\R^d \setminus M_{3^{j-1}}(x))$ for any $x \in M_2$ and $j \in \N$. Upon the observation that by assumption it holds that $w^-(t,y) \le 2(6^{\gamma_0}\cdot 3^{j \gamma_0} -  1)$ for $(t,y) \in (-2,0) \times (\R^d \setminus M_{3^j})$, we have
\begin{equation*}
\begin{split}
\int_{\R^d \setminus M_3} w^-(t,y) \mu(t,x,\d y) &= \sum_{j=1}^{\infty} \int_{M_{3^{j+1}} \setminus M_{3^j}} w^{-}(t,y) \mu(t,x,\d y)\\
&\le 2\sum_{j=1}^{\infty} (6^{\gamma_0}\cdot 3^{j \gamma_0} -  1) \mu(t,x,\R^d \setminus M_{3^{j-1}}(x))\\
&\le c_1 \sum_{j=1}^{\infty} (6^{\gamma_0}\cdot 3^{j \gamma_0} -  1)3^{-j \theta},
\end{split}
\end{equation*}
for every ${t,x} \in D(1)$, where $c_1 = c_1(d,\alpha_0,\Lambda) > 0$ is a constant and we applied \eqref{assmu1} respectively \autoref{assumption:tail}. The quantity $c_1\sum_{j=1}^{\infty} (6^{\gamma_0}\cdot 3^{j \gamma_0} -  1)3^{-j \theta}$ can be bounded from above by $\eps_0$ as follows. First, one finds $N \in \N$ large enough such that $6^{\gamma_0}c_1\sum_{j=N+1}^{\infty} 3^{-(\theta-\theta/2)j} \le \eps_0/2$. Then one chooses $\gamma_0 \in (0, \theta/2)$ small enough such that also $c_1\sum_{j=1}^{N} (6^{\gamma_0}\cdot 3^{j \gamma_0} -  1)3^{-j \theta} < \eps_0/2$, which is possible since $(6^{\gamma_0}\cdot 3^{j \gamma_0} -  1) \to 0$ as $\gamma_0 \to 0$ for every $1 \le j \le N$. This ends the proof.
\end{proof}

\begin{theorem}[Decay of oscillation]
\label{thmoscdecay}
There exists $\gamma = \gamma(d,\alpha_0,\Lambda,\theta) \in (0,1)$ such that for every $\mu \in \cK(\alpha_0,\Lambda)$ and every solution $u$ to $\partial_t u -Lu = 0$ in $\hat{D}(1)$ it holds that for every $\nu \in \Z$:
\begin{align}
\label{oscdecay}
\osc_{\hat{D}(6^{-\nu})} u := \sup_{\hat{D}(6^{-\nu})} u - \inf_{\hat{D}(6^{-\nu})} u \le 2 \Vert u \Vert_{L^{\infty}((-2,0) \times \R^d)} 6^{-\nu \gamma}.
\end{align}
\end{theorem}

\begin{proof}
Let $S_0 = \sup_{(-2,0) \times \R^d} u$, $s_0 = \inf_{(-2,0) \times \R^d} u$ and define $K = S_0-s_0$. Furthermore, set
\begin{align*}
\gamma := \min \left(\gamma_0, \frac{\log(2/(2-\delta))}{\log(6)} \right),
\end{align*}
where $\delta,\gamma_0 \in (0,1)$ are the constants from \autoref{thmHRhelp}. Note that by its definition, $\gamma$ satisfies $1 - \delta/2 < 6^{-\gamma}$. 
Our aim is to construct an increasing sequence $(s_{\nu})_{\nu \in \Z}$ and a decreasing sequence $(S_{\nu})_{\nu \in \Z}$ such that for each $\nu \in \Z$
\begin{align}
\label{Mm}
s_{\nu} \le u \le S_{\nu} \text{ a.e. in } \hat{D}(6^{-\nu}), \qquad S_{\nu} - s_{\nu} = K 6^{-\nu \gamma}.
\end{align}
The existence of such sequences easily implies the desired result \eqref{oscdecay}.  \\
We construct such sequences by induction. First, we define for $n \in \N$ the elements $S_{-n}, s_{-n}$ by $S_{-n} = S_0$ and $s_{-n} = s_0$. Assume there is $k\in\N$ and there are $S_n $,$s_n$ such that \eqref{Mm} is true for $n \le k-1$. We define the auxiliary function
\begin{align*}
v(t,x) = \left[ u\left( \frac{t}{6^{\am(k-1)}}, \left(\frac{x_k}{(6^{k-1})^{\am/\alpha_k}}\right)_{k = 1}^d \right) - \frac{S_{k-1} + s_{k-1}}{2} \right] \frac{2 \cdot 6^{\gamma (k-1)}}{K},
\end{align*}
and observe by scaling (see \autoref{thmscaling}) that $v$ satisfies $\partial_t v - \tilde{L}v = 0$ in $\hat{D}(1)$ and that $\vert v \vert \le 1$ in $\hat{D}(1)$ by the induction hypothesis. Here, $\tilde{L}$ corresponds to an energy form $\tilde{\cE} := \cE^{\tilde{\mu}}$, where $\tilde{\mu} \in \cK(\alpha_0,\Lambda)$ is defined as in \autoref{thmscaling}. Therefore, we can apply the results we have established so far also to $v$.\\
We estimate $v$ on $(-2,0) \times (\R^d \setminus M_3)$ as follows: For $(t,y) \in (-2,0) \times (\R^d \setminus M_3)$, we fix some 
$j \in \N$ such that $6^{j-1} \le \hat{\rho}(t,y) \le 6^j$, i.e. $(t,y) \in \hat{D}(6^j) \setminus \hat{D}(6^{j-1})$. Then
\begin{align*}
\frac{K}{2 \cdot 6^{\gamma(k-1)}} v(t,y) &= u\left( \frac{t}{6^{\am(k-1)}},\left(\frac{y_k}{\left(6^{k-1}\right)^{\am/\alpha_k}}\right)_{k = 1}^d \right) - \frac{S_{k-1}+s_{k-1}}{2}\\
&\le S_{k-j-1} - s_{k-j-1} + s_{k-j-1} - \frac{S_{k-1}+s_{k-1}}{2}\\
&\le S_{k-j-1} - s_{k-j-1} - \frac{S_{k-1}-s_{k-1}}{2}\\
&\le K 6^{-\gamma(k-j-1)} - \frac{K}{2} 6^{-\gamma(k-1)}.
\end{align*}
Since $(t,y) \in \hat{D}(6^j) \setminus \hat{D}(6^{j-1})$, we have $(\tilde{t},\tilde{y}) := \left(\frac{t}{6^{\am(k-1)}},\left(\frac{x_k}{(6^{k-1})^{\am/\alpha_k}}\right)_{k=1}^d\right) \in \hat{D}(6^{j-(k-1)})$.
Consequently, $v(t,y) \le 2\cdot 6^{j\gamma} - 1$ for a.e. $(t,y) \in \hat{D}(6^j) \setminus \hat{D}(6^{j-1})$ and therefore,
\begin{align*}
v(t,y) \le 2 \left( 6 \hat{\rho}(t,y) \right)^{\gamma} - 1
\end{align*}
for a.e. $(t,y) \in (-2,0) \times (\R^d \setminus M_3)$ by the definition of $\hat{\rho}$. Analogously,
\begin{align*}
v(t,y) \ge 1 - 2\left(6 \hat{\rho}(t,y) \right)^{\gamma}
\end{align*}
for a.e. $(t,y) \in (-2,0) \times (\R^d \setminus M_3)$. \\
The existence of $s_k$ and $S_k$ will now follow from \autoref{thmHRhelp} for whose applicability the foregoing two lines are important. 
We distinguish between two cases. 
First, we assume that
\begin{align*}
\vert D_{\ominus} \cap \lbrace v \le 0 \rbrace \vert \ge \frac{1}{2} \vert D_{\ominus} \vert.
\end{align*}
In this case, we can apply \autoref{thmHRhelp} to $w := 1-v$ and therefore $w \ge \delta$ a.e. in $D_{\oplus}$, i.e.
\begin{align}
\label{vest}
v \le 1 - \delta \quad \text{ a.e. in } D_{\oplus}.
\end{align}
Since $\hat{D}(6^{-1}) \subset D_{\oplus}$, we obtain that for a.e. $(t,x) \in \hat{D}(6^{-k})$, using \eqref{vest}, \eqref{Mm} and $1 - \delta/2 \le 6^{-\gamma}$,
\begin{align*}
u(t,x) &=  \frac{K}{2 \cdot 6^{\gamma(k-1)}}\ v\left(6^{\am (k-1)} t ,\left(\left(6^{k-1}\right)^{\am/\alpha_k} x_k \right)_{k = 1}^d \right) + \frac{S_{k-1}+s_{k-1}}{2} \\
&\le \frac{K (1-\delta)}{2 \cdot 6^{\gamma(k-1)}} + s_{k-1} + \frac{S_{k-1}-s_{k-1}}{2} \\
&= \frac{K (1-\delta)}{2 \cdot 6^{\gamma(k-1)}} + s_{k-1} + \frac{K}{2 \cdot 6^{\gamma(k-1)}} = s_{k-1} + \left(1-\frac{\delta}{2} \right) K 6^{-\gamma(k-1)} \\
&\le s_{k-1} + K 6^{-\gamma k}. 
\end{align*}
If we now set $s_k = s_{k-1}$ and $S_k = s_{k-1} + K 6^{-\gamma k}$, then by the induction hypothesis and the previous observation, \eqref{Mm} is satisfied for $k$. 

In the second case, let $\vert D_{\ominus} \cap \lbrace v > 0 \rbrace \vert \ge \frac{1}{2} \vert D_{\ominus} \vert$. Here, we can apply \autoref{thmHRhelp} to $w := 1+v$ and obtain that $v \ge \delta - 1$ a.e. in $D_{\oplus}$. Using a similar argumentation, shows that we can choose $s_k = S_{k-1} - K 6^{-\gamma k}$ and $S_k = S_{k-1}$ such that \eqref{Mm} is satisfied. This finishes the proof.
\end{proof}

Finally, we have all necessary tools to prove our main result concerning Hölder regularity estimates. 
Note that \autoref{thmoscdecay} captures the anisotropy of the jump intensities whereas the following result provides an isotropic Hölder regularity estimate.

Recall that in the formulation of the theorem, $I$ denotes an open, bounded interval in $\R$ and $\Omega$ is a bounded domain in $\R^d$ such that $\Omega \subset M_r(x_0)$ for some $x_0 \in \R^d$ and $r \in (0,3]$.

\begin{proof}[Proof of \autoref{thm:HR}]
Let $\eta_0 = \eta_0(Q,Q') = \sup \left\lbrace r \in (0,1/2] ~\vert~ \forall (t,x) \in Q' : \hat{D}_r((t,x)) \subset Q \right\rbrace$\\
and fix $(t,x),(s,y) \in Q'$ with $t \le s$.
We distinguish between the two cases $\hat{\rho}((t,x)-(s,y)) < \eta_0$ and $\hat{\rho}((t,x)-(s,y)) \geq \eta_0.$

If $\hat{\rho}((t,x)-(s,y)) < \eta_0$, we can find $n \in \N_0$ such that $\frac{\eta_0}{6^{n+1}} \le \hat{\rho}((t,x)-(s,y)) < \frac{\eta_0}{6^n}$.\\
By scaling (see \autoref{thmscaling}), it can be seen that
\begin{align*}
\U(t,x) = u\left( \eta_0^{\am} t +s , \left( \eta_0^{\am/\alpha_k} x_k + y_k \right)_{k = 1}^d \right)
\end{align*}
satisfies $\partial_t \U - \tilde{L} \U = 0$ in $\hat{D}(1)$. Here $\tilde{L}$ is associated to an energy form $\tilde{\cE} := \cE^{\tilde{\mu}}$, where $\tilde{\mu} \in \cK(\alpha_0,\Lambda)$ is given as in \autoref{thmscaling}. Note that $\eta_0 \le 1/2$ is sufficient for $\tilde{\mu} \in \cK(\alpha_0,\Lambda)$.\\
Our aim is to prove
\begin{align}\label{eq:distineq}
\left(\frac{\hat{\rho}((t,x)-(s,y))}{\eta_0} \right)^{\gamma} \leq \left(\frac{\vert x-y\vert + \vert t-s \vert^{1/\am}}{c(\eta_0,\alpha_0)} \right)^{\gamma \alpha_0 / 2},
\end{align}
which then implies the result of the theorem in the first case. Let us first show how the Hölder regularity estimate follows from \eqref{eq:distineq}.
By \autoref{thmoscdecay},
\begin{align*}
%\label{regularityhelp1}
\vert u(t,x) - u(s,y)\vert &= \left\vert \U(\eta_0^{-\am}(t-s), \left( \eta_0^{-\am/\alpha_k}(x_k-y_k)\right)_{k = 1}^d) - \U(0,0) \right\vert\\
%&\le 2 \Vert \U \Vert_{L^{\infty}((-2,0) \times \R^d)} 6^{-n \gamma} \\
&\le 2 \Vert u \Vert_{L^{\infty}(I \times \R^d)} (6^{-n-1})^{\gamma}6^{\gamma}\\
&\le 12 \Vert u \Vert_{L^{\infty}(I \times \R^d)} \left(\frac{\hat{\rho}((t,x)-(s,y))}{\eta_0} \right)^{\gamma}\\
& \le 12 \Vert u \Vert_{L^{\infty}(I \times \R^d)} \left(\frac{\vert x-y\vert + \vert t-s \vert^{1/\am}}{c(\eta_0,\alpha_0)} \right)^{\gamma \alpha_0 / 2},
\end{align*}
for some constant $c(\eta_0,\alpha_0) > 0$, where we used \eqref{eq:distineq} in the last inequality. 
Hence, it remains to prove \eqref{eq:distineq}. 
Recall that by definition of $\hat{\rho}$ (see \eqref{def:hatrho}), either $\hat{\rho}((t,x)-(s,y)) = \frac{1}{2}\vert t-s\vert^{1/\am}$ or $\hat{\rho}((t,x)-(s,y)) = \frac{1}{3} \sup_{k = 1, \dots, d} \vert x_k-y_k\vert^{\alpha_k / \am}$. \\
If $\hat{\rho}((t,x)-(s,y)) = \frac{1}{2}\vert t-s\vert^{1/\am}$, we obtain by using $\hat{\rho}((t,x)-(s,y)) < \eta_0 \le \frac{1}{2}$, 
\begin{align*}
\hat{\rho}((t,x)-(s,y))^{\gamma} \le \left(\vert t-s\vert^{1/\am}\right)^{\gamma \alpha_0 /2} \le \left( \vert x-y\vert + \vert t-s \vert^{1/\am} \right)^{\gamma \alpha_0 /2}.
\end{align*}
If $\hat{\rho}((t,x)-(s,y)) = \frac{1}{3} \sup_{k = 1, \dots, d} \vert x_k-y_k\vert^{\alpha_k / \am}$, we get by using $\hat{\rho}((t,x)-(s,y)) < \eta_0 \le \frac{1}{2}$,
\begin{align*}
\hat{\rho}((t,x)-(s,y))^{\gamma} &\le \left( \sup_{k = 1, \dots, d} (1/3)^{\alpha_0/\alpha_k} \vert x_k - y_k \vert^{\alpha_0 / 2} \right)^{\gamma} \le \left( (1/3)^{\alpha_0/2} \vert x - y \vert^{\alpha_0/2}\right)^{\gamma}\\
& \le \left( (1/3)\left( \vert x-y \vert + \vert t-s \vert^{1/\am} \right) \right)^{\gamma \alpha_0/2}.
\end{align*}
This proves \eqref{eq:distineq} and thus the proof of the theorem is finished in the case $\hat{\rho}((t,x)-(s,y)) < \eta_0$.

In the remaining case $\hat{\rho}((t,x)-(s,y)) \ge \eta_0$, the assertion is an immediate consequence of the following computation
\begin{align*}
\vert u(t,x) - u(t,y)\vert \le 2 \Vert u \Vert_{L^{\infty}(I \times \R^d)} \le  2 \Vert u \Vert_{L^{\infty}(I \times \R^d)}  \frac{\hat{\rho}((t,x)-(s,y))^{\gamma\alpha_0/2} }{\eta_0^{\gamma\alpha_0/2}},
\end{align*}
and the fact that
\begin{align*}
\hat{\rho}((t,x)-(s,y))^{\gamma\alpha_0/2} \le \left(\tilde{c}(\alpha_0,\eta_0)\left(\vert x-y \vert + \vert t-s \vert^{1/\am}\right)\right)^{\gamma\alpha_0/2}
\end{align*}
for some constant $\tilde{c}(\alpha_0,\eta_0)$, which can be proven in a similar way as \eqref{eq:distineq}.
\enlargethispage{5ex}
Altogether, we obtain the desired result with $\gamma = \gamma \alpha_0/2$ and $\eta$ depending on $c$ and $\tilde{c}$.
\end{proof}

\section{Examples}\label{sec:examples}

In this section we present several examples of admissible pairs $(\mu,A)$ in the sense of \autoref{def:admissible-pairs} satisfying the given assumptions. The section consists of the following three subsections:
\begin{enumerate}
	\item[7.1] \emph{Nonlocal operators with jumping measures supported on coordinate axes}
	\item[7.2] \emph{Mixed local-nonlocal operators}
	\item[7.3] \emph{Nonlocal operators with jumping measures supported on cusps}	
\end{enumerate}

\subsection{Nonlocal operators with jumping measures supported on coordinate axes}
\label{sec:mu-axes}

The first prototype example covers the case where the underlying stochastic process consists of jump processes in every coordinate direction and there are no diffusive components. Thus we set $d = d_1$ and $d_2 = 0$. Let $\alpha_1,\dots,\alpha_d \in (\alpha_0,2)$ for some $\alpha_0 \in (0,2)$ and define 
\begin{align*}
\ma(x,\d y)=\sum_{k=1}^d \Big( (2-\alpha_k) |x_k-y_k|^{-1-\alpha_k}\,  \d y_k\prod_{i\neq k}\delta_{\{x_i\}}(\d y_i) \Big).
\end{align*}

In the sequel we verify all necessary conditions for \autoref{thm:weakHarnack} and \autoref{thm:HR}.

\begin{remark}
	Once we have shown that $\ma$ satisfies \autoref{assumption:func-ineq}, we can allow for any family of measures $(\mu(x,\d y))_{x \in \R^d}$ satisfying \autoref{assumption:symmetry} and \autoref{assumption:tail}. Then \autoref{assumption:func-ineq} follows if we prove for every $x_0 \in \R^d$, $\rho > 0$
	\begin{align*}
		\Lambda^{-1} \cE^{\mu}_{M_{\rho}(x_0)}(v,v) \le \cE^{\ma}_{M_{\rho}(x_0)}(v,v) \le \Lambda \cE^{\mu}_{M_{\rho}(x_0)}(v,v) ~~ \forall v \in V(M_{\rho}(x_0)|\R^d) \,.
	\end{align*}
	In particular, one could replace $|x_k-y_k|^{-1-\alpha_k}$ by $a_k(x,y) |x_k-y_k|^{-1-\alpha_k}$, where $\Lambda^{-1} \leq a_k \leq \Lambda$ is any symmetric function. 
\end{remark}

Clearly, $\ma$ satisfies \autoref{assumption:symmetry}. The properties \autoref{assumption:tail}, and \eqref{eq:fct-space-comp} hold true with a choice of $\Lambda$ that depends on $\alpha_1, \ldots, \alpha_d$ only through $\alpha_0$. See also \eqref{eq:tail-mu-axes}. The Sobolev inequality \eqref{eq:sobolev-assum} with $\Lambda = \Lambda(\alpha_0, d)$ has been proved already in \cite[Theorem 2.5]{ChKa20}. \\ In order to verify $(\ma,0) \in \cK(\alpha_0,\Lambda)$ uniformly in $\alpha_1, \ldots, \alpha_d$, it remains to prove the Poincar\'e inequality \eqref{eq:poincare-assum} \footnote{We thank B. Dyda for helpful discussions regarding the Poincar\'e inequality with us.}.

\begin{theorem}[Poincar\'e inequality]
\label{thmpoincare}
There is a constant $c_1 = c_1(d) > 0$ such that for every $r > 0$, $x_0\in \R^d$ and every measurable $v:\R^d\to\R$ with $v\in L^2(M_r(x_0))$, 
the following Poincar\'e inequality holds:
 \[ \|v-[v]_{M_r(x_0)}\|_{L^2(M_r(x_0))}^2 \leq c_1r^{\am}\mathcal{E}^{\ma}_{M_r(x_0)}(v,v). \]
\end{theorem}

\begin{remark}
A nice property of the aforementioned Poincar\'e inequality is the fact that the constants are independent of $\alpha_1,\dots,\alpha_d$, which has the consequence that for sufficiently nice functions it recovers in the limit, $\alpha_k\nearrow 2$ for all $k\in\{1,\dots,d\}$, the classical Poincar\'e inequality with the classical Dirichlet energy, that is
\[ \|u-[u]_{Q_r(x_0)}\|_{L^2(Q_r(x_0))}^2 \leq cr^2\int_{Q_r(x_0)}|\nabla u(x)|^2\, \d x,\]
where $Q_r(x_0)$ is the cube with center $x_0$ and side-length $2r$. 
Furthermore, it is even possible to consider the limit in some directions. The corresponding energy form can be described as a mixed local and nonlocal energy with local effects in the directions of the limit and nonlocal effects in the other directions and will be considered in \autoref{sec:lnl}.
\end{remark}

Let us start the proof of \autoref{thmpoincare} by proving the following technical lemma.
\begin{lemma}\label{lemma}
 Let $r>0$, $x_0\in\R^d$, $a>0$ and $N\in\N$. 
 For every $u\in L^2(M_r(x_0))$ and $k\in\{1,\dots,d\}$
 \begin{align*}
 & \int_{M_r(x_0)}  \int_{(x_0)_k-r^\frac{\am}{\alpha_k}}^{(x_0)_k+r^\frac{\am}{\alpha_k}} (u(x)-u(x+e_k(y_k-x_k)))^2\mathds{1}_{[a,2a)}(|x_k-y_k|) \,\d y_k \,\d x \\
   & \leq  N^{3}\int_{M_r(x_0)} \int_{(x_0)_k-r^{\frac{\am}{\alpha_k}}}^{(x_0)_k+r^{\frac{\am}{\alpha_k}}} (u(x)-u(x+e_k(y_k-x_k)))^2\mathds{1}_{[N^{-1}a,N^{-1}2a)}(|x_k-y_k|) \,\d y_k \,\d x. 
 \end{align*}
\end{lemma}
\begin{proof}
Without loss of generality, let $x_0=0$. Let
$z^j := x+e_k\frac{j(y_k-x_k)}{N}$ for $j=0,1,\dots,N$. 
As a consequence of convexity of the set $M_r$, $z^j\in M_r$, $j=0,1,\dots,N$ for every $x\in M_r$ and $y_k\in (-r^\frac{\am}{\alpha_k},r^\frac{\am}{\alpha_k})$. Furthermore, by construction $z^0=x$ and $z^N = x+e_k(y_k-x_k)$.
%For given $y\in\R^d$ and $j\in\{0,\dots,N\}$, we define the diffeomorphism $\Phi^y_j:\R^d\to\R^d$, $\Phi_j^y(x)=\diag (c^j_1,\dots,c^j_d)x+\frac{j}{N}e_ky$, where
%\[ c_i^j = \begin{cases} 1, & \text{if } i\neq j, \\ 1-\frac{j}{N}, & \text{ if } i=j. \end{cases} \]
Note that $|z^{j}_k-z^{j-1}_k|=\frac{|y_k-x_k|}{N}$ and $z^j=z^{j-1}+e_k(z^j_k-z^{j-1}_k)$. Combining these properties, leads to
\begin{align*}
 \int_{M_r} & \int_{-r^{\frac{\am}{\alpha_k}}}^{r^{\frac{\am}{\alpha_k}}} (u(x)-u(x+e_k(y_k-x_k)))^2\mathds{1}_{[a,2a)}(|x_k-y_k|) \,\d y_k \,\d x \\
 & \leq N\sum_{j=1}^N  \int_{M_r} \int_{-r^\frac{\am}{\alpha_k}}^{r^\frac{\am}{\alpha_k}} (u(z^{j-1})-u(z^j))^2\mathds{1}_{[a,2a)}(|x_k-y_k|) \,\d y_k \,\d x \\
 & = N\sum_{j=1}^N  \int_{M_r} \int_{-r^\frac{\am}{\alpha_k}}^{r^\frac{\am}{\alpha_k}} (u(z^{j-1})-u(z^{j}))^2\mathds{1}_{[a/N,2a/N)}(|z^{j-1}_k-z^j_k|) \,\d y_k \,\d x \\
 & \leq N\sum_{j=1}^N  \int_{M_r} \int_{-r^{\frac{\am}{\alpha_k}}}^{r^\frac{\am}{\alpha_k}} (u(z^{j-1})-u(z^{j}))^2\mathds{1}_{[a/N,2a/N)}(|z^{j-1}_k-z^j_k|) \,\frac{N}{j}\d z^j_k \,\frac{N}{N-j+1}\d z^{j-1} \\
 &\leq N^3 \int_{M_r} \int_{-r^{\frac{\am}{\alpha_k}}}^{r^\frac{\am}{\alpha_k}} (u(x)-u(x+e_k(y_k-x_k)))^2\mathds{1}_{[a/N,2a/N)}(|x_k-y_k|) \,\d y_k \,\d x,
\end{align*}
where we renamed $z^{j-1}$ and $z^{j}_k$ by $x$ respectively $y_k$ and used the estimate $\frac{1}{j(N-j+1)} \leq \frac{1}{N}$ in order to get rid of the sum over $j$ in the last inequality.
\end{proof}

We are now in the position to prove the Poincar\'e inequality.

\begin{proof}[Proof of \autoref{thmpoincare}]
 For simplicity of notation and without loss of generality, we assume $x_0=0$. 
 The proof for general $x_0\in\R^d$ works analogously.
 By Jensen's inequality
 \begin{align}\label{jensen}
  \|v-[v]_{M_r}\|_{L^2(M_r)}^2 \leq \frac{1}{|M_r|}\int_{M_r} \int_{M_r} (v(x)-v(y))^2\, \,\d y \, \,\d x := J.
\end{align}  

We define the polygonal chain $\ell=(\ell_0(x,y),\dots,\ell_d(x,y))$ joining $x$ and $y$ as in the proof of \autoref{thmcutoffest}.
This leads to,                                                     
\begin{align}\label{decomb}
J \leq \frac{d}{|M_r|}\sum_{k=1}^d \int_{M_r} \int_{M_r} (v(\ell_{k-1}(x,y))-v(\ell_k(x,y)))^2\, \d y \, \d x := \frac{d}{|M_r|}\sum_{k=1}^d I_k.
\end{align} 
Let us fix $k\in\{1,\dots,d\}$ and define $w:=\ell_{k-1}(x,y) = (y_1,\dots,y_{k-1},x_k,\dots,x_d)$. Moreover let $z:= x+y-w=(x_1,\dots,x_{k-1},y_k,\dots,y_d)$. Then
$\ell_k(x,y) = w+e_k(z_k-w_k) = (y_1,\dots,y_{k},x_{k+1},\dots,x_d)$. By Fubini's Theorem
\begin{align}\label{I-k}
I_k = 2^{d-1}r^{\sum_{j\neq k} \frac{\am}{\alpha_j}} \int_{M_r} \int_{-r^{\frac{\am}{\alpha_k}}}^{r^{\frac{\am}{\alpha_k}}} (v(w)-v(w+e_k(z_k-w_k)))^2 \,\d z_k \,\d w.
\end{align} 
In the following, we aim to show that there is a constant $c_1>0$ depending on the dimension $d$, such that
\begin{align}\label{aimpoinc}
\begin{aligned}
&\int_{M_r}  \int_{-r^{\frac{\am}{\alpha_k}}}^{r^{\frac{\am}{\alpha_k}}} (v(w)-v(w+e_k(z_k-w_k)))^2 \,\d z_k \,\d w \\
& \leq c_1 r^{\frac{\am}{\alpha_k} + \am} \int_{M_r} \int_{-r^{\frac{\am}{\alpha_k}}}^{r^{\frac{\am}{\alpha_k}}} (v(w)-v(w+e_k(z_k-w_k)))^2 
\frac{2-\alpha_k}{|z_k-w_k|^{1+\alpha_k}}\d z_k \,\d w. 
\end{aligned}
\end{align}
Once we have this inequality, the proof would be finished. Let us first show how \eqref{aimpoinc} implies the assertion of the theorem.
Combining \eqref{jensen}, \eqref{decomb}, \eqref{I-k} and \eqref{aimpoinc} leads to the existence of constants $c_2,c_3>0$, depending on the dimension $d$ only, such that
\begin{align*}
  & \|v-[v]_{M_r}\|_{L^2(M_r)}^2 \\
  & \leq \frac{d}{|M_r|} \sum_{k=1}^d 2^{d-1}r^{\sum_{j\neq k} \frac{\am}{\alpha_j}} \int_{M_r} \int_{-r^{\frac{\am}{\alpha_k}}}^{r^{\frac{\am}{\alpha_k}}} (v(w)-v(w+e_k(z_k-w_k)))^2 \,\d z_k \,\d w \\
   & \leq \frac{c_2r^{\sum_{j=1}^{d}  \frac{\am}{\alpha_j}}}{|M_r|}r^{\am} \sum_{k=1}^d \int_{M_r} \int_{-r^{\frac{\am}{\alpha_k}}}^{r^{\frac{\am}{\alpha_k}}} (v(w)-v(w+e_k(z_k-w_k)))^2 \,\frac{2-\alpha_k}{|z_k-w_k|^{1+\alpha_k}} \d z_k \,\d w\\
   & = c_3r^{\am} \mathcal{E}_{M_r}^{\ma}(v,v).
\end{align*}
Hence, it remains to prove \eqref{aimpoinc}.
Let $\beta_{j,n}\geq 0$ for all $j,n\in \N\cup\{0\}$ such that $\sum_{j=0}^{\infty} \beta_{j,n}\geq 1$ for all $n\in\N\cup\{0\}$. The sequence $\beta_{j,n}$
will be chosen later. By \autoref{lemma} for $N=N(j)=2^{j+1}$ and $a=2^{-n}r^{\frac{\am}{\alpha_k}}$, 
\begin{align*}
& \int_{M_r} \int_{-r^{\frac{\am}{\alpha_k}}}^{r^{\frac{\am}{\alpha_k}}} (v(w)-v(w+e_k(z_k-w_k)))^2 \,\d z_k \,\d w \\
& = \sum_{n=0}^{\infty} \int_{M_r} \int_{-r^{\frac{\am}{\alpha_k}}}^{r^{\frac{\am}{\alpha_k}}} (v(w)-v(w+e_k(z_k-w_k)))^2\mathds{1}_{[2^{-n}r^{\frac{\am}{\alpha_k}},2^{-n+1}r^{\frac{\am}{\alpha_k}})}(|w_k-z_k|) \,\d z_k \,\d w \\
%& \leq \sum_{j,n=0}^{\infty} \beta_{j,n} \int_{M_r} \int_{-r^{\frac{\am}{\alpha_k}}}^{r^{\frac{\am}{\alpha_k}}} (v(w)-v(w+e_k(z_k-w_k)))^2\mathds{1}_{[2^{-n}r^{\frac{\am}{\alpha_k}},2^{-n+1}r^{\frac{\am}{\alpha_k}})}(|w_k-z_k|) \,\d z_k \,\d w \\
& \leq \sum_{j,n=0}^{\infty}\beta_{j,n}2^{3(j+1)} \int_{M_r} \int_{-r^{\frac{\am}{\alpha_k}}}^{r^{\frac{\am}{\alpha_k}}} (v(w)-v(w+e_k(z_k-w_k)))^2\\
& \hspace*{6cm}\mathds{1}_{[2^{-n-j-1}r^{\frac{\am}{\alpha_k}},2^{-n-j}r^{\frac{\am}{\alpha_k}})}(|w_k-z_k|) \,\d z_k \,\d w \\
& = \int_{M_r} \int_{-r^{\frac{\am}{\alpha_k}}}^{r^{\frac{\am}{\alpha_k}}} (v(w)-v(w+e_k(z_k-w_k)))^2 h_k(|w_k-z_k|)\,\d z_k \,\d w,
\end{align*}
with 
\[ h_k(s)=\sum_{n=0}^{\infty}\left(\sum_{j=0}^{n}\beta_{j,n-j}2^{3(j+1)}\right)\mathds{1}_{[2^{-n-1}r^{\frac{\am}{\alpha_k}},2^{-n}r^{\frac{\am}{\alpha_k}})}(s). \]
Let
\[ \beta_{j,n}:= \ln(2)(2-\alpha_k)2^{-j(2-\alpha_k)}. \]
Then
\[ 1\leq \sum_{j=0}^{\infty}\beta_{j,n}= \frac{\ln(2)(2-\alpha_k)}{1-2^{-2+\alpha_k}}< 2 \quad \text{ for all } k \in \lbrace 1 \dots d\rbrace.\]
Hence,
\begin{align*}
h_k(s)&= \ln(2)(2-\alpha_k)\sum_{n=0}^{\infty}\sum_{j=0}^{n}2^{3(j+1)}2^{-j(2-\alpha_k)}\mathds{1}_{[2^{-n-1}r^{\frac{\am}{\alpha_k}},2^{-n}r^{\frac{\am}{\alpha_k}})}(s) \\
%& = 8\ln(2)(2-\alpha_k)\sum_{n=0}^{\infty}\sum_{j=0}^{n}2^{j(1+\alpha_k)}\mathds{1}_{[2^{-n-1}r^{\frac{\am}{\alpha_k}},2^{-n}r^{\frac{\am}{\alpha_k}})}(s) \\
& = 8\ln(2)(2-\alpha_k)\sum_{n=0}^{\infty}\frac{2^{(n+1)(1+\alpha_k)}-1}{2^{(1+\alpha_k)}-1}\mathds{1}_{[2^{-n-1}r^{\frac{\am}{\alpha_k}},2^{-n}r^{\frac{\am}{\alpha_k}})}(s) \\
& \leq 64\ln(2)(2-\alpha_k)\sum_{n=0}^{\infty}2^{n(1+\alpha_k)}\mathds{1}_{[2^{-n-1}r^{\frac{\am}{\alpha_k}},2^{-n}r^{\frac{\am}{\alpha_k}})}(s) \\
& \leq 64\ln(2)(2-\alpha_k) \left(\frac{r^\frac{\am}{\alpha_k}}{s}\right)^{1+\alpha_k} = 64\ln(2)(2-\alpha_k) \frac{r^{\am/\alpha_k + \am}}{s^{1+\alpha_k}},
\end{align*}
which proves \eqref{aimpoinc} and finishes the proof.
\end{proof} 

We close this section by giving two simple examples that are directly derived from $\ma$.

\begin{example}
Consider for given $\alpha,\beta \in [\alpha_0,2)$ the following family of measures $(\mu_{\alpha,\beta}(x,\cdot))_{x \in \R^3}$ on $\R^3$ given by
\begin{align*}
\mu_{\alpha,\beta}(x,\d y) &= (2-\alpha)\left(\sqrt{\vert x_1 - y_1 \vert^2 + \vert x_2 - y_2 \vert^2}\right)^{-2-\alpha} \d y_1 \d y_2 \delta_{\lbrace x_3 \rbrace}(\d y_3)\\
& + (2-\beta) \vert x_3 - y_3 \vert^{-1-\beta} \delta_{\lbrace x_1 \rbrace}(\d y_1) \delta_{\lbrace x_2 \rbrace}(\d y_2).
\end{align*}
The associated energy form $\cE^{\mu_{\alpha,\beta}}$ corresponds to the process $(X_t,Y_t)$ with state space $\R^3$, where $X_t$ and $Y_t$ are independent and $X_t$ is a $2$-dimensional $\alpha$-stable process and $Y_t$ is a $1$-dimensional $\beta$-stable process.\\
It can be seen easily that $\cE^{\mu_{\alpha,\beta}}$ is an admissible energy form since $\mu_{\alpha,\beta}$ satisfies \autoref{assumption:symmetry}, \autoref{assumption:tail} and \autoref{assumption:func-ineq}, where $\ma$ is defined with respect to $\alpha_1 = \alpha_2 = \alpha$, $\alpha_3 = \beta$.
\end{example}

\begin{proposition}
Let $\alpha,\beta \in [\alpha_0,2)$. Then there exists $\Lambda = \Lambda(\alpha_0)$ such that: 
\begin{itemize}
\item[(i)] For every $x_0 \in \R^3$, $\rho > 0$ it holds
\begin{align*}
&\mu_{\alpha,\beta}(x_0,\R^3 \setminus E_{\rho}^k(x_0)) \le \Lambda (2-\alpha)\rho^{-\max\{\alpha,\beta\}} \quad \text{for every } k \in \lbrace 1,2 \rbrace,\\
& \mu_{\alpha,\beta}(x_0,\R^3 \setminus E_{\rho}^3(x_0)) \le \Lambda (2-\beta)\rho^{-\max\{\alpha,\beta\}}.
\end{align*}
\item[(ii)] For every $x_0 \in \R^3$, $\rho > 0$ and $v \in H^{\ma}(M_{\rho}(x_0))$ it holds
\begin{align*}
\Lambda^{-1} \cE^{\mu_{\alpha,\beta}}_{M_{\rho}(x_0)}(v,v) \le \cE^{\ma}_{M_{\rho}(x_0)}(v,v) \le \Lambda \cE^{\mu_{\alpha,\beta}}_{M_{\rho}(x_0)}(v,v).
\end{align*}
\end{itemize}
In particular, $(\mu_{\alpha,\beta},0) \in \cK(\alpha_0,\Lambda)$.
\end{proposition}

\begin{proof}
(i) follows from the following easy calculations:
\begin{align*}
\mu_{\alpha,\beta}(x_0,\R^3 \setminus E_{\rho}^1(x_0)) &= 4(2-\alpha)\int_{\rho^{\max\{\alpha,\beta\}/\alpha}}^{\infty} \int_{0}^{\infty} \vert h \vert^{-2-\alpha} \d h_2 \d h_1 \\
& \le 4^3(2-\alpha) \int_{\rho^{\max\{\alpha,\beta\}/\alpha}}^{\infty} \int_{0}^{\infty} (h_1 + h_2)^{-2-\alpha} \d h_2 \d h_1 \\
&= \frac{4^3(2-\alpha)}{1+\alpha} \int_{\rho^{\max\{\alpha,\beta\}/\alpha}}^{\infty} h_1^{-1-\alpha} \d h_1 \\
&= \frac{4^3(2-\alpha)}{(1+\alpha)\alpha} \rho^{-\max\{\alpha,\beta\}}.
\end{align*}
Analogously we establish
\begin{align*}
\mu_{\alpha,\beta}(x_0,\R^3 \setminus E_{\rho}^2(x_0)) \le \frac{4^3(2-\alpha)}{(1+\alpha)\alpha} \rho^{-\max\{\alpha,\beta\}}.
\end{align*} 
Finally, we compute
\begin{align*}
\mu_{\alpha,\beta}(x_0,\R^3 \setminus E_{\rho}^3(x_0)) = 2(2-\beta) \int_{\rho^{\max\{\alpha,\beta\}/\beta}}^{\infty} \vert h \vert^{-1-\beta} \d h = \frac{2(2-\beta)}{\beta} \rho^{-\max\{\alpha,\beta\}}.
\end{align*}

(ii) follows from \cite[Lemma 3.3]{DyKa15} applied to balls of the form $M_{\rho}(x_0)$.\\
In particular, $(\mu_{\alpha,\beta},0) \in \cK(\alpha_0,\Lambda)$, since by the first estimate in (ii) the Sobolev- and Poincar\'e inequality \eqref{eq:sobolev-assum}, \eqref{eq:poincare-assum} are inherited from $(\ma,0)$.
\end{proof}

\begin{example}
\label{diffjumpintensityatinfty}

Let $\alpha_0 \in (0,2)$ and $\alpha_1,\dots,\alpha_d,\beta_1,\dots,\beta_d \in [\alpha_0,2)$ with $\beta_k \le \alpha_k$ for every $k \in \lbrace 1,\dots,d\rbrace$. We define $\mu_{(\alpha_k),(\beta_k)}(x,\cdot)$ by
\begin{align*}
\mu_{(\alpha_k),(\beta_k)}(x,\d y) = \sum_{k=1}^d \Big( (2-\alpha_k) |x_k-y_k|^{-1-\alpha_k} + (2-\beta_k) |x_k-y_k|^{-1-\beta_k}\,  \d y_k\prod_{i\neq k}\delta_{\{x_i\}}(\d y_i) \Big).
\end{align*}
Note that $\mu_{(\alpha_k),(\beta_k)} \in \cK(\alpha_0,\Lambda)$, where $\ma$ is defined with respect to $( \alpha_k)$. This is a direct consequence of the fact that
\begin{align*}
(2-\alpha_k) |x_k-y_k|^{-1-\alpha_k} &\le (2-\alpha_k) |x_k-y_k|^{-1-\alpha_k} + (2-\beta_k) |x_k-y_k|^{-1-\beta_k}\\
&\le c (2-\alpha_k) |x_k-y_k|^{-1-\alpha_k},
\end{align*}
for every $x,y \in M_3(x_0)$ and $x_0 \in \R^d$ and $k \in \lbrace 1,\dots,d \rbrace$, where $c = c(\alpha_0,\am) > 0$ is some constant. Note that the parameter $\Lambda$ depends on $\alpha_0$ and $\am$.
\end{example}

\subsection{Mixed local-nonlocal operators}
\label{sec:lnl}

The second prototype example covers the case where the underlying stochastic process consists of jump processes in the first $d_1$ directions and diffusive components in the remaining $d_2$ directions. To this end, we consider $d_1,d_2\geq 1$, $d=d_1 + d_2$, and $\alpha_1,\dots,\alpha_{d_1}\in(\alpha_0,2)$ for some $\alpha_0 \in (0,2)$. 
Let $(\ma,A)$ be given as follows:
\begin{align*}
\ma(x,\d y)=\sum_{k=1}^{d_1} \Big( (2-\alpha_k) |x_k-y_k|^{-1-\alpha_k}\,  \d y_k\prod_{i\neq k}\delta_{\{x_i\}}(\d y_i) \Big), \quad (A_{j,k})_{j,k=d_1+1}^d = (\delta_{j,k})_{j,k=d_1+1}^d.
\end{align*}

Let us now verify $(\ma,A) \in \cK(\alpha_0,\Lambda)$ with a choice of $\Lambda$ that depends on $\alpha_1, \ldots, \alpha_{d_1}$ only through $\alpha_0$. 

\begin{remark}
The above choice of $(\ma, A)$ extends directly to more general situations because \autoref{assumption:symmetry}, \autoref{assumption:tail}, and \autoref{assumption:func-ineq} allow for non-translation invariant versions of $\ma$ and $A$. \autoref{thm:genmainresult-ln} covers such a general setting.
\end{remark}

Clearly, $(\ma,A)$ satisfies \autoref{assumption:symmetry},   \autoref{assumption:tail}, and \eqref{eq:fct-space-comp}. In order to verify $(\ma,A) \in \cK(\alpha_0,\Lambda)$ uniformly in $\alpha_1, \ldots, \alpha_{d_1}$, it remains to prove the Poincar\'e inequality \eqref{eq:poincare-assum} and the Sobolev inequality \eqref{eq:sobolev-assum}.

\begin{lemma}
\label{lemma:lnl-sob}
There is a constant $\Lambda=\Lambda(d, \alpha_0) \ge 1$ such that for every $r \in (0,2]$, $\lambda \in (1,2]$, $x_0\in \R^d$ and every $u\in V(M_{\lambda r}(x_0) | \R^d)$,
the following Sobolev inequality holds:
 \[ \Vert u \Vert_{L^{\frac{2\beta}{\beta-1}}(M_{r}(x_0))}^2 \le \Lambda \mathcal{E}_{M_{\lambda r}(x_0)}(u,u) + \Lambda r^{-2}\left( \sum_{k=1}^d (\lambda^{\frac{2}{\alpha_k}} - 1)^{-\alpha_k}\right) \Vert u \Vert_{L^2(M_{\lambda r}(x_0))}^2,\\ \]
\end{lemma}

\begin{proof}
First, one establishes a global Sobolev inequality 
\begin{align*}
\Vert u \Vert_{L^{\frac{2\beta}{\beta-1}}(\R^d)}^2 \le c \cE(u,u)
\end{align*}
as in \cite[Theorem 2.4]{ChKa20}, using that $\cE(u,u) = c \int_{\R^d} |\hat{u}(\xi)|^2 \Psi(\xi) \d \xi$, where $\Psi(\xi) = \sum_{k = 1}^{d_1} |\xi_k|^{\alpha_k} + \sum_{k = d_1 + 1}^d |\xi_k|^2$. Since $(\ma,A)$ satisfies \autoref{assumption:tail}, and thus \eqref{cutoffest}, a localization argument as in \cite[Theorem 2.5]{ChKa20} yields the desired result.
\end{proof}

The main work consists in the verification of the Poincar\'e inequality. Its proof relies on \autoref{thmpoincare} and crucially uses its robustness with respect to $\alpha_k \nearrow 2$.

\begin{theorem}
\label{lemma:lnl-poinc}
There is a constant $c_1=c_1(d) > 0$ such that for every $r \in (0,2]$, $x_0\in \R^d$ and every $u\in V(M_r(x_0)|\R^d)$,
the following Poincar\'e inequality holds:
 \[ \|u-[u]_{M_r(x_0)}\|_{L^2(M_r(x_0))}^2 \leq c_1r^{2}\mathcal{E}_{M_r(x_0)}(u,u). \]
\end{theorem}

First, we take $\alpha_{d_1},\dots,\alpha_d \in (0,2)$ and define for $r > 0$ and $x_0 \in \R^d$.
\begin{align*}
M_r'(x) &= \bigtimes_{k=1}^{d} \left(x_k-r^{2/\alpha_k},x_k+r^{2/\alpha_k}\right),\\
\ma'(x,\d y) &= \sum_{k=1}^{d} \Big( (2-\alpha_k) |x_k-y_k|^{-1-\alpha_k}\,  \d y_k\prod_{i\neq k}\delta_{\{x_i\}}(\d y_i) \Big).
\end{align*}
The first step in the proof of \autoref{lemma:lnl-poinc} is to establish convergence of the purely nonlocal energy form $\cE^{\ma'}$ towards $\cE$ as $\alpha_k \nearrow 2$ for all $k \in \lbrace d_1 + 1,\dots,d\rbrace$.  For a related result, we refer to \cite{DuBo23}, where a Bourgain-Brezis-Mironescu formula is proved for a similar family of nonlinear anisotropic energies.

\begin{lemma}\label{formconv}
Let $\alpha_{d_1},\dots,\alpha_d \in (0,2)$. Then, there is a constant $c_1 = c_1(d) > 0$ such that for every $r > 0$, $x_0 \in \R^d$ 
\begin{equation}
\cE^{\ma'}_{M_r'(x_0)}(u,u) \to c_1 \cE_{M_r(x_0)}(u,u) ~~ \forall u \in C^2_b(\R^d),
\end{equation}
as $\alpha_k \nearrow 2$ for all $k \in \lbrace d_1 + 1,\dots,d\rbrace$.
\end{lemma}

\begin{proof}
We show that for each $\rho > 0$, $x \in \R^d$, $k \in \lbrace d_1 + 1, \dots, d\rbrace$
\begin{align}
\label{eq:formconvaux}
\int_{M_{\rho}'(x)} \hspace{-0.2cm} (u(x)-u(y))^2 \mu^{(k)}(x,\d y) \to c(\partial_k u(x))^2, \quad \int_{\R^d \setminus M_{\rho}'(x)} \hspace{-0.2cm} (u(x)-u(y))^2 \mu^{(k)}(x,\d y) \to 0,
\end{align}
as $\alpha_k \nearrow 2$, where $c$ is independent of $x,\rho$. For the first claim in \eqref{eq:formconvaux}, we apply Taylor's theorem
\[(u(x) - u(x+e_k(y_k-x_k)))^2 = \partial_ku(x)^2(x_k-y_k)^2 + R(x,x+e_k(y_k-x_k))^2|x_k-y_k|^{3},  \]
where $R$ is a bounded remainder term. Hence
\begin{align*}
\begin{aligned}
&\int_{M_{\rho}'(x)} (u(x)-u(y))^2 \mu^{(k)}(x,\d y) = (2-\alpha_k) \int_{-\rho^{2/\alpha_k}}^{\rho^{2/\alpha_k}} (u(x)-u(x+e_k h_k))^2 |h_k|^{-1-\alpha_k} \, \d h_k \\
& = (2-\alpha_k) \left(\int_{-\rho^{2/\alpha_k}}^{\rho^{2/\alpha_k}} (\partial_k u(x))^2|h_k|^{1-\alpha_k} \, \d h_k + \int_{-\rho^{2/\alpha_k}}^{\rho^{2/\alpha_k}} R(x,x+e_k h_k)^2 |h_k|^{2-\alpha_k} \, \d h_k\right) \\
& =(2-\alpha_k)(I_1+I_2).
\end{aligned}
\end{align*}
Let us first consider $I_1$:
\begin{equation*}
(2-\alpha_k) I_1 = (\partial_k u(x))^2 (2-\alpha_k)\int_{-\rho^{2/\alpha_k}}^{\rho^{2/\alpha_k}} |h_k|^{1-\alpha_k} \, \d h_k = 2(\partial_k u(x))^2 \rho^{4/\alpha_k - 2} \to 2(\partial_k u(x))^2,
\end{equation*}
as $\alpha_k \nearrow 2$.
Furthermore, we have for $I_2$:
\begin{equation*}
(2-\alpha_k)I_2 \le (2-\alpha_k) c_2 \int_{-\rho^{2/\alpha_k}}^{\rho^{2/\alpha_k}} |h_k|^{2-\alpha_k} \, \d h_k \le \frac{2-\alpha_k}{3-\alpha_k} c_2\rho^{6/\alpha_k - 2} \to 0,
\end{equation*}
as $\alpha_k \nearrow 2$, where $c_2 > 0$. 
Together, this proves the first claim in \eqref{eq:formconvaux}. For the second claim:
\begin{equation*}
\begin{split}
\int_{\R^d \setminus M_{\rho}'(x)} (u(x)-u(y))^2 \mu^{(k)}(x,\d y) &\le 8(2-\alpha_k)\Vert u \Vert^2_{\infty} \int_{\rho^{2/\alpha_k}}^{\infty} \vert h_k \vert^{-1-\alpha_k} \, \d h_k\\
&\le 8(2-\alpha_k)\alpha_k^{-1} \Vert u \Vert^2_{\infty} \rho^{-2} \to 0.
\end{split}
\end{equation*}
This proves \eqref{eq:formconvaux}. The result follows by choosing for each $x \in M_r'(x_0)$ a number $\rho = \rho (x,x_0,r) \in (0,1)$ such that $M_{\rho}'(x) \subset M_r'(x_0)$ and applying \eqref{eq:formconvaux} to show that
\begin{equation*}
\int_{M_{r}'(x_0)} (u(x)-u(y))^2 \mu^{(k)}(x,\d y) \to c(\partial_k u(x))^2.
\end{equation*}
Finally, since by similar arguments as above:
\begin{equation*}
\begin{split}
&\int_{M_{r}'(x_0)} (u(x)-u(y))^2 \mu^{(k)}(x,\d y)\\
& \le \int_{M_{1}'(x)} (u(x)-u(y))^2 \mu^{(k)}(x,\d y) + \int_{\R^d \setminus M_{1}'(x)} (u(x)-u(y))^2 \mu^{(k)}(x,\d y) \\
&\le 2\Vert\partial_k u\Vert_{\infty}^2 + 4\frac{2-\alpha_k}{\alpha_k}\Vert u\Vert_{\infty}^2 < \infty,
\end{split}
\end{equation*}

we are in the position to apply dominated convergence and conclude the proof.

\end{proof}

Since \autoref{formconv} only holds true for smooth functions, we require a Meyers-Serrin-type result for $H(\Omega)$. Its proof goes in a similar way as \cite[Theorem 3.67]{Fog20}\footnote{We thank G. Foghem for helpful discussions on density results in nonlocal Sobolev spaces.} but extends the ideas therein to L\'evy measures that are not necessarily absolutely continuous with respect to the Lebesgue measure. Note that related density results (for non-singular L\'evy measures) have been investigated in the literature for instance in \cite{DyKi22}, \cite{FSV15}, \cite{LuVa17}, and \cite{FoKa22}.

\begin{lemma}
\label{lemma:MS-loc-nonloc}
Let $\Omega \subset \R^d$ be open and bounded and $u \in V(\Omega | \R^d)$. Then, for any $\eps > 0$, there is $v \in C^{\infty}_c(\R^d)$ such that $\Vert u - v \Vert_{H(\Omega)} < \eps$.
\end{lemma}

\begin{proof}
Let $u \in V(\Omega | \R^d)$. 

\textbf{Step 1:} Let $\phi \in C_c^{\infty}(B_1(0))$ with $\phi \ge 0$ and $\int_{B_1(0)} \phi \d x = 1$ and consider the mollifier $\phi_{\delta}(x) = \delta^{-d} \phi(x/\delta)$. Let $\Omega'\Subset \Omega$ be an open set. We claim that for any $0 < \delta < \dist(\Omega',\partial \Omega)$ it holds $\phi_{\delta} \ast u \in C^{\infty}(\Omega') \cap V^{(k)}(\Omega' | \R^d)$, and moreover
\begin{align}
\label{eq:approx-convolution}
\Vert \phi_{\delta} \ast u - u \Vert_{V^{(k)}(\Omega' | \R^d)} := \Vert \phi_{\delta} \ast u - u \Vert_{L^2(\Omega')} + [ \phi_{\delta} \ast u - u ]_{V^{(k)}(\Omega' | \R^d)} \to 0, ~~ \text{ as } \delta \searrow 0
\end{align} 
for every $k \in \{1,\dots,d\}$.
We will first give the proof for $k \in \{ 1,\dots,d_1 \}$. To see that $\phi_{\delta} \ast u \in V^{(k)}(\Omega' | \R^d)$, it suffices to compute for $\delta < \dist(\Omega',\partial \Omega)$ using Jensen's inequality
\begin{align*}
[u \ast \phi_{\delta}]_{V^{(k)}(\Omega'|\R^d)}^2 &\le \int_{\Omega'} \int_{\R^d} \left| \int_{\R^d} \phi(z) (u(x-\delta z) - u(x - \delta z + h)) \d z \right|^2 \mu^{(k)}(0,\d h) \d x\\
&\le \int_{B_1} \phi(z) \left(\int_{\Omega'} \int_{\R^d} |u(x-\delta z) - u(x - \delta z + h)|^2  \mu^{(k)}(0,\d h) \d x\right) \d z \\
&\le \int_{B_1} \phi(z) \left(\int_{\Omega} \int_{\R^d} |u(x) - u(x+h)|^2  \mu^{(k)}(0,\d h) \d x\right) \d z < \infty,
\end{align*}
where we used that for $x \in \Omega'$ and $z \in B_1$, it holds $x-\delta z \in \Omega'+ B_{\delta} \subset \Omega$.
In order to show \eqref{eq:approx-convolution}, we get from a similar computation using Fubini's theorem
\begin{align*}
[u \ast &\phi_{\delta} - u]_{V^{(k)}(\Omega'|\R^d)}^2 \\
&\le \int_{B_1} \phi(z) \left(\int_{\Omega'} \int_{\R^d} |[u(x-\delta z) - u(x - \delta z + h)]-[u(x)-u(x+h)]|^2  \mu^{(k)}(0,\d h) \d x\right) \d z \\
&= \int_{B_1} \int_{\R^d} \left( \int_{\Omega'} |U_h(x-\delta z)-U_h(x)|^2 \d x \right) \mu^{(k)}(0,\d h) \phi(z) \d z,
\end{align*}
where we denote $U_h(x) = u(x) - u(x+h)$. Note that $U_h(\cdot - \delta z) \in L^2(\Omega')$ for a.e. $h \in \R^d$ and $z \in B_1$ since $u,u \ast \phi_{\delta} \in V^{(k)}(\Omega'| \R^d)$. Thus, by continuity of the shift in $L^2(\Omega')$ it holds
\begin{align*}
\int_{\Omega'} |U_h(x-\delta z)-U_h(x)|^2 \d x \to 0 ~~ \text{ as } \delta \searrow 0
\end{align*}
for a.e. $h \in \R^d$ and $z \in B_1$. Since we also have
\begin{align*}
\left( \int_{\Omega'} |U_h(x-\delta z)-U_h(x)|^2 \d x \right) \le 4 \int_{\Omega} |U_h(x)|^2 \d x \in L^1(B_1 \times \R^d ; \phi \d z \times \mu^{(k)}(0,\d h)), 
\end{align*}
we deduce from the dominated convergence theorem that $[u \ast \phi_{\delta} - u]_{V^{(k)}(\Omega'|\R^d)} \to 0$, as $\delta \searrow 0$.\\
Moreover, we have for any $k \in \{ d_1+1,\dots,d \}$
\begin{align}
\label{eq:approx-local}
\begin{split}
[\partial_k(\phi_{\delta} \ast u - u)]_{L^2(\Omega')}^2 &= \int_{\Omega'} \left( \int_{\R^d} \phi(z)(\partial_k u(x-\delta z) - \partial_k u(x)) \d z \right)^2 \d x\\
&\le \int_{B_1(0)} \hspace{-0.1cm} \left( \int_{\Omega'} |\partial_k u(x-\delta z) - \partial_k u(x)|^2 \d x \right) \phi(z) \d z \to 0 ,~~ \text{ as } \delta \searrow 0
\end{split}
\end{align}
by the continuity of the shift in $L^2(\Omega')$ and dominated convergence. Altogether, we get \eqref{eq:approx-convolution}.

\textbf{Step 2}: Using \eqref{eq:approx-convolution}, the desired result follow by the same standard covering argument as in the proof of \cite[Theorem 3.67]{Fog20}, where also the approximating function $v \in C^{\infty}_c(\R^d)$ is defined in the exact same way. For $k \in \{1,\dots, d_1 \}$, we follow the steps in \cite[Theorem 3.67]{Fog20} in order to show $\cE^{(k)}_{\Omega}(u-v,u-v) < \eps$, and for $k \in \{ d_1 + 1, \dots, d \}$ we use
\begin{align*}
\cE^{(k)}_{\Omega}(u-v,u-v) = [\partial_k(\phi_{\delta} \ast u_j - u_j)]^2_{L^2(\Omega)} = [\partial_k(\phi_{\delta} \ast u_j - u_j)]^2_{L^2(O_{j+4})} \to 0, ~~ \text{ as } \delta \searrow 0,
\end{align*}
as a consequence of \eqref{eq:approx-local}. Altogether, these ingredients yield the desired result.
\end{proof}

We are now ready to prove \autoref{lemma:lnl-poinc}.

\begin{proof}[Proof of \autoref{lemma:lnl-poinc}]
Recall that by \autoref{thmpoincare} the following anisotropic and nonlocal robust Poincar\'e inequality holds, given $\alpha_{d_1+1},\dots,\alpha_d \in (0,2)$:
\[ \|u-[u]_{M_r'(x_0)}\|_{L^2(M_r'(x_0))}^2 \leq c_1r^{2}\mathcal{E}^{\ma}_{M_r'(x_0)}(u,u),\]
where $c_1 = c_1(d) > 0$ only depends on $d$. Taking the limit $\alpha_{d_1+1},\dots,\alpha_d \nearrow 2$ on both sides,
\[ \|u-[u]_{M_r(x_0)}\|_{L^2(M_r(x_0))}^2 \leq c_1r^{2}\mathcal{E}_{M_r(x_0)}(u,u),\]
for $u \in C^2_b(\R^d)$, where we used \autoref{formconv}. We obtain the result for any $u \in V(M_r(x_0) | \R^d)$, by application of the density result \autoref{lemma:MS-loc-nonloc}, thereby concluding the proof.
\end{proof}

Since $L$ defined in \eqref{eq:frac_aniso_laplace} satisfies the Poincar\'e- and Sobolev inequality \eqref{eq:poincare-assum} and \eqref{eq:sobolev-assum}, the verification of these properties for other local-nonlocal operators is reduced to comparing their energy with the one of $L$. This observation is summarized in the following theorem and allows to construct further interesting examples of admissible operators.

\begin{theorem}
\label{thm:genmainresult-ln}
Let $\mu = (\mu(x,\cdot))_{x \in \R^d}$ be a symmetric family of measures on $\R^d$, and $A = \left(A_{j,k}\right)_{j,k = d_1+1}^d : I \times \R^d \rightarrow \R^{d_2 \times d_2}$, and $a : I \times \R^d \times \R^d \rightarrow [1,2]$ be measurable functions. Furthermore, assume that $A(t,x)$ is a symmetric matrix and that $a(t,x,y) = a(t,y,x)$ for every $(t,x,y)$. Let $\alpha_1,\dots,\alpha_{d_1} \in (0,2)$. For $u : I \times \R^d \rightarrow \R$, define
\begin{align}
\label{eq:Lgen-ln}
-\mathcal{L} u (t,x) = \text{ p.v.} \int_{\R^d} (u(t,x) - u(t,y)) a(t,x,y)\mu(x,\d y) - \dvg \left(A(t,x)(\partial_k u(t,x))_{k = d_1+1}^{d}\right).
\end{align}
Assume that there is $\Lambda \ge 1$ such that the following holds true:
\begin{itemize}
\item[(i)] For every $x_0 \in \R^d$, $\rho \in (0,2)$ and every $k \in \lbrace 1, \dots , d_1 \rbrace$:
\begin{align}
\label{ass:tailestimate-ln}
\mu(x_0, \R^d \setminus E_{\rho}^k(x_0)) \le \Lambda (2-\alpha_k)\rho^{-2}.
\end{align}

\item[(ii)] For every $x_0 \in \R^d$, $\rho \in (0,3)$ and every $v \in L^2(M_{\rho}(x_0))$:
\begin{align}
\label{ass:nlcomparability-ln}
& \Lambda^{-1} \cE^{\mu}_{M_{\rho}(x_0)}(v,v) \le \cE^{\ma}_{M_{\rho}(x_0)}(v,v) \le \Lambda \cE^{\mu}_{M_{\rho}(x_0)}(v,v).
\end{align}

\item[(iii)] For every $t \in I$, $ x_0 \in \Omega$ and $\xi = (\xi_{d_1+1},\dots,\xi_d) \in \R^{d_2}$:
\begin{align}
\label{ass:lcomparability-ln}
\Lambda^{-1} \vert \xi \vert^2 \le \sum_{j,k = d_1 + 1}^d A_{j,k}(t,x_0)\xi_j \xi_k\le \Lambda \vert \xi \vert^2.
\end{align}
\end{itemize}
Then, \autoref{thm:weakHarnack} holds true for $\mathcal{L}$ as in \eqref{eq:Lgen-ln}. If in addition the following holds true:
\begin{itemize}
\item[(iv)] For some $\beta > 0$ and every $x_0 \in \R^d$, $\rho > 2$ and every $k \in \lbrace 1, \dots , d_1 \rbrace$:
\begin{align}
\label{ass:tailestimate-ln-glob}
\mu(x_0, \R^d \setminus E_{\rho}^k(x_0)) \le \Lambda (2-\alpha_k)\rho^{-\beta},
\end{align}
\end{itemize} 
then also \autoref{thm:HR} holds true for $\mathcal{L}$ as in \eqref{eq:Lgen-ln}.
\end{theorem}

\autoref{thm:genmainresult-ln} does not require the family of measures $(\mu(x, \cdot ))_{x \in \R^d}$ to have a product structure as in \eqref{def:muaxes}. Below, we provide a corresponding example $(\mu, A)$.

\begin{example}
Let $\alpha \in (0,2)$ and for $x \in \R^d$:
\begin{align*}
\mu(x,\d y) = (2-\alpha)\left\vert \left(x-y \right)_{k=1}^{d_1} \right\vert^{-d_1-\alpha}\d y_1 \dots \d y_{d_1}\delta_{x_{d_1+1}}(\d y_{d_1 + 1})\dots \delta_{x_d}(\d y_d).
\end{align*}
Define $a \equiv 1$ and $(A_{j,k})_{j,k=d_1+1}^d = (\delta_{j,k})_{j,k=d_1+1}^d$. Then $\mathcal{L}$ reads
\begin{align*}
-\mathcal{L}u(x) &= \text{ p.v.} \int_{\R^d}(u(x)-u(y))\mu(x,\d y) - \sum_{k = d_1 + 1}^{d}  \partial_{k}^2 u(x)\\
&= \left(\sum_{k = 1}^{d_1} -\partial_k^2 \right)^{\alpha/2}u(x) + \left(\sum_{k = d_1 + 1}^{d} -\partial_k^2\right) u(x).
\end{align*}
$\mathcal{L}$ satisfies the assumptions (i)-(iv) of \autoref{thm:genmainresult-ln} with $\alpha = \alpha_1 = \dots = \alpha_{d_1}$ and some $\Lambda \ge 1$. Except for (ii), this follows by simple arguments. For (ii), we refer the reader to \cite{KaSc14}.\\
Note that the stochastic process $Z_t$ associated to $\mathcal{L}$ is given by $Z_t = (X_t,Y_t)$, where $X_t,Y_t$ are independent, and $X_t$ is the $d_1$-dimensional rotationally invariant $\alpha$-stable process and $Y_t$ is a $d_2$-dimensional Brownian motion.

\end{example}

\begin{remark}
Note that the aforementioned results are also satisfied by the operator $L$ as in \eqref{eq:frac_aniso_laplace} when $d_1 = 0$ or $d_2 = 0$, as long as $d \ge 2$. In the case $d_1 = 0$ we re-obtain the classical local situation since $L$ becomes the Laplace operator, whereas the case $d_2 = 0$ is covered in \autoref{sec:mu-axes}.
\end{remark}

\subsection{Nonlocal operators with jumping measures supported on cusps}
\label{sec:cusps}

We complete this section with an example of an admissible family $(\mu(x,\cdot))_{x \in \R^d}$ that is absolutely continuous with respect to the Lebesgue measure. It still satisfies the assumptions of our main results \autoref{thm:weakHarnack} and \autoref{thm:HR}. To be more precise, we show that there are a set $\Gamma$ and a parameter $\gamma$, which can be very large (in particular larger than $2$), such that $\mu(t,x,\d y) \asymp \mathds{1}_{\Gamma}(x-y) |x-y|^{-d-\gamma}$ satisfies
\autoref{assumption:symmetry}, \autoref{assumption:tail}, and \autoref{assumption:func-ineq}. In the special case $\alpha_1=\cdots=\alpha_d$, it turns out that $\Gamma=\R^d$ and $\gamma=\alpha_1=\cdots=\alpha_d$. Therefore a nice feature of this example is that it recovers the fractional Laplacian case (see \cite{DyKa15} or \cite{KaSc14} for instance).

We demonstrate the example in two dimensions to keep the calculations simple and the notation clear.
Let $\alpha_0\in(0,2)$ be given and consider $\alpha_1,\alpha_2\in [\alpha_0,2)$. We start our analysis by the definition of important quantities and introducing the family of admissible measures.
\begin{definition}
For given $\alpha_1,\alpha_2\in[\alpha_0,2)$, let
\begin{align}\label{gamma}
 \gamma= \frac{|\alpha_1-\alpha_2|+\alpha_1\alpha_2}{\min\{\alpha_1,\alpha_2\}}.
\end{align}
Furthermore, let $b=(b_1,b_2)\in\R^2$ be defined by
\begin{align}\label{b}
 b_1=\frac{1}{1+\gamma-\alpha_1} \quad \text{and} \quad b_2=\frac{1}{1+\gamma-\alpha_2}. 
\end{align}
\end{definition}
Note that from the definition of $b_1,b_2$ and $\gamma$, it can easily be seen that
\begin{align}\label{eq:bgamma}
b_1(-1-\gamma)\geq -1-\alpha_2 \quad \text{and} \quad b_2(-1-\gamma)\geq -1-\alpha_1.
\end{align} 
Using the predefined quantities, we are ready to introduce the family of measures $(\mu(x,\cdot))_{x \in \R^d}$.

\begin{example}\label{exa:cusp}
Let
\begin{align*}
\Gamma &:= \{(x_1,x_2)\in \R^2 \colon |x_2|\leq |x_1|^{1/b_1} \ \text{or} \ |x_1|\leq |x_2|^{1/b_2} \}\quad \text{and} \quad \Gamma(z):= \Gamma-\{z\}.
\end{align*}
Furthermore, we define the kernel $k:\R^2\times\R^2\to [0,\infty]$ as follows:
\[ k(z) = C(\gamma,b)|z|^{-2-\gamma} \mathds{1}_{\Gamma}(z), \]
where $C(\gamma,b):=1-\gamma+\frac{1}{\max\{b_1,b_2\}}=(2-\alpha_{\max})$.\
Finally, let
\begin{align}
\mu(x,\d y) = k(x-y)\, \d y.
\end{align}
\end{example}

\begin{remark}
Note that the value of $\gamma$ is very large if $\min\{\alpha_1,\alpha_2\} \ll \max\{\alpha_1,\alpha_2\}$ is very small. As a result the set $\Gamma$ will have strong cusps close to the origin. Actually $\gamma$ can be bounded from below by $\gamma\geq\max\{\alpha_1,\alpha_2\}$. In order to see this, assume without loss of generality $\alpha_1\leq\alpha_2$. Then
\[   \gamma-\alpha_2 = \frac{\alpha_2-\alpha_1 + \alpha_1\alpha_2}{\alpha_1}-\alpha_2 = \frac{\alpha_2-\alpha_1 + \alpha_1\alpha_2-\alpha_1\alpha_2}{\alpha_1} = \frac{\alpha_2-\alpha_1}{\alpha_1}\geq 0. \] 
\end{remark}
By the lower bound for $\gamma$ and the definition of $b_1,b_2$, we have $0<b_1,b_2\le 1$. 
Hence, we have that 
\begin{align}\label{eq:fullcomplement}
 \Gamma\cap \left(\R^2\setminus [-1,1]^2\right) =  \R^2\setminus [-1,1]^2.
 \end{align}
 Furthermore, note that $\gamma$ is bounded from above in the following way:
 \begin{align}\label{eq:upperbdgamma}
 \gamma\leq 1+\frac{2}{\alpha_0}.
 \end{align}

Let us illustrate some examples for $\Gamma$ for different values of $\alpha_1,\alpha_2\in(0,2)$. Due to \eqref{eq:fullcomplement}, we restrict the figures to the relevant area $[-1,1]^2$.

\begin{figure}[H]
 \begin{minipage}{0.3\textwidth}
\resizebox{\textwidth}{!}{%
\includegraphics{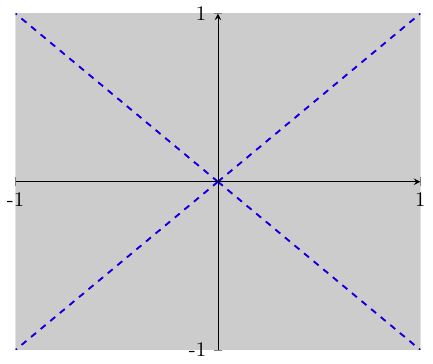}}
\end{minipage}
 \begin{minipage}{0.3\textwidth}
\resizebox{\textwidth}{!}{%
\includegraphics{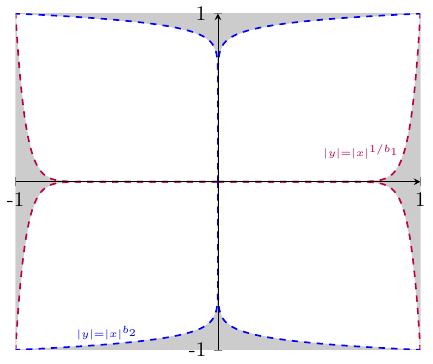}}
\end{minipage}
 \begin{minipage}{0.3\textwidth}
\resizebox{\textwidth}{!}{%
\includegraphics{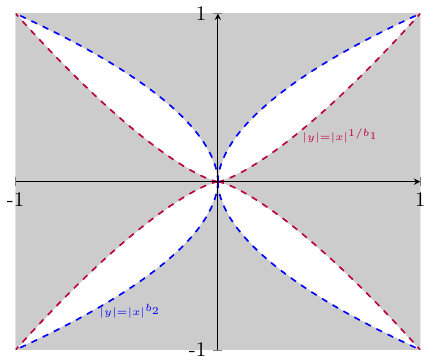}}
\end{minipage}
\caption{The three figures show examples of the set $\Gamma$. The first figure shows the case $\alpha_1=\alpha_2\in(0,2)$. 
The second figure shows the case where the distance between $\alpha_1$ and $\alpha_2$ is relatively large. 
In this case ($\alpha_1=0.1$, $\alpha_2=1.9$), we see that $\Gamma$ is quite large near the origin. Here: $\gamma = 19.9$. 
The last figure shows the example $\alpha_1=1.97$, $\alpha_2=1.48$. In this case $\gamma \approx 2.3$.}
\end{figure}

\autoref{assumption:symmetry} follows immediately by the definition of $\mu$. Hence it remains to prove \autoref{assumption:tail} and \autoref{assumption:func-ineq}. \\
If we consider the case $\alpha_1=\alpha_2=\alpha\in(0,2)$ for some $\alpha\in(0,2)$, we have $\gamma=\alpha$ and $b_1=b_2=1$. Hence, $k(x,y) = (2-\alpha)|x-y|^{-2-\alpha}$. In this case, it is known that the family of measures $\mu(x,\d y)=k(x,y)\, d y$ satisfies \eqref{eq:tailbd}. Furthermore, the corresponding energy form for $\ma$ is 
comparable to the $H^{\alpha/2}$-seminorm, which implies the second part of the subsequent \autoref{thm:comparability}.

\begin{theorem}\label{thm:comparability}
Let $\alpha_1,\alpha_2\in[\alpha_0,2)$.
\begin{enumerate}
\item There is a constant $\Lambda>0$, depending on $\alpha_0$ only, such that for every $x_0\in\R^2, r>0$ and every $k\in\{1,2\}$
\begin{align}\label{eq:tailbd}
\mu(x_0,\R^2\setminus E^k_r(x_0))\leq \Lambda (2-\alpha_k) r^{-\alpha_{\max}}.
\end{align} 

\item There is a constant $\Lambda \ge 1$ such that $\mu$ satisfies \autoref{assumption:func-ineq}.

\end{enumerate}
\end{theorem}

\begin{proof}[Proof of \autoref{thm:comparability}]
We first prove the tail estimate \eqref{eq:tailbd}, which follows by a direct calculation. 
Let $x_0\in\R^2$ and $r>0$.
\begin{figure}[H]
\includegraphics{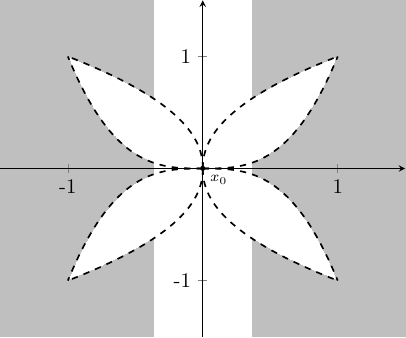}
\caption{The gray colored area shows the set $(\R^2\setminus E_r^1(x_0))$ and its intersection with $\Gamma(x_0)$ for some $r\in(0,1)$ and some given $\alpha_1,\alpha_2$.}
\end{figure}
We just prove the assertion in the case $k=1$. The case $k=2$ follows similarly. 
Note, that we can easily estimate the volume $\mu(x_0,\R^2\setminus E_{r}^1(x_0))$ as follows:
\begin{align*}
& \mu(x_0,\R^2\setminus E_{r}^1(x_0)) = \int_{\R^2\setminus E_r^1(x_0)} \mu(x_0,\d y) \\
& \leq 4(2-\alpha_{\max})\int_{r^{\alpha_{\max}/\alpha_1}}^{1} \int_{0}^{h_1^{1/b_1}} |h|^{-2-\gamma} \, \d h_2\, \d h_1 
+ 4(2-\alpha_{\max})\int_{r^{\alpha_{\max}/\alpha_1}}^{1} \int_{h_1^{b_2}}^{\infty} |h|^{-2-\gamma} \, \d h_2\, \d h_1 \\
&\qquad+ (2-\alpha_{\max})\int_{\R^2\setminus [-1,1]^2} |h|^{-2-\gamma} \, \d h.\\
& = 4(2-\alpha_{\max}) I_1 + 4(2-\alpha_{\max}) I_2 + (2-\alpha_{\max}) I_3.
\end{align*}
Let us first study $I_1$. We have, 
\begin{align*}
I_1  \leq \int_{r^{\alpha_{\max}/\alpha_1}}^{\infty} \int_{0}^{h_1^{1/b_1}} h_1^{-2-\gamma} \, \d h_2\, \d h_1 = \int_{r^{\alpha_{\max}/\alpha_1}}^{\infty} h_1^{-1-\alpha_1} \, \d h_1
 = \frac{1}{\alpha_1} r^{-\alpha_{\max}}.
\end{align*}
Furthermore, the term $I_2$ can be estimated by using the fact that $h_1\in[r^{\alpha_{\max}/\alpha_1},1]$, which means in particular $h_1\leq 1$. 
\begin{align*}
I_2 &  \leq \int_{r^{\alpha_{\max}/\alpha_1}}^{\infty} \int_{h_1^{b_2}}^{\infty} h_2^{-2-\gamma} \, \d h_2\, \d h_1 = \frac{1}{1+\gamma} \int_{r^{\alpha_{\max}/\alpha_1}}^{\infty} h_1^{(-1-\gamma)b_2} \,  \d h_1  \\
& \leq \frac{1}{1+\gamma} \int_{r^{\alpha_{\max}/\alpha_1}}^{\infty} h_1^{-1-\alpha_1} \,  \d h_1 
 = \frac{1}{\alpha_1(1+\gamma)}r^{-\alpha_{\max}},
\end{align*}
where we used the inequality $b_2(-1-\gamma)\geq -1-\alpha_1$ from \eqref{eq:bgamma}.\\
It remains to estimate $I_3$. The formula for integration of rotational symmetric functions leads to the existence of a uniform constant $c_1>0$ such that
\begin{align*}
I_3 = \int_{\R^2\setminus [-1,1]^2} |h|^{-2-\gamma} \, \d h \leq c_1 \int_{1}^{\infty} s^{-\gamma-1}\, \d s = \frac{c_1}{\gamma} \leq  \frac{c_1}{\gamma} r^{-\alpha_{\max}}.
\end{align*}

Hence, there is a constant $c>0$, depending on $\alpha_0$, such that
\begin{align*}
\mu(x_0,\R^2\setminus E_r^1(x_0)) \leq c(2-\alpha_{\max}) r^{-\alpha_{\max}}, 
\end{align*}
which proves assertion \eqref{eq:tailbd}.

Next, we verify \autoref{assumption:func-ineq}. Before we address the proof, let us give a decomposition of the set $M_r \cap \Gamma(x)$. 
We define for given $r>0$ and $x\in\R^2$
\begin{align*}
A^r(x)& := \{ z\in M_r\colon |z_2-x_2|\leq |z_1-x_1|^{1/b_1}\}\cap \{ z\in\R^2\setminus M_1 \colon |z_2-x_2|\leq |z_1-x_1| \},\\
\underline{A}^r(x)& : = \{ z \in A^r(x) : z_1 \le x_1 \},~ \overline{A}^r(x) = A^r(x) \setminus \underline{A}^r(x),\\
B^r(x)& :=\{ z\in M_r\colon |z_1-x_1|\leq |z_2-x_2|^{1/b_2}\}\cap \{ z\in\R^2\setminus M_1 \colon |z_1-x_1|\leq |z_2-x_2| \},\\
\underline{B}^r(x)& : = \{ z \in B^r(x) : z_2 \le x_2 \},~ \overline{B}^r(x) = A^r(x) \setminus \underline{B}^r(x).
\end{align*}

 \begin{figure}[H]
 \begin{minipage}{0.45\textwidth}
\resizebox{\textwidth}{!}{%
\includegraphics{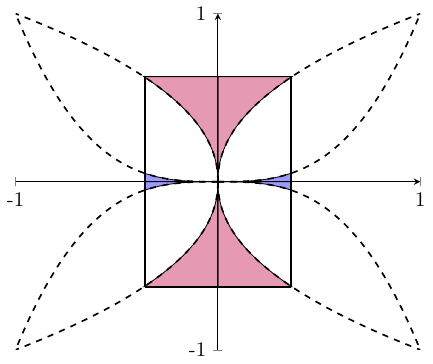}}
\end{minipage}
 \begin{minipage}{0.45\textwidth}
\resizebox{\textwidth}{!}{%
\includegraphics{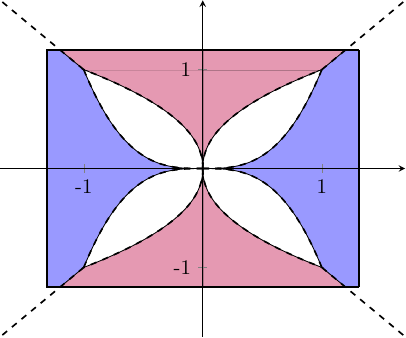}}
\end{minipage} {\ } \\

 \begin{minipage}{0.45\textwidth}
\resizebox{\textwidth}{!}{%
\includegraphics{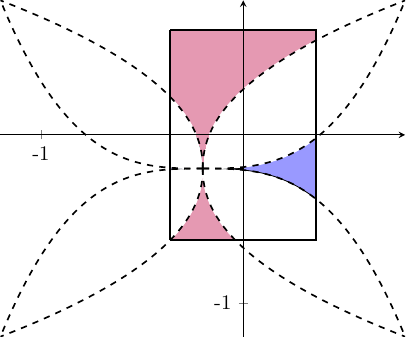}}
\end{minipage}
 \begin{minipage}{0.45\textwidth}
\resizebox{\textwidth}{!}{%
\includegraphics{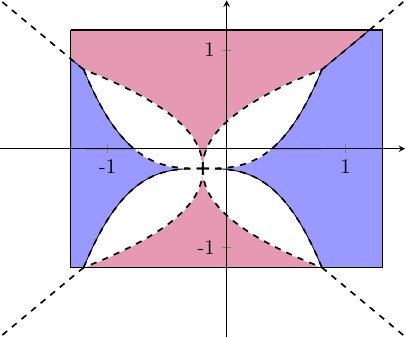}}
\end{minipage}
\caption{This figure illustrates the sets $A^r(0)$, $B^r(0)$ (top) resp. $A^r(x)$, $B^r(x)$ (bottom) for $x=(-0.2,-0.2)$ and some $r<1$ on the left and $r>1$ on the right. 
The boundary of the set $M_r$ is given by the black edged rectangle. 
The set $A^r$ resp. $A^r(x)$ is given by the blue colored set, while the red colored set shows $B^r$ resp. $B^r(x)$.
As we can see, the union of sets $A^r(0)$ and $B^r(0)$ resp. $A^r(x)$ and $B^r(x)$ is the intersection of $M_r$ with $\Gamma(0)$
resp. $M_r$ with $\Gamma(x)$.}
 \end{figure}
 {\ } \\
 
We start by proving the Sobolev inequality \eqref{eq:sobolev-assum}. Without loss of generality, we assume $x_0=0$. Let $r \in (0,3]$, $\lambda > 1,$ and $u\in H(M_{\lambda r})$. Let $\tau = \tau_{r,\lambda}$ be as in \autoref{cutofffct}. Then $u \equiv \tau u$ in $M_r$. We have by the Plancherel formula and Fubini's theorem
\begin{align*}
\mathcal{E}_{M_{r}}^{\ma} (u,u) & = \int_{M_r} \int_{M_r} (u(x)-u(y))^2\, \ma(x,\d y) \, \d x \\
& \le \int_{\R^2} \int_{M_{\lambda r}} (\tau u(x)-\tau u(x+z))^2 \ma(0,\d z) \, \d x \\
& \le \int_{M_{\lambda r}} \int_{\R^2} 4 \sin^2\left(\frac{\xi \cdot z}{2}\right) \vert \widehat{\tau u}(\xi)\vert^2 \d \xi \, \ma(0, \d z)\\
& \le c \sum_{k = 1}^2 \int_{\R^2} \left( \int_{0}^{(\lambda r)^{\am/\alpha_k}} \sin^2\left(\frac{\xi_k z_k}{2}\right) \vert z_k \vert^{-1-\alpha_k} \d z_k \right) \vert \widehat{\tau u}(\xi)\vert^2 \d \xi \\
& = c\sum_{k = 1}^2 \int_{\R^2} J_k(\xi) \vert \widehat{\tau u}(\xi)\vert^2 \d \xi,
\end{align*}
where $c > 0$ is a constant and
\begin{equation}
J_k(\xi) = \int_{0}^{(\lambda r)^{\am/\alpha_k}} \sin^2\left(\frac{\xi_k z_k}{2}\right) \vert z_k \vert^{-1-\alpha_k} \d z_k, ~~ \xi \in \R^2, ~k = 1,2.
\end{equation}

Note that for every $a > 0$, $\xi \in \R^2$ and $z_1 \in (0,(\lambda r)^{\am/\alpha_1})$ it holds:
\begin{equation}
\label{eq:cusphelp}
\sin^2\left(\frac{\xi_1 z_1}{2}\right) \le \frac{2}{a} \int_{-a}^a \sin^2\left(\frac{\xi_1 z_1 + \xi_2 z_2 }{2}\right) \d z_2.
\end{equation}
This follows from $\sin^2(x/2) \le 2\left(1 - \cos(x)\frac{\sin{y}}{y} \right)$ and the fact that
\begin{equation}
1 - \cos(\xi_1 z_1)\frac{\sin(\xi_2 a)}{\xi_2 a} = \frac{1}{a} \int_{-a}^a \sin^2 \left(\frac{\xi_1 z_1 + \xi_2 z_2}{2} \right) \d z_2.
\end{equation}

We want to apply \eqref{eq:cusphelp} with $a = z_1^{1/b_1} \wedge z_1$ for $z_1\in (0,(\lambda r)^{\am/\alpha_1}]$. Note that there is a constant $c_1 > 1$ such that $\vert z_1 \vert^{-1} \le c_1 \vert z_1\vert^{-1/b_1}$ for all $z_1$ and therefore $1/a \le c_1 \vert z_1\vert^{-1/b_1}$. 
Using \eqref{eq:cusphelp}, we compute
\begin{align*}
J_1(\xi) & = \int_{0}^{(\lambda r)^{\am/\alpha_1}} \sin^2\left(\frac{\xi_1 z_1}{2}\right) \vert z_1 \vert^{-1-\alpha_1} \d z_1 \\
& \le c_1 \int_{0}^{(\lambda r)^{\am/\alpha_1}}  \int_{-(z_1^{1/b_1} \wedge z_1)}^{z_1^{1/b_1} \wedge z_1} \sin^2\left(\frac{\xi \cdot z}{2}\right) \vert z_1\vert ^{-1/b_1} \vert z_1 \vert^{-1-\alpha_1} \d z\\
& = c_1\int_{0}^{(\lambda r)^{\am/\alpha_1}}  \int_{-(z_1^{1/b_1} \wedge z_1)}^{z_1^{1/b_1} \wedge z_1} \sin^2\left(\frac{\xi \cdot z}{2}\right) \vert z_1 \vert^{-2-\gamma} \d z \\
& \le c_1\int_{A^{\lambda r}(0)} \sin^2\left(\frac{\xi \cdot z}{2}\right) \vert z \vert^{-2-\gamma} \d z.
\end{align*}
In the ultimate step, we used that $\vert z_2 \vert \le \vert z_1 \vert$ in $A^{\lambda r}(0)$.\\
Analogously, one establishes 
\begin{align*}
J_2(\xi) \le c_2\int_{B^{\lambda r}(0)} \sin^2\left(\frac{\xi \cdot z}{2}\right) \vert z \vert^{-2-\gamma} \d z,
\end{align*}
where $c_2 > 0$ is some constant.
Together, we obtain, using again Fubini's theorem and the Plancherel formula
\begin{align}
\label{eq:cusp-Sob-help-1}
\begin{split}
\mathcal{E}_{M_{r}}^{\ma} (u,u) & \le c_3 \int_{\R^2} \left(\int_{\Gamma(0)} \sin^2\left(\frac{\xi \cdot z}{2}\right) \vert z \vert^{-2-\gamma} \d z \right) \vert \widehat{\tau u}(\xi)\vert^2 \d \xi \\
& = c_4 \int_{\R^2} \int_{\R^2} 4\sin^2\left(\frac{\xi \cdot z}{2}\right)\vert \widehat{\tau u}(\xi)\vert^2 \d \xi \, \mu(0,\d z) \\
& = c_4 \cE^{\mu}(\tau u,\tau u),
\end{split}
\end{align}
with some constants $c_3,c_4 > 0$. 
Note that we can estimate $\cE^{\mu}(\tau u,\tau u)$ as follows, using \autoref{thmcutoffest}:
\begin{align}
\label{eq:cusp-Sob-help-2}
\cE^{\mu}(\tau u, \tau u) \le c_5 \cE^{\mu}_{M_{\lambda r}}(u, u) + c_5 r^{-\am} \sum_{k = 1}^2 (\lambda^{\frac{\am}{\alpha_k}}-1)^{-\alpha_k} \Vert u \Vert_{L^2(M_{\lambda r})}^2,
\end{align}
where $c_5 > 0$ is a constant. 
%In the lines above we applied analogues of \cite[Lemma 5.1 \& Lemma 5.3]{NPV12} with $\mu$ instead of the $\alpha$-stable kernel, as well as the following observation:
%\begin{equation}
%\Vert \tau u \Vert_{L^2(\R^d)}^2 = \Vert \tau u \Vert_{L^2(M_{\lambda r})}^2 \le \Vert u \Vert_{L^2(M_{\lambda r})}^2.
%\end{equation}
%The adaptation of \cite[Lemma 5.1 \& Lemma 5.3]{NPV12} to $\mu$ is straightforward.\\
By combining \eqref{eq:cusp-Sob-help-1} and \eqref{eq:cusp-Sob-help-2}  with the Sobolev inequality for $\ma$ (see \cite[Corollary 2.6]{ChKa20}), we obtain \eqref{eq:sobolev-assum}, as desired.

Next, we show the Poincar\'e inequality \eqref{eq:poincare-assum}. First, Jensen's inequality yields
\begin{align*}
\Vert u - [u]_{M_{\rho}} \Vert_{L^2(M_{\rho})}^2 &\le \int_{M_{\rho}} \dashint_{M_{\rho}} (u(x)-u(y))^2\d y \d x\\
& = \int_{M_{\rho}} \dashint_{M_{\rho}} \dashint_{M_{7\rho} \cap \Gamma(x) \cap \Gamma(y)} (u(x)-u(z)+u(z)-u(y))^2 \d z\d y \d x\\
& \le \frac{2}{\vert M_{7\rho} \cap \Gamma(x) \cap \Gamma(y) \vert} \Bigg( \int_{M_{\rho}} \int_{M_{7\rho} \cap \Gamma(x)} (u(x)-u(z))^2 \d z \d x \\
& \hspace{15em} + \int_{M_{7\rho}} \int_{M_{\rho} \cap \Gamma(y)} (u(y)-u(z))^2 \d z \d y \Bigg)\\
& \le \frac{c\rho^{2+\gamma}}{\vert M_{7\rho} \cap \Gamma(x) \cap \Gamma(y) \vert} \cE^{\mu}_{M_{7\rho}}(u,u).
\end{align*}
Next, we prove that 
\begin{align}
\label{eq:cusp-set-estimate}
\inf_{x,y \in M_{\rho}} \vert M_{7\rho} \cap \Gamma(x) \cap \Gamma(y) \vert \ge c\rho^{2+\gamma-\am}.
\end{align}
Without loss of generality we assume that $\alpha_1 \le \alpha_2$. Then $b_2 = \alpha_1 / \am$, $b_2 \ge b_1$, and $2+\gamma - \am = 1 + \frac{1}{b_2}$. Let $x,y \in M_{\rho}$ with $x_1 \le y_1$ and $x_2 \le y_2$. One can handle the other cases in the same way by changing roles of $x,y$ and considering $\overline{B}$ instead of $\underline{B}$ in the following argument. Note that
\begin{equation}
\vert \underline{B}^{7\rho}(x) \cap \underline{B}^{7\rho}((y_1,x_2))\vert \le \vert \underline{B}^{7\rho}(x) \cap \underline{B}^{7\rho}(y)\vert \le \vert M_{7\rho} \cap \Gamma(x) \cap \Gamma(y) \vert,
\end{equation}
so it is enough to assume that $x_2 = y_2$ and prove that $\vert \underline{B}^{7\rho}(x) \cap \underline{B}^{7\rho}(y)\vert \ge c \rho^{2+\gamma-\am}$. We compute:
\begin{align*}
\vert \underline{B}^{7\rho}(x) & \cap \underline{B}^{7\rho}(y)\vert \\
&  \ge\int_{-\rho^{1/b_2}}^{x_1} \left((7\rho + x_2) - (x_1 - w)^{b_2}\right)\d w + \int_{x_1}^{\frac{x_1+y_1}{2}} \left((7\rho + x_2) - (w - x_1)^{b_2}\right) \d w\\
&\qquad + \int_{\frac{x_1+y_1}{2}}^{y_1} \left((7\rho + x_2) - (y_1 - w)^{b_2}\right) \d w + \int_{y_1}^{\rho^{1/b_2}} \left((7\rho + x_2) - (w - y_1)^{b_2}\right) \d w\\
& = 2\rho^{1/b_2}(7\rho + x_2) - \frac{1}{1+b_2}\left((x_1 + \rho^{1/b_2})^{1+b_2} + 2\left(\frac{y_1 -x_1}{2} \right)^{1+b_2} + (\rho^{1/b_2}-y_1)^{1+b_2} \right)\\
& \ge 2\rho^{1/b_2}(7\rho + x_2) - \frac{1}{1+b_2}\left( 2^{1+b_2} + 2 + 2^{1+b_2} \right)\rho^{1+1/b_2}\\
& \ge \left(12 - \frac{2}{1+b_2}(1+2^{1+b_2}) \right)\rho^{1+ 1/b_2}\\
& \ge 2\rho^{2 + \gamma - \am}.
\end{align*}
In the second step, we used that from $x,y \in M_{\rho}$ it follows that: $7\rho + x_2 \ge 6\rho$, $x_1 \le \rho^{1/b_2}$, $y_1 - x_1 \le 2\rho^{1/b_2}$ and $y_1 \ge -\rho^{1/b_2}$. This concludes the proof of \eqref{eq:cusp-set-estimate}.\\ Thus, we have shown the following weak Poincar\'e inequality:
\begin{align}
\label{eq:weak-Poincare}
\Vert u - [u]_{M_{\rho}} \Vert_{L^2(M_{\rho})}^2 \le c \rho^{\am}  \cE^{\mu}_{M_{7\rho}}(u,u).
\end{align}
Weak Poincar\'e inequalities can be ''upgraded`` to actual Poincar\'e inequalities as \eqref{eq:poincare-assum} with the help of a Whitney covering argument. This is a very general procedure which works in the same way for local and nonlocal energies on doubling metric measure spaces and is explained in detail in \cite[Proof of Corollary 5.3.5]{Sal02}. For the convenience of the reader, we state (and prove) such result in a very general context in \autoref{lemma:weak-implies-strong}. In our case, the metric measure space is given by $(\R^2,m, \lambda^2)$, where $m(x,y) = \max_{k = 1,2} \left(|x_k - y_k|^{\frac{\alpha_k}{\am}}\mathbbm{1}_{|x_k - y_k| \le 1}(x,y) + \mathbbm{1}_{|x_k - y_k| > 1}(x,y) \right)$ denotes the underlying metric, and $\lambda^2$ stands for  the two-dimensional Lebesgue measure. The corresponding metric balls are given by the cubes $M_r(x)$. 
%We believe the changes that are needed in order to adapt the proof in \cite{Sal02} to our setting to be very straightforward and decided not to provide any further details. 
This concludes the proof of \eqref{eq:poincare-assum}.

It remains to prove \eqref{eq:fct-space-comp}. We have
\begin{align*}
 \mathcal{E}_{M_r}^{\mu}(u,u) & = C(\gamma,b)\int_{M_r}\int_{M_r\cap\Gamma(x)} (u(x)-u(y))^2\, |x-y|^{-2-\gamma} \d y \, \d x \\
 & \leq  2C(\gamma,b)\int_{M_r}\int_{M_r\cap\Gamma(x)}  \left((u(x)-u(x_1,y_2))^2 + (u(x_1,y_2) - u(y))^2\right) |x-y|^{-2-\gamma} \d y \, \d x \\
 & = 2C(\gamma,b)(I_1 + I_2).
\end{align*}
We can estimate $I_1$ as follows: $I_1\leq (A+B+C)$, by enlarging the area of integration $M_r\cap\Gamma(x)$ for the inner integral, as follows:
\begin{align*}
&A = \int_{M_r}\int_{-r^{\amax/\alpha_2}}^{r^{\amax/\alpha_2}} \int_{x_1-|x_2-y_2|^{1/b_2}}^{x_1+|x_2-y_2|^{1/b_2}} (u(x)-u(x+e_2(y_2-x_2)))^2 |x-y|^{-2-\gamma} \d y \, \d x, \\
&B = \int_{M_r}\int_{-r^{\amax/\alpha_2}}^{r^{\amax/\alpha_2}} \int_{x_1+|x_2-y_2|^{b_1}}^{\infty} (u(x)-u(x+e_2(y_2-x_2)))^2 |x-y|^{-2-\gamma} \d y \, \d x, \\
&C = \int_{M_r}\int_{-r^{\amax/\alpha_2}}^{r^{\amax/\alpha_2}} \int_{-\infty}^{x_1-|x_2-y_2|^{b_1}} (u(x)-u(x+e_2(y_2-x_2)))^2 |x-y|^{-2-\gamma} \d y \, \d x.
\end{align*}
We start by estimating $A$. For this purpose, we study the innermost integral, that is the integral with respect to $y_1$. 
By symmetry, it is sufficient to consider the area of integration $[x_1,x_1+|y_2-x_2|^{1/b_2}]$ for $y_1$, where $x\in M_r$ and $y_2\in [-r^{\amax/\alpha_2},-r^{\amax/\alpha_2}]$.
Let $h_1=y_1-x_1$ and $h_2=y_2-x_2$. Then
\begin{align*}
 \int_{x_1}^{x_1+|y_2-x_2|^{1/b_2}} |x-y|^{-2-\gamma} \, \d y_1 \leq \int_{x_1}^{x_1+|y_2-x_2|^{1/b_2}} |x_2-y_2|^{-2-\gamma} \, \d y_1 =  |x_2-y_2|^{-1-\alpha_2}.
 \end{align*}
 Hence
\[  A \leq 2\int_{M_r}\int_{-r^{\amax/\alpha_2}}^{r^{\amax/\alpha_2}} (u(x)-u(x+e_2(y_2-x_2)))^2 |x_2-y_2|^{-1-\alpha_2} \d y_2 \, \d x. \]
Again, we study the innermost integral to investigate $B$. Using \eqref{eq:bgamma}, leads to
\begin{align*}
 & \int_{x_1+|y_2-x_2|^{b_1}}^{\infty} |x-y|^{-2-\gamma} \, \d y_1 \leq \int_{x_1+|y_2-x_2|^{b_1}}^{\infty} |x_1-y_1|^{-2-\gamma} \, \d y_1  \\ & = \frac{1}{1+\gamma} \left( |h_2|^{b_1(-1-\gamma)} \right) \leq |h_2|^{-1-\alpha_2} = |x_2-y_2|^{-1-\alpha_2},
\end{align*}
and therefore
\[ B \leq \int_{M_r}\int_{-r^{\amax/\alpha_2}}^{r^{\amax/\alpha_2}} (u(x)-u(x+e_2(y_2-x_2)))^2 |x_2-y_2|^{-1-\alpha_2} \d y_2 \, \d x. \]
The term $C$ can be estimated in the same way as $B$.
Altogether, 
\[ I_1 \leq 4\int_{M_r}\int_{-r^{\amax/\alpha_2}}^{r^{\amax/\alpha_2}} (u(x)-u(x+e_2(y_2-x_2)))^2 |x_2-y_2|^{-1-\alpha_2} \d y_2 \, \d x, \]
and, since we can use the same arguments for $I_2$ as for $I_1$,
\[ I_2 \leq 4\int_{M_r}\int_{-r^{\amax/\alpha_1}}^{r^{\amax/\alpha_1}} (u(x)-u(x+e_1(y_1-x_1)))^2 |x_1-y_1|^{-1-\alpha_1} \d y_1 \, \d x. \]
Since $C(\gamma,b)\leq (2-\alpha_k)$ for $k\in\{1,2\}$, there is a universal constant $c>0$ such that
\begin{align*}
 \mathcal{E}_{M_r}^{\mu}(u,u) \leq 2(2-\alpha_1)I_1 + 2(2-\alpha_2)I_2 \leq  c\mathcal{E}_{M_r}^{\ma}(u,u).
\end{align*}
\end{proof}

\section{Appendix: Weak implies strong Poincar\'e inequality}
\label{sec:Sal}

In this section we prove that, in general doubling metric measure spaces, a weak Poincar\'e inequality (i.e. a Poincar\'e inequality with larger balls on the right hand side) implies a strong Poincar\'e inequality of the form \eqref{eq:poincare-assum}. The proof follows from a Whitney covering argument and is a direct adaptation of \cite[Corollary 5.3.5]{Sal02}. The result is used in the proof of \autoref{thm:comparability}.

Let $(X,d,\lambda)$ be a metric measure space. Let us denote the metric balls by $M_r(x_0)$ for $x_0 \in X$ and $r > 0$. We assume that $\mu$ is doubling, i.e. there exists $c_0 > 0$ such that $\lambda(M_{2r}(x_0)) \le c_0 \lambda(M_r(x_0))$ for any $r > 0$ and $x_0 \in X$. Let $(\mu(x,\cdot))_{x \in X}$ be a family of measures on $X$ satisfying
\begin{align*}
\sup_{x \in X} \int_{X} \min(1,d(x,y)^2) \mu(x,\d y) < \infty.
\end{align*}
For $M \subset X$ and $u \in L^2(M)$, we define
\begin{align*}
\cE^{\mu}_{M}(u,u) = \int_{M} \int_{M} (u(x) - u(y))^2 \mu(x,\d y) \lambda(\d x), \qquad [u]_{M} = \dashint_{M} u(x) \lambda(\d x).
\end{align*}

\begin{lemma}
\label{lemma:weak-implies-strong}
Let $(X,d,\lambda)$ be a doubling metric measure space.  Moreover, assume that there are $c_1 > 0$, $A > 1$, and $\alpha > 0$ such that for any $r \in (0,1)$ and $x_0 \in X$ it holds
\begin{align}
\label{eq:weak-Poincare-abstract}
\int_{M_r(x_0)} (u(x) - [u]_{M_r(x_0)})^2 \lambda(\d x) \le c_1 r^{\alpha} \cE^{\mu}_{M_{Ar}(x_0)}(u,u) ~~ \forall u \in L^2(M_r(x_0)).
\end{align}
Then, there exists $c > 0$, depending only on $c_0,c_1,A$, such that for any $r \in (0,1)$ and $x_0 \in X$:
\begin{align*}
\int_{M_r(x_0)} (u(x) - [u]_{M_r(x_0)})^2 \lambda(\d x) \le c r^{\alpha} \cE^{\mu}_{M_{r}(x_0)}(u,u) ~~ \forall u \in L^2(M_r(x_0)).
\end{align*}
\end{lemma}

\begin{proof}
Let us fix $M_r(x_0)$ for some $0 < r < R < 1$. We closely follow the arguments in \cite[Theorem 5.3.4]{Sal02} where we set $\Phi = \mathbbm{1}_{M_R(x_0)}$. First, we consider a covering $\mathcal{F}$ of disjoint metric balls $M$ such that $M_r(x_0) = \bigcup_{M \in \cF} 2M$, $10^3M \subset M_r(x_0)$ and 
\begin{align}
\label{eq:covering-prop-0}
\sup_{\eta \in M_r(x_0)} |\{M \in \cF : \eta \in 100M \}| \le K < \infty,
\end{align}
where $K > 0$ depends only on $d,\alpha_1,\alpha_2$.
The construction goes in the same way as in \cite[p. 135]{Sal02} and only requires a doubling metric measure space. Moreover, one can show that for any $M \in \cF$, there exists a string of balls $\cF(M) = (M_0 = M_{x_0},M_1,\dots,M_{l(M)-1}=M)$ such that $\overline{2M_i} \cap \overline{2 M_{i+1}} \not= \emptyset$, where $M_{x_0} \in \cF$ is such that $x_0 \in M_{x_0}$. Any two consecutive balls $M_i,M_{i+1}$ satisfy 
\begin{align}
\label{eq:covering-prop}
r(M_i) \asymp r(M_{i+1}), \qquad M_{i+1} \subset 4M_i, \qquad |4M_i \cap 4M_{i+1}| \ge c \max(|M_i|,|M_{i+1}|),
\end{align}
and moreover, for any $A \in \cF(M)$, it holds $M \subset 10^4 A$.\\
First, by the triangle inequality:
\begin{align*}
\int_{M_r(x_0)} |u - [u]_{4M_{x_0}}|^2 \lambda(\d x) &\le \sum_{M \in \cF} \int_{2M} |u - [u]_{4M_{x_0}}|^2 \lambda(\d x) \\
&\le c \sum_{M \in \cF} \int_{2M} |u - [u]_{4M}|^2 \lambda(\d x) + c \sum_{M \in \cF} \lambda(4M) |[u]_{4M} - [u]_{4M_{x_0}}|^2 \\
&= I_1 + I_2.
\end{align*}
Using \eqref{eq:weak-Poincare}, and \eqref{eq:covering-prop-0}, we deduce
\begin{align*}
I_1 \le c\sum_{M \in \cF} r(4M)^{\alpha} \cE_{8M}^{\mu}(u,u) \le c r^{\alpha} \cE_{M_r(x_0)}^{\mu}(u,u).
\end{align*}
Moreover, to bound $I_2$, we observe that using \eqref{eq:covering-prop}, \eqref{eq:weak-Poincare-abstract}, and the triangle inequality, one derives as in \cite[Lemma 5.3.9]{Sal02} for any $M \in \cF$ and $M_i,M_{i+1} \in \cF(M)$:
\begin{align*}
|[u]_{4M_i} - [u]_{4M_{i+1}}| \le c \frac{r(M)^{\alpha/2}}{\lambda(M)^{1/2}} \cE^{\mu}_{32M}(u,u)^{1/2}.
\end{align*}
Therefore, using \eqref{eq:covering-prop} as in \cite[p. 142]{Sal02}:
\begin{align*}
|[u]_{4M} - [u]_{4M_x}| \mathbbm{1}_{M} &\le \sum_{ i = 1}^{l(M)-1} |[u]_{4M_i} - [u]_{4M_{i+1}}| \mathbbm{1}_{M} \le c \sum_{ i = 1}^{l(M)-1} \frac{r(M_i)^{\alpha/2}}{\lambda(M)^{1/2}} \cE^{\mu}_{32M}(u,u)^{1/2} \mathbbm{1}_{M} \\
&\le c \sum_{A \in \cF} \frac{r(A)^{\alpha/2}}{\lambda(A)^{1/2}} \cE^{\mu}_{32A}(u,u)^{1/2} \mathbbm{1}_{10^4 A} \mathbbm{1}_{M},
\end{align*}
and since the $M \in \cF$ are disjoint,
\begin{align*}
\sum_{M \in \cF} |[u]_{4M} - [u]_{4M_x}|^2 \mathbbm{1}_{M} \le c \left( \sum_{A \in \cF} \frac{r(A)^{\alpha/2}}{\lambda(A)^{1/2}} \cE^{\mu}_{32A}(u,u)^{1/2} \mathbbm{1}_{10^5 A} \right)^{2}. 
\end{align*}
Next, using \cite[Lemma 5.3.12]{Sal02} in the same way as in \cite{Sal02}, and also \eqref{eq:covering-prop-0}, we obtain 
\begin{align*}
I_2 = \int_{M_r(x_0)} \sum_{M \in \cF} |[u]_{4M} - [u]_{4M_x}|^2 \mathbbm{1}_{M} \lambda(\d x) &\le  c \int_{M_r(x_0)} \left( \sum_{A \in \cF} \frac{r(A)^{\alpha/2}}{\lambda(A)^{1/2}} \cE^{\mu}_{32A}(u,u)^{1/2} \mathbbm{1}_{A} \right)^2 \lambda(\d x)\\
&\le c \int_{M_r(x_0)} \left( \sum_{A \in \cF} \frac{r(A)^{\alpha}}{\lambda(A)} \cE^{\mu}_{32A}(u,u) \mathbbm{1}_{A} \right) \lambda(\d x) \\
&\le c r^{\alpha} \sum_{A \in \cF} \cE^{\mu}_{32A}(u,u)\\
&\le c r^{\alpha} \cE^{\mu}_{M_r(x_0)}(u,u).
\end{align*}
Altogether, we have shown
\begin{align*}
\int_{M_r(x_0)} |u - [u]_{4M_{x_0}}|^2 \lambda(\d x) \le c r^{\alpha} \cE^{\mu}_{M_r(x_0)}(u,u).
\end{align*}
Therefore,
\begin{align*}
\int_{M_r(x_0)} |u - [u]_{M_r(x_0)}|^2 \lambda(\d x) &\le 2\int_{M_r(x_0)} |u - [u]_{4 M_{x_0}}|^2 \lambda(\d x) + 2 |[u]_{4M_{x_0}} - [u]_{M_r(x_0)}|^2  \\
&\le c \int_{M_r(x_0)} |u - [u]_{4 M_{x_0}}|^2 \lambda(\d x)\\
&\le c r^{\alpha} \cE^{\mu}_{M_r(x_0)}(u,u),
\end{align*}
which proves the desired result.
\end{proof}

%\bibliographystyle{alpha}
%\bibliography{reg-nonlocal_literature}

\begin{thebibliography}{BDVV22}
	
	\bibitem[AC23]{AnCo23}
	Carlo~Alberto Antonini and Matteo Cozzi.
	\newblock Global gradient regularity and a {H}opf lemma for quasilinear
	operators of mixed local-nonlocal type.
	\newblock {\em arXiv:2308.06075}, 2023.
	
	\bibitem[APT23]{APT23}
	Karthik Adimurthi, Harsh Prasad, and Vivek Tewary.
	\newblock Gradient regularity for mixed local-nonlocal quasilinear parabolic
	equations.
	\newblock {\em arXiv:2307.02363}, 2023.
	
	\bibitem[BC06]{BaCh06}
	Richard~F. Bass and Zhen-Qing Chen.
	\newblock Systems of equations driven by stable processes.
	\newblock {\em Probab. Theory Related Fields}, 134(2):175--214, 2006.
	
	\bibitem[BC10]{BaCh10}
	Richard~F. Bass and Zhen-Qing Chen.
	\newblock Regularity of harmonic functions for a class of singular stable-like
	processes.
	\newblock {\em Math. Z.}, 266(3):489--503, 2010.
	
	\bibitem[BDVV21]{BDVV21}
	Stefano Biagi, Serena Dipierro, Enrico Valdinoci, and Eugenio Vecchi.
	\newblock Semilinear elliptic equations involving mixed local and nonlocal
	operators.
	\newblock {\em Proc. R. Soc. Edinburgh Sect.}, 151(5):1611--1641, 2021.
	
	\bibitem[BDVV22]{BDVV22}
	Stefano Biagi, Serena Dipierro, Enrico Valdinoci, and Eugenio Vecchi.
	\newblock Mixed local and nonlocal elliptic operators: regularity and maximum
	principles.
	\newblock {\em Comm. Partial Differential Equations}, 47(3):585--629, 2022.
	
	\bibitem[BL02]{BaLe02}
	Richard~F. Bass and David~A. Levin.
	\newblock Transition probabilities for symmetric jump processes.
	\newblock {\em Trans. Amer. Math. Soc.}, 354(7):2933--2953, 2002.
	
	\bibitem[BLS23]{BLS23}
	Sun-Sig Byun, Ho-Sik Lee, and Kyeong Song.
	\newblock Regularity results for mixed local and nonlocal double phase
	functionals.
	\newblock {\em arXiv:2301.06234}, 2023.
	
	\bibitem[BMP23]{BMP23}
	Stefano Biagi, Giulia Meglioli, and Fabio Punzo.
	\newblock Uniqueness for local-nonlocal elliptic equations.
	\newblock {\em arXiv:2307.02209}, 2023.
	
	\bibitem[BMS23]{BMS23}
	Anup Biswas, Mitesh Modasiya, and Abhrojyoti Sen.
	\newblock Boundary regularity of mixed local-nonlocal operators and its
	application.
	\newblock {\em Ann. Mat. Pura Appl.}, 202(2):679--710, 2023.
	
	\bibitem[BP23]{BiPr23}
	Nirjan Biswas and Harsh Prasad.
	\newblock Lipschitz potential estimates for diffusion with jumps.
	\newblock {\em arXiv:2307.02803}, 2023.
	
	\bibitem[BS23]{BySo23}
	Sun-Sig Byun and Kyeong Song.
	\newblock Mixed local and nonlocal equations with measure data.
	\newblock {\em Calc. Var. Partial Differential Equations}, 62(1):14, 2023.
	
	\bibitem[Cha19]{Cha19}
	Jamil Chaker.
	\newblock The martingale problem for a class of nonlocal operators of diagonal
	type.
	\newblock {\em Math. Nachr.}, 292(11):2316--2337, 2019.
	
	\bibitem[Cha20]{Cha16}
	Jamil Chaker.
	\newblock Regularity of solutions to anisotropic nonlocal equations.
	\newblock {\em Math. Z.}, 296(3-4):1135--1155, 2020.
	
	\bibitem[CK03]{ChKu03}
	Zhen-Qing Chen and Takashi Kumagai.
	\newblock Heat kernel estimates for stable-like processes on {$d$}-sets.
	\newblock {\em Stochastic Process. Appl.}, 108(1):27--62, 2003.
	
	\bibitem[CK20]{ChKa20}
	Jamil Chaker and Moritz Kassmann.
	\newblock Nonlocal operators with singular anisotropic kernels.
	\newblock {\em Comm. Partial Differential Equations}, 45(1):1--31, 2020.
	
	\bibitem[CK22]{ChKi22}
	Jamil Chaker and Minhyun Kim.
	\newblock Regularity estimates for fractional orthotropic p-{L}aplacians of
	mixed order.
	\newblock {\em Adv. in Nonlinear Anal.}, 11(1):1307--1331, 2022.
	
	\bibitem[CKW23]{CKW23}
	Jamil Chaker, Minhyun Kim, and Marvin Weidner.
	\newblock The concentration-compactness principle for the nonlocal anisotropic
	p-{L}aplacian of mixed order.
	\newblock {\em Nonlinear Anal.}, 232:113254, 2023.
	
	\bibitem[CZZ21]{CZZ17}
	Zhen-Qing Chen, Xicheng Zhang, and Guohuan Zhao.
	\newblock Supercritical {SDE}s driven by multiplicative stable-like {L}\'{e}vy
	processes.
	\newblock {\em Trans. Amer. Math. Soc.}, 374(11):7621--7655, 2021.
	
	\bibitem[Das23]{Das23}
	Stuti Das.
	\newblock Gradient {H}{\"o}lder regularity in mixed local and nonlocal linear
	parabolic problem.
	\newblock {\em arXiv:2306.07021}, 2023.
	
	\bibitem[DB23]{DuBo23}
	Ignacio~Ceresa Dussel and Juli{\'a}n~Fern{\'a}ndez Bonder.
	\newblock A {B}ourgain-{B}rezis-{M}ironescu formula for anisotropic fractional
	{S}obolev spaces and applications to anisotropic fractional differential
	equations.
	\newblock {\em J. Math. Anal. Appl.}, 519(2):126805, 2023.
	
	\bibitem[DCKP14]{DKP14}
	Agnese Di~Castro, Tuomo Kuusi, and Giampiero Palatucci.
	\newblock Nonlocal {H}arnack inequalities.
	\newblock {\em J. Funct. Anal.}, 267(6):1807--1836, 2014.
	
	\bibitem[DCKP16]{DKP16}
	Agnese Di~Castro, Tuomo Kuusi, and Giampiero Palatucci.
	\newblock Local behavior of fractional {$p$}-minimizers.
	\newblock {\em Ann. Inst. H. Poincar\'{e} C Anal. Non Lin\'{e}aire},
	33(5):1279--1299, 2016.
	
	\bibitem[DF13]{DeFo13}
	Arnaud Debussche and Nicolas Fournier.
	\newblock Existence of densities for stable-like driven {SDE}'s with
	{H}\"{o}lder continuous coefficients.
	\newblock {\em J. Funct. Anal.}, 264(8):1757--1778, 2013.
	
	\bibitem[DFM22]{DeMi22}
	Cristiana De~Filippis and Giuseppe Mingione.
	\newblock Gradient regularity in mixed local and nonlocal problems.
	\newblock {\em Math. Ann.}, pages 1--68, 2022.
	
	\bibitem[DFZ23]{DFZ23}
	Mengyao Ding, Yuzhou Fang, and Chao Zhang.
	\newblock Local behaviour of the mixed local and nonlocal problems with
	nonstandard growth.
	\newblock {\em arXiv:2304.00778}, 2023.
	
	\bibitem[DK13]{DyKa13}
	Bart{\l}omiej Dyda and Moritz Kassmann.
	\newblock On weighted {P}oincar\'{e} inequalities.
	\newblock {\em Ann. Acad. Sci. Fenn. Math.}, 38(2):721--726, 2013.
	
	\bibitem[DK20]{DyKa15}
	Bart\l{}omiej Dyda and Moritz Kassmann.
	\newblock Regularity estimates for elliptic nonlocal operators.
	\newblock {\em Anal. PDE}, 13(2):317--370, 2020.
	
	\bibitem[DK22]{DyKi22}
	Bart{\l}omiej Dyda and Micha{\l} Kijaczko.
	\newblock On density of compactly supported smooth functions in fractional
	{S}obolev spaces.
	\newblock {\em Ann. Mat. Pura Appl.}, 201(4):1855--1867, 2022.
	
	\bibitem[FJR21]{FPR18}
	Martin Friesen, Peng Jin, and Barbara Rüdiger.
	\newblock Existence of densities for stochastic differential equations driven
	by {L}évy processes with anisotropic jumps.
	\newblock {\em Ann. Inst. H. Poincaré Probab. Statist.}, 57(1):250--271, 2021.
	
	\bibitem[FK13]{FeKa13}
	Matthieu Felsinger and Moritz Kassmann.
	\newblock Local regularity for parabolic nonlocal operators.
	\newblock {\em Comm. Partial Differential Equations}, 38(9):1539--1573, 2013.
	
	\bibitem[FK22]{FoKa22}
	Guy Foghem and Moritz Kassmann.
	\newblock A general framework for nonlocal {N}eumann problems.
	\newblock {\em arXiv:2204.06793}, 2022.
	
	\bibitem[Foo09a]{Foo09b}
	Mohammud Foondun.
	\newblock Harmonic functions for a class of integro-differential operators.
	\newblock {\em Potential Anal.}, 31(1):21--44, 2009.
	
	\bibitem[Foo09b]{Foo09a}
	Mohammud Foondun.
	\newblock Heat kernel estimates and {H}arnack inequalities for some {D}irichlet
	forms with non-local part.
	\newblock {\em Electron. J. Probab.}, 14(11):314--340, 2009.
	
	\bibitem[FSV15]{FSV15}
	Alessio Fiscella, Raffaella Servadei, and Enrico Valdinoci.
	\newblock Density properties for fractional {S}obolev spaces.
	\newblock {\em Ann. Acad. Sci. Fenn. Math}, 40(1):235--253, 2015.
	
	\bibitem[FSZ22]{FSZ22}
	Yuzhou Fang, Bin Shang, and Chao Zhang.
	\newblock Regularity theory for mixed local and nonlocal parabolic p-{L}aplace
	equations.
	\newblock {\em J. Geom. Anal.}, 32(1):22, 2022.
	
	\bibitem[FV17]{FaVa17}
	Alberto Farina and Enrico Valdinoci.
	\newblock Regularity and rigidity theorems for a class of anisotropic nonlocal
	operators.
	\newblock {\em Manuscripta Math.}, 153(1-2):53--70, 2017.
	
	\bibitem[FV19]{FaVa19}
	Alberto Farina and Enrico Valdinoci.
	\newblock Gradient estimates for a class of anisotropic nonlocal operators.
	\newblock {\em NoDEA Nonlinear Differential Equations Appl.}, 26(4):Paper No.
	25, 13, 2019.
	
	\bibitem[GK22]{GaKi22}
	Prashanta Garain and Juha Kinnunen.
	\newblock On the regularity theory for mixed local and nonlocal quasilinear
	elliptic equations.
	\newblock {\em Trans. Amer. Math. Soc.}, 375(08):5393--5423, 2022.
	
	\bibitem[GK23]{GaKi23}
	Prashanta Garain and Juha Kinnunen.
	\newblock Weak {H}arnack inequality for a mixed local and nonlocal parabolic
	equation.
	\newblock {\em J. Differential Equations}, 360:373--406, 2023.
	
	\bibitem[GKK23]{GKK23}
	Prashanta Garain, Wontae Kim, and Juha Kinnunen.
	\newblock On the regularity theory for mixed anisotropic and nonlocal $ p
	$-{L}aplace equations and its applications to singular problems.
	\newblock {\em arXiv:2303.14362}, 2023.
	
	\bibitem[Gof22]{Gof22}
	Alessandro Goffi.
	\newblock A priori {L}ipschitz estimates for nonlinear equations with mixed
	local and nonlocal diffusion via the adjoint-{B}ernstein method.
	\newblock {\em Boll. Unione Mat. Italiana}, pages 1--18, 2022.
	
	\bibitem[Gou20]{Fog20}
	Guy Fabrice~Foghem Gounoue.
	\newblock L2-theory for nonlocal operators on domains.
	\newblock {\em PhD Thesis (Bielefeld University)}, 2020.
	
	\bibitem[Kas03]{Kas03}
	Moritz Kassmann.
	\newblock On regularity for {B}eurling-{D}eny type {D}irichlet forms.
	\newblock {\em Potential Anal.}, 19(1):69--87, 2003.
	
	\bibitem[Kas09]{Kas09}
	Moritz Kassmann.
	\newblock A priori estimates for integro-differential operators with measurable
	kernels.
	\newblock {\em Calc. Var. Partial Differential Equations}, 34(1):1--21, 2009.
	
	\bibitem[KKK22]{KKK19}
	Moritz {Kassmann}, Kyung-Youn {Kim}, and Takashi {Kumagai}.
	\newblock Heat kernel bounds for nonlocal operators with singular kernels.
	\newblock {\em J. Math. Pures Appl.}, 164:1--26, 2022.
	
	\bibitem[KKR22a]{KKR22b}
	Tadeusz Kulczycki, Alexei Kulik, and Micha{\l} Ryznar.
	\newblock On weak solution of {SDE} driven by inhomogeneous singular {L}\'{e}vy
	noise.
	\newblock {\em Trans. Amer. Math. Soc.}, 375(7):4567--4618, 2022.
	
	\bibitem[KKR22b]{KKR22a}
	Tadeusz {Kulczycki}, Oleksii {Kulyk}, and Micha{\l} {Ryznar}.
	\newblock Drift reduction method for {SDE}s driven by inhomogeneous singular
	{L}\'{e}vy noise.
	\newblock {\em arXiv:2208.06595}, 2022.
	
	\bibitem[KPP23]{KPP23}
	Alexei~M. Kulik, Szymon Peszat, and Enrico Priola.
	\newblock Gradient formula for transition semigroup corresponding to stochastic
	equation driven by a system of independent {L}\'{e}vy processes.
	\newblock {\em NoDEA Nonlinear Differential Equations Appl.}, 30(1):Paper No.
	7, 27, 2023.
	
	\bibitem[KR18]{KuRy18}
	Tadeusz Kulczycki and Michal Ryznar.
	\newblock Transition density estimates for diagonal systems of {SDE}s driven by
	cylindrical $\alpha$-stable processes.
	\newblock {\em ALEA, Lat. Am. J. Probab. Math. Stat.}, 15:1335--1375, 2018.
	
	\bibitem[KRS21]{KRS18}
	Tadeusz Kulczycki, Michal Ryznar, and Pawel Sztonyk.
	\newblock Strong {F}eller property for {SDE}s driven by multiplicative
	cylindrical stable noise.
	\newblock {\em Potential Anal.}, 55:75--126, 2021.
	
	\bibitem[KS14]{KaSc14}
	Moritz Kassmann and Russell~W. Schwab.
	\newblock Regularity results for nonlocal parabolic equations.
	\newblock {\em Riv. Math. Univ. Parma (N.S.)}, 5(1):183--212, 2014.
	
	\bibitem[KW22a]{KaWe22a}
	Moritz Kassmann and Marvin Weidner.
	\newblock Nonlocal operators related to nonsymmetric forms {I}: {H}{\"o}lder
	estimates.
	\newblock {\em arXiv:2203.07418}, 2022.
	
	\bibitem[KW22b]{KaWe22b}
	Moritz Kassmann and Marvin Weidner.
	\newblock Nonlocal operators related to nonsymmetric forms {II}: {H}arnack
	inequalities.
	\newblock {\em Anal. PDE (accepted for publication)}, 2022.
	
	\bibitem[LV17]{LuVa17}
	Hannes Luiro and Antti~V. V{\"a}h{\"a}kangas.
	\newblock Beyond local maximal operators.
	\newblock {\em Potential Anal.}, 46:201--226, 2017.
	
	\bibitem[Nak22]{Nak22}
	Kenta Nakamura.
	\newblock Local boundedness of a mixed local--nonlocal doubly nonlinear
	equation.
	\newblock {\em J. Evol. Equat.}, 22(3):75, 2022.
	
	\bibitem[Nak23]{Nak23}
	Kenta Nakamura.
	\newblock Harnack’s estimate for a mixed local--nonlocal doubly nonlinear
	parabolic equation.
	\newblock {\em Calc. Var. Partial Differential Equations}, 62(2):40, 2023.
	
	\bibitem[SC02]{Sal02}
	Laurent Saloff-Coste.
	\newblock {\em Aspects of {S}obolev-type inequalities}, volume 289 of {\em
		London Mathematical Society Lecture Note Series}.
	\newblock Cambridge University Press, Cambridge, 2002.
	
	\bibitem[SZ23]{ShZh23}
	Bin Shang and Chao Zhang.
	\newblock Harnack inequality for mixed local and nonlocal parabolic p-{L}aplace
	equations.
	\newblock {\em The Journal of Geometric Analysis}, 33(4):124, 2023.
	
	\bibitem[Wei22]{Wei22}
	Marvin Weidner.
	\newblock Energy methods for nonsymmetric nonlocal operators.
	\newblock {\em PhD Thesis (Bielefeld University)}, 2022.
	
	\bibitem[Xu13]{Xu13}
	Fangjun Xu.
	\newblock A class of singular symmetric {M}arkov processes.
	\newblock {\em Potential Anal.}, 38(1):207--232, 2013.
	
\end{thebibliography}

\end{document}